\documentclass[a4paper]{article}
\usepackage[dvipsnames]{xcolor}
\usepackage{amsmath}
\usepackage{amssymb}
\usepackage{amsthm}
\usepackage{mathrsfs}
\usepackage{enumerate}
\usepackage{color}
\usepackage{geometry}
\usepackage{bbm}
\usepackage{slashed}
\usepackage{hyperref} 
\usepackage{enumerate,cancel,amssymb,centernot}
 \usepackage{latexsym}
 \usepackage{verbatim}
\usepackage{graphicx}
\usepackage{tikz}

\newtheorem{thm}{Theorem}[section]
\newtheorem{lem}[thm]{Lemma}
\newtheorem{cor}[thm]{Corollary}

\newtheorem{rmk}[thm]{Remark}
\newtheorem{prop}[thm]{Proposition}
\newtheorem{defi}[thm]{Definition}
\numberwithin{equation}{section}

\newcounter{cnstcnt}

\def\mbb#1{\mathbb{#1}}
\def\mcc#1{\mathcal{#1}}
\def\mss#1{\mathscr{#1}}
\def\pari{\partial_{int}}
\def\pare{\partial_{ext}}
\def\eps{\epsilon}
\def\lamn{\Lambda_N}
\def\P{\mathbb{P}}
\def\cA{\mathcal{A}}

\newcommand{\1}{\mathbf{1}}
\newcommand{\cC}{\mathcal{C}}
\newcommand{\cH}{\mathcal{H}}
\newcommand{\cV}{\mathcal{V}}
\newcommand{\cZ}{\mathcal{Z}}
\newcommand{\cB}{\mathcal{B}}
\newcommand{\cE}{\mathcal{E}}

\newcommand{\edge}{\rm{ E}}
\newcommand{\E}{{\mathbb E}}
\def\0{\textbf{0}}
\def\var{\mathsf {var}}
\def\dist{\mathsf{dist}}
\def\cF{\mathcal{F}}
\def\cov{\mathsf {cov}}

\newcommand{\R}{{\mathbb R}}
\def\C{\mathcal{C}}
\newcommand{\Z}{\mathbb{Z}}
\def \fkh#1{\phi_{p,\lamn}^{{#1},\epsilon h}}

\def \fk#1{\phi_{p,\lamn}^{#1,0}}
\def \fkc#1{\phi_{p_c,\lamn}^{#1,0}}
\def \fkco#1{\phi_{p_c,\Omega}^{#1,0}}

\def \fkhco#1{\phi_{p_c,\Omega}^{#1,\eps h}}

\def \fkhlow{\phi_{p,\Omega}^{\w,\eps h}}
\def \fklow{\phi_{p,\Omega}^{\w,0}}
\def \fkhhof{\phi_{p,\Omega}^{\f,\eps h}}
\def \fkhof{\phi_{p,\Omega}^{\f,0}}

\def\f {\mathrm f}
\def\w {\mathrm{w}}
\def\fC {\mathfrak{C}}
\def\sS {\mathscr S}
\def\sT {\mathscr T}

\def\calO {\mathcal O}

\def\cV {\mathcal V}
\def\cN {\mathcal N}
\def\cR {\mathcal R}
\def\cB {\mathcal B}

\def\sL {\mss L}

\def\TV#1#2{\Vert #1-#2 \Vert_{\mathrm {TV}}}

\def\cT{\mathcal{T}}

\DeclareSymbolFont{sfoperators}{OT1}{ptm}{m}{n}
\DeclareSymbolFontAlphabet{\mathsf}{sfoperators}

\makeatletter
\def\operator@font{\mathgroup\symsfoperators}
\makeatother

\title{A phase transition for the two-dimensional random field Ising/FK-Ising model}
\author{Chenxu Hao  \\ Peking University \and Fenglin Huang  \\ Peking University \and Aoteng Xia \\ Peking University}
\begin{document}
\maketitle

\begin{abstract}
    We study the total variation (TV) distance between the laws of the 2D Ising/FK-Ising model in a box of side-length $N$ with and without an i.i.d.\ Gaussian external field with variance $\epsilon^2$.
    Letting the external field strength $\epsilon = \epsilon(N)$ depend on the size of the box, we derive a phase transition for each model depending on the order of $\epsilon(N)$. For the random field Ising model, the critical order for $\epsilon$ is $N^{-1}$. For the random field FK-Ising model, the critical order depends on the temperature regime: for $T>T_c$, $T=T_c$ and $T\in (0, T_c)$ the critical order for $\epsilon$ is, respectively, $N^{-\frac{1}{2}}$, $N^{-\frac{15}{16}}$ and $N^{-1}$. In each case, as $N \to \infty$ the TV distance under consideration converges to $1$ when $\epsilon$ is above the respective critical order and converges to $0$ when below.
\end{abstract}

\section{Introduction}

The Ising model is one of the most significant models in statistical physics exhibiting the phenomenon of phase transition. The impact of minor perturbations on  the phase transition phenomenon stands as an important question in disordered systems, and extensive study has been dedicated to exploring this topic within the context of the Ising model. Specifically, Imry-Ma \cite{IM75} predicts that the {\bf long-range order}---or, equivalently, a positive lower bound for the {\bf boundary influence} (as defined subsequently)---exists for the random field Ising model (RFIM) with weak disorder when the temperature $T$ is small in dimensions three and higher, but not in two dimensions for all temperatures $T$. 
Regarding the two-dimensional behavior, the absence of long-range order was first proved in \cite{AW90}. Some estimations on the decay rate of the boundary influence were first given in \cite{AP19, Cha18}, and it was finally shown in \cite{AHP20,DX21} that the boundary influence decays exponentially, again at all temperatures. For dimensions $d\geq 3$, let $T_c(d)$ denote the critical temperature in $d$ dimensions of the usual Ising model, with no external field. The existence of long-range order was proved at low enough temperature and small enough disorder in \cite{Imbrie85, BK88}. The existence of long-range order at low enough temperature was recently reproved in \cite{DZ21} and then extended to all temperatures below the critical temperature $T_c(d)$ in \cite{DLX24}.
Additionally, a correlation inequality (from \cite{DSS22}) indicates that the long-range order does not exist for any temperatures at or above $T_c(d)$.

Despite the significant progress regarding the Imry-Ma predictions, some relevant properties of the RFIM are still not fully comprehended. One such property is the near-critical behavior of the two-dimensional RFIM. More specifically, examining what happens when the disorder strength $\eps$ varies with the box size $N$ at different rates. In this direction, \cite{DW20}, \cite{DZ21} and \cite{DHX23} study the {\bf correlation length}---roughly speaking, when critical order that the boundary influence ``perceive'' the external field, see Section~\ref{sec: Boundary influence and correlation length} for detailed definitions---of the two-dimensional RFIM for all $T\le T_c.$ Furthermore, in \cite{CGN15,CGN16,CJN20}, they study the Ising model with a uniform external field and show the phase transition of boundary influence with respect to the strength of the external field.

 However, despite its significance, the boundary influence is merely a statistic of the random field Ising measure. This naturally leads us to ponder upon the following question: when does the measure itself ``perceive'' the external field? More formally, how does the total variation distance between the Ising measures with and without external field vary with $\eps$? The primary objective of this paper is to provide a definitive answer to this question.

\subsection{The random field Ising model: a warm-up}

Let $\mu^{\xi,\eps h}_{T,\lamn}$
and $\mu^{\xi,0}_{T,\lamn}$ stand for the two-dimensional Ising measures (see Section \ref{Definition-model-RFIM} for detailed definitions) with and without random external field $\eps h$, where $h=\{h_x:x\in\mbb Z^d\}$ is a family of i.i.d. standard normal variables with law $\mbb P$. Our first result concerns the phase transition of the total variation (denoted by $\Vert\cdot\Vert_{\mathrm{TV}}$) distance between them.

\begin{thm} \label{thm:RFIM}
    For any temperature $T>0$, there exists a positive constant $c=c(T)$ such that for any box size $N$, external field strength $\eps>0$ and boundary condition $\xi$:
    \begin{equation}\label{eq:RFIM-absolute-continuity}
        \mbb P \left(\Vert \mu^{\xi,\eps h}_{T,\lamn}-\mu^{\xi,0}_{T,\lamn}\Vert_{\mathrm{TV}}>c^{-1}\sqrt{\eps N}\right)\leq c^{-1}\exp\left(-c\eps^{-\frac{1}{4}}N^{-\frac{1}{4}}\right).
    \end{equation}
        \begin{equation}\label{eq:RFIM-singularity}
        \mbb P \left(\Vert \mu^{\xi,\eps h}_{T,\lamn}-\mu^{\xi,0}_{T,\lamn}\Vert_{\mathrm{TV}}<1-c^{-1}\sqrt{\eps^{-1}N^{-1}}\right)\leq c^{-1}\exp\left(-c\eps N\right).
    \end{equation}
\end{thm}

 Writing $a(N)\ll b(N)$ if $\lim\limits_{N\rightarrow\infty}\frac{a(N)}{b(N)}=0$ and $a(N)\gg b(N)$ if $\lim\limits_{N\rightarrow\infty}\frac{a(N)}{b(N)}=\infty$, we have the following corollary immediately:

\begin{cor}\label{cor: RFIM} Fix any temperature $T>0,$ then for any boundary condition $\xi$:
    \begin{enumerate}[(i)]
        \item if $\eps=\eps(N)\ll N^{-1}$, we have $\Vert \mu^{\xi,\eps h}_{T,\lamn}-\mu^{\xi,0}_{T,\lamn}\Vert_{\mathrm{TV}}\rightarrow 0$ in probability as $N\to\infty$;
        \item if $\eps=\eps(N)\gg N^{-1}$, we have $\Vert \mu^{\xi,\eps h}_{T,\lamn}-\mu^{\xi,0}_{T,\lamn}\Vert_{\mathrm{TV}}\rightarrow 1$ in probability as $N\to\infty$.
    \end{enumerate}
\end{cor}

In a word, the total variation distance is negligible for $\eps\ll N^{-1},$ and significant for $\eps\gg N^{-1}.$ To heuristically understand this result, one may note that there are $N^2$ Gaussian variables in $\lamn$ thus the fluctuation of the Radon-Nikodym derivative between the measures should be of order $\eps^2N^2.$

\subsection{Main results}

Thanks to the power of the Edward-Sokal coupling, the FK-Ising model frequently plays a central role in the study of the Ising model, even with the presence of a random external field. Our main result studies the behavior of 
the total variation distance between the
two-dimensional FK-Ising model and its random field variation, where we call it the random field FK-Ising model (RFFKIM).  In this setting, we observe a more interesting phenomenon, showing that the behavior differs according to the temperature regime.
More precisely, let $T_c=T_c(2)$ denote the critical temperature of the two-dimensional Ising model. Let $\phi^{\gamma,\eps h}_{p,\lamn}$ and $\phi^{\gamma,0}_{p,\lamn}$ denote the FK-Ising measures with and without external field, where $\gamma$ stands for the boundary condition and $p=1-\exp(-\frac{2}{T})$ is the parameter for the FK-Ising model, see Section \ref{Definition-model-RFFM} for detailed definitions. Throughout this paper, we will keep the convention $$p=1-\exp\left(-\frac{2}{T}\right)$$ thereby treating $p$ and $T$ as two interchangeable representations of a single parameter.  We introduce the following notation for simplicity:

 $$\alpha(T)=\left\{\begin{aligned}
         &1~~~&0<T<T_c\,,\\
    &\frac{15}{16}~~~&T=T_c\,,\\
    &\frac{1}{2}~~~&T>T_c\,.
 \end{aligned}
\right.
$$
\begin{thm}\label{thm-all-temperature}
For any temperature $T>0$ and corresponding parameter $p=1-\exp(-\frac{2}{T}),$
there exists a constant $c=c(T)>0$ such that  for any boundary condition $\gamma$:
\begin{equation}\label{eq:RFFKIM-absolute-continuity}
\P\left(\TV{\fk\gamma}{\fkh\gamma}>c^{-1}\sqrt{\epsilon N^{\alpha(T)}}\right)\le c^{-1}\exp\left(-c\eps^{-\frac{1}{4}}N^{-\frac{\alpha(T)}{4}}\right).
\end{equation}
\begin{equation}\label{eq:RFFKIM-singularity}
\P\left(\TV{\fk\gamma}{\fkh\gamma}<1-c^{-1}\sqrt{\epsilon^{-1} N^{-\alpha(T)}}\right)\le c^{-1}\exp\left(-c\epsilon N^{\alpha(T)}\right).
\end{equation}
\end{thm}

\begin{cor}\label{cor-all-temperature}
    Fix $T\in(0,+\infty)$. Then for any boundary condition $\gamma$, the followings hold:
    \begin{enumerate}[(i)]
         \item if $\eps=\eps(N)\ll N^{-\alpha(T)}$, we have $\TV{\fk\gamma}{\fkh\gamma}\rightarrow 0$ in probability as $N\to\infty$;
        \item if $\eps=\eps(N)\gg N^{-\alpha(T)}$, we have $\TV{\fk\gamma}{\fkh\gamma}\rightarrow 1$ in probability as $N\to\infty$.
    \end{enumerate}
\end{cor}

\begin{rmk}
    We leave out the case $T=0$ in both Theorems~\ref{thm:RFIM} and \ref{thm-all-temperature} for the following reasons. For the former case, when $T=0$ the Ising measure is supported on the ground states which leads to a different phase transition point as $N^{-1}$ (see \cite{DW20} for a related work). For the latter case, the definition of the FK-Ising model with external field does not directly extend to the $T=0$ case. Although in principle it is possible to give a definition using the Edward-Sokal coupling via the RFIM, the resulting object is not a natural extension of the FK-Ising model and is less interesting in our case.
 \end{rmk}

\subsection{Organization of the paper}

In Section~\ref{sec:preliminaries-and-notation}, we first present some notations and precise definitions of the RFIM and the RFFKIM, along with the fundamental properties of these models. Subsequently, in Section~\ref{sec:warm-up}, we will prove Theorem~\ref{thm:RFIM} as a warm-up, since it reveals the core concepts of deriving upper and lower bounds of total variation distance through relatively straightforward calculations. Section~\ref{sec: TV exponential decay to 0} will be dedicated to proving the upper bound of the total variation distance between the FK-Ising model with and without  external field (i.e. \eqref{eq:RFFKIM-absolute-continuity}). The corresponding lower bounds (i.e. \eqref{eq:RFFKIM-singularity}) will be postponed to 
Section~\ref{sec:singularity}, utilizing a coarse-graining framework to accumulate small differences on small boxes. Before that, we demonstrate that the differences in small boxes indeed exist due to the effect of the external field on boxes of an appropriate size in Section~\ref{sec: near-critical behavior}. In Appendix~\ref{sec: LDP at critical}, we prove an LDP result for the (discrete) FK-Ising model at critical temperature. While similar results exist for the continuous scaling limit \cite[Theorem 3]{CJN20appendix}, the discrete case cannot be derived directly from it. 

Throughout our proof, we have used several estimates on the FK-Ising model and Gaussian concentration inequalities, which may be considered intuitive or even elementary for experts in these fields; however, despite our best efforts, we could not locate precise references for these results. For the sake of completeness, we have included proofs of them in the Appendices~\ref{sec: analytic results} and \ref{sec: FK Ising not critical}.

\section{Preliminaries and notation}
\label{sec:preliminaries-and-notation}

In this section, we introduce some basic definitions and notation.

\subsection{Notation}
We use $c,C,c_i,C_i $ to represent positive constants whose actual values may vary from line to line. For a finite set $A$, we denote by $|A|$ the cardinality of $A$.
We use $A^c$ to denote the complement of the set (or event) $A$. If $A$ is an event, we denote its indicator by $\1_A$. 
Throughout the paper, we use $\lfloor a\rfloor$ to denote the floor of a real number $a$, i.e., the largest integer less than or equal to $a$.

Let $\mbb Z^d:=\{u=(u_1,u_2,\cdots,u_d):u_1,u_2,\cdots,u_d\in\mbb Z\}$ denote the $d$-dimensional integer lattice.
For two different vertices $u,v\in\mbb Z^d,$ we say $u,v$ are adjacent (which we denote as $u\sim v$) if their $l_1$-distance is $1$. For a vertex set $V\subset \mbb Z^2,$ let $G=(V,E)$ denote the induced subgraph of $V$ in $\mbb Z^2$ (and sometimes we will slightly abuse the notation that we also use $V$ to denote the induced graph), that is $E=\{e=\{x,y\}:x,y\in V, x\sim y\}.$
Let $diam(G)=\max_{u,v\in {V}}\{||u-v||_{\infty}\}.$
We use the notation $\pari G$ to denote its interior boundary and $\pare G$ to denote its exterior boundary. That is, $\pari G=\{u\in {V}:u\sim v\mbox{~for~some~}v\in {V^c}\}$ and $\pare G=\{v\in { V}^c:v\sim u\mbox{~for~some~}u\in { V}\}$. For any configuration $\omega\in\mathbb{R}^{G}$ and $G'\subset G$ we use the notation $\omega_{|G'}$ to denote its restriction on $G'$, i.e., $\omega_{|G'}(u)=\omega_u$ for any $u\in G'$.
We define a sequence of different vertices $u_1,u_2,\cdots,u_n$ to be a circuit if  $\{u_i,u_{i+1}\}\in E$ for any $1\le i\le n$ where $u_1=u_{n+1}$. Let $\Gamma$ be the region enclosed by a circuit $u_1,u_2,\cdots,u_n$, then we say $u_1,u_2,\cdots,u_n$ is the circuit boundary of $\Gamma$. We also define a circuit to be open if $\{u_i,u_{i+1}\}$ is open for any $1\le i\le n$.

For two (discrete) probability measures $\mu$ and $\nu$ on the same measurable space $(\Xi,\mcc F)$, the total variation distance between them is defined as $\TV\mu\nu=\frac{1}{2}\sum_{\omega\in\Xi}\vert\mu(\omega)-\nu(\omega)\vert.$

% {For the sake of simplicity,} we use the notation $f(x)=\ln(\cosh(x))$ where $\cosh(x)=\frac{e^x+e^{-x}}{2}$ throughout the paper{since this function appears many times.}

For simplicity, we denote $f(x)=\ln(\cosh(x))$ throughout the paper, where $\cosh(x)=\frac{e^x+e^{-x}}{2}$, as this function appears frequently.

\subsection{The random field Ising model (RFIM)}\label{Definition-model-RFIM}

For $\epsilon>0$, we define the random field Ising model (RFIM) Hamiltonian $H_{T,G}^{\xi,\epsilon h}$ on a finite subgraph $G=(V,E)$ of $\Z^d$ with the boundary condition ($\xi\in \{1,0,-1\}^{\Z^d})$, and with the external field $\epsilon h=\{\epsilon h_v:v\in \mathbb{Z}^d\}$ by
\begin{equation}\label{def-Hamiltonian}
H_{T,G}^{\xi,\epsilon h}(\sigma)=-\left(\sum_{{\substack{u,v\in V\\ u\sim v}}}\sigma_u\sigma_v+\sum_{{\substack{u\in V\\ v\notin V\\ u\sim v}}}\sigma_u\xi_v+\sum_{u\in V}\epsilon h_u\sigma_u\right)~~\mbox{for}~~\sigma\in\Sigma,
\end{equation}
where $\Sigma:=\{-1,1\}^{V}$.
For $T>0$, we define $\mu_{T,G}^{\xi,\epsilon h}$ to be a Gibbs measure on $\Sigma$ at temperature $T$ by 

\begin{equation}\label{def-Ising-measure}
\mu_{T,G}^{\xi,\epsilon h}(\sigma)=\dfrac{1}{\cZ_{\mu,T,G}^{\xi,\epsilon h}}e^{-\frac{1}{T}H_{T,G}^{\xi,\epsilon h}(\sigma)}~~\mbox{for}~~\sigma\in\Sigma,
\end{equation}
where $\cZ_{\mu,T,G}^{\xi,\epsilon h}$ is the partition function given by $\cZ_{\mu,T,G}^{\xi,\epsilon h}=\sum_{\sigma\in\Sigma}e^{-\frac{1}{T}H_{T,G}^{\xi,\epsilon h}(\sigma)}.$
Note that $\mu_{T,G}^{\xi,\epsilon h}$ and $\cZ_{\mu,T,G}^{\xi,\epsilon h}$ are random variables depending on the external field $\epsilon h$. We will denote by $\langle\cdot\rangle_{T,G}^{\xi,\epsilon h}$ the expectation taken with respect to $\mu_{T,G}^{\xi,\epsilon h}$.
By taking $\eps=0$ we get the pure Ising model.

\subsection{The random field FK-Ising model (RFFKIM)}\label{Definition-model-RFFM}

In this subsection, we briefly review the FK-Ising model with external field, which was introduced in \cite{ES88}. For a subgraph $G=(V,E)$ of $\mathbb{Z}^d$, we can consider an edge configuration $\omega\in\Xi=\{0,1\}^E$ where $0$ indicates that the edge is {\bf closed} and $1$ indicates the edge is {\bf open}.
Throughout this paper, a boundary condition $\gamma$ is an edge configuration on $G^c$, i.e., $\gamma\in\{0,1\}^{\mathrm{E}(\mbb Z^d)\setminus E},$ where $\mathrm{E}(\mbb Z^d)$ is the edge set of $\mbb Z^d.$ We use $\w$ and $\f$ to denote the wired and free boundary conditions, that is, all the edges in $\gamma$ are open or closed, respectively. 

For any $x \neq y\in V$, we say $x$ is {\bf connected} to $y$ in $\omega$ under boundary condition $\gamma$ if there exists a sequence $x_0=x,x_1,\cdots,x_{n}=y$ such that  $\{x_i,x_{i+1}\}\in \mathrm{E}(\mbb Z^d)$ and $\{x_i,x_{i+1}\}$ is open for all $0\leq i\leq n-1$; this sequence is called an open path in $\omega$. So given any $\omega\in\Xi,$ the graph $G$ is divided into a disjoint union of connected components, we also call them open clusters. 
Let $\kappa(\omega^\gamma)$ be the number of open clusters in $\omega$ under boundary condition $\gamma$. We use the notation $\fC(\omega^\gamma)=\{\cC_1,\cC_2,\cdots, \cC_{\kappa(\omega^\gamma)}\}$ to denote the collection of all clusters. For notation clarity, we drop the superscript $\gamma$ in $\omega^\gamma$ when $\gamma$ is clear in the context.

Then the FK-Ising model on $G$ with parameter $p\in(0,1)$ (and corresponding temperature $T\in(0,\infty)$ with $p= 1-\exp(-\frac{2}{T})$) is a probability measure on  $\Xi$ given by
\begin{equation}\label{eq-def-FK-external-field}
\phi^{\gamma,\eps h}_{p,G}(\omega)=\frac{1}{\cZ_{\phi,p,G}^{\gamma,\eps h} } \prod_{e\in E}\hspace{0.5em}p^{\omega(e)}(1-p)^{1-\omega(e)}\prod_{j=1}^{\kappa(\omega^\gamma)}\hspace{0.5em}2 \cosh\left(\frac{\eps h_{\cC_j}}{T}\right),
\end{equation}
where $$\cZ_{\phi,p,G}^{\gamma,\eps h}=\sum_{\omega\in\Xi}\prod_{e\in E}\hspace{0.5em}p^{\omega(e)}(1-p)^{1-\omega(e)}\prod_{j=1}^{\kappa(\omega^\gamma)}\hspace{0.5em}2 \cosh\left(\frac{\eps h_{\cC_j}}{T}\right)$$ is the normalizing constant (partition function) of $\phi^{\gamma,\eps h}_{p,G}$ and $h_{\mcc C_j}=\sum_{x\in\mcc C_j}h_x.$
Similarly, we obtain the FK-Ising model without external field and the corresponding partition function by taking $\eps=0.$

For a subgraph $G=(V,E)$ of $\mathbb{Z}^d$, an edge configuration $\omega\in\Xi=\{0,1\}^E$ and a subset $E'\subset E$, recall that  $\omega_{|{E'}}$ is the restriction of $\omega$ on $E'$. In addition, if there exists $V'\subset V$ such that $E'=\{(x,y)\in E: x,y\in V'\}$, then we also write $\omega_{|V'}$ to denote the restriction of $\omega$ on $E'$.

Throughout this paper, we focus on the case where $G$ is a subgraph of $\mathbb{Z}^2$, particularly when $G=\Lambda_N:=[-N, N]^2\cap\mathbb{Z}^2$, which represents a box of side length $2N$ centered at the origin. Note that the boundary condition $\gamma$ influences the measure $\phi^{\gamma,\eps h}_{p,\lamn}$ only through the extra connections using edges outside $\lamn$. Thus, we can also choose the boundary condition as a partition on the interior boundary $\pari\lamn.$ 
The external field $\eps h$ will always be chosen as $\eps h=\{\eps h_x: x\in G\}$ where $h_x$'s are i.i.d. standard Gaussian variables which law will be denoted as $\mbb P$ and $\eps$ stands for the strength of the external field.  In this setup,  the model will be called the {\bf two-dimensional random field FK-Ising model}.

\subsection{The Edward-Sokal coupling}

The Edwards-Sokal coupling due to the paper \cite{ES88}, enables us to express the correlation functions of the Ising model in terms of the FK-Ising model.
The Edward-Sokal coupling between the FK-Ising model and the Ising model with parameter $p=1-\exp(-\frac{2}{T})$ can be described as follows: for $G = (V,E)$, consider a probability measure $\pi$ on $\Sigma\otimes\Xi$ defined by
$$
\pi(\sigma,\omega)=\frac{1}{\cZ_{\pi}}\prod_{e=\{x,y\}\in E}\Big[(1-p)\delta(\omega(e),0)+p\delta(\omega(e),1)\delta(\sigma_x,\sigma_y)\Big],
$$
where $\delta(\cdot,\cdot)$ is a $\{0,1\}$-valued function such that $\delta(a,b)=1$ if and only if $a=b$ and $\cZ_{\pi}$ is the normalizing constant. Then we have the following properties:
\begin{enumerate}[(i)]
\item {\bf Marginal on $\Sigma$}: for a fixed $\sigma\in\Sigma$, summing over all $\omega\in\Xi$ we have
\begin{eqnarray*}
&\pi(\sigma)=\displaystyle\sum_{\omega\in\Xi}\pi(\sigma,\omega)=\dfrac{1}{\cZ_{\pi}}\prod_{e=\{x,y\}\in E}\Big[(1-p)+p\delta(\sigma_x,\sigma_y)\Big]
\propto(1-p)^{\sum_{\{x,y\}\in E}\1_{\{\sigma_x\neq\sigma_y\}}}.
\end{eqnarray*}
This coincides with the Ising measure $\mu_{T,G}^{0,0}$ with $0$ boundary condition and $1-p=\exp(-\frac{2}{T})$.

\item {\bf Marginal on $\Xi$}: for a fixed $\omega\in\Xi$, summing over all $\sigma\in\Sigma$ gives
\begin{eqnarray*}
\pi(\omega)=\sum_{\sigma\in\Sigma}\pi(\sigma,\omega)=\dfrac{1}{\cZ_{\pi}}\prod_{\omega(e)=0}(1-p)\sum_{\sigma\in\Sigma}\prod_{\omega(e)=1}p\delta(\sigma_x,\sigma_y).
\end{eqnarray*}
In order for the preceding product to be non-zero, the two spins on the endpoints of every open edge must agree and thus the spins on every open cluster are the same. As a result, there are $2^{\kappa(\omega)}$ different choices for $\sigma$, thus we have
\begin{eqnarray*}
\pi(\omega)=\dfrac{1}{\cZ_{\pi}}\prod_{e\in E}\Big[p^{\omega(e)}(1-p)^{1-\omega(e)}\Big]\times 2^{\kappa(\omega)},
\end{eqnarray*}
which corresponds to the measure $\phi^{\f,0}_{p,G}$ of the FK-Ising model on $\{0,1\}^E$.

\item {\bf Conditioned on a spin configuration}: for a fixed $\sigma\in\Sigma$, let $E_1=\{\{x,y\}\in E:\sigma_x=\sigma_y\}$ and $E_2=E\setminus E_1$. Dividing $\pi(\sigma,\omega)$ by $\pi(\sigma)$ yields that
\begin{eqnarray*}
\pi_{|\sigma}(\omega)\propto\prod_{e\in E_2}\delta(\omega(e),0)\prod_{e\in E_1}\Big[(1-p)\delta(\omega(e),0)+p\delta(\omega(e),1)\Big].
\end{eqnarray*}
Thus, under $\pi_{|\sigma}$, the edges in $E_2$ must be closed and in addition, edges in $E_1$ are open with probability $p$ and closed with probability $1-p$ independently. In other words, the conditional measure for $\omega$ can be viewed as a Bernoulli bond percolation within each spin cluster.

\item {\bf Conditioned on an edge configuration}: for a fixed $\omega\in\Xi$ let $E_3=\{e\in E: \omega(e)=1\}$ and $E_4=E\setminus E_3$. Then we have
\begin{eqnarray*}
\pi_{|\omega}(\sigma)\propto\prod_{e\in\{x,y\}\in E_3}p\delta(\sigma_x,\sigma_y).
\end{eqnarray*}
Thus, the conditional measure $\pi_{|\omega}$ is uniform among spin configurations where each open cluster in $\omega$ receives the same spin.
\end{enumerate}

Furthermore, the above coupling can be extended to the case with external field $\{\epsilon h_x:x\in V\}$. To this end, we define
\begin{eqnarray*}
\pi_h(\sigma,\omega)=\frac{1}{\cZ_{\pi_h}}\prod_{e=\{x,y\}\in E}\Big[(1-p)\delta(\omega(e),0)+p\delta(\omega(e),1)\delta(\sigma_x,\sigma_y)\Big]\cdot\exp\left(-\frac1T\sum_{x\in V}\epsilon h_x\sigma_x\right),
\end{eqnarray*}
where $\cZ_{\pi_h}$ is the normalizing constant. By a straightforward computation, we see
\begin{eqnarray*}
\cZ_{\pi_h}=\sum_{\sigma\in\Sigma}\sum_{\omega\in\Xi}\pi_h(\sigma,\omega)=\sum_{\sigma\in\Sigma}\exp\left[\frac1T\left(\sum_{\{x,y\}\in E}\sigma_x\sigma_y+\sum_{x\in V}\epsilon h_x\sigma_x\right)\right]\cdot \exp\left(-\frac1T|E|\right),
\end{eqnarray*}
implying that the normalizing constant $\cZ_{\pi_h}$ and the partition function $\cZ_{\mu_h}$ for the Ising measure $\mu_{T,G}^{\xi,\epsilon h}$ are
equal up to a factor of $\exp(-\frac1T|E|)$, which in particular does not depend on the external field $\epsilon h$. One can verify that the aforementioned properties (in the case without external field) can be extended as follows:

\begin{enumerate}[(i)]
\item For a fixed $\sigma\in\Sigma$, the additional term $\exp(\frac1T\sum_{x\in V}\epsilon h_x\sigma_x)$ is a constant in $\omega$. As a result, the marginal distribution of $\pi_h$ on $\Sigma$ is exactly the Ising measure with external field $\epsilon h$. In addition, the conditional distribution for percolation configuration given a fixed $\sigma$ is also equivalent to a Bernoulli bond percolation within each spin cluster.
\item For a fixed $\omega\in\Xi$, in order for $\pi_h(\sigma,\omega)$ to be non-zero, each open cluster must have the same spin. Thus, writing $h_A=\sum_{x\in A}h_x$ for any subset $A\subset V$, we have
\begin{eqnarray*}
\sum_{\sigma\in\Sigma}\pi_h(\sigma,\omega)=\dfrac{1}{\cZ_{\pi_h}}\prod_{e\in E}p^{\omega(e)}(1-p)^{1-\omega(e)}\prod_{j=1}^{\kappa(\omega)}2\cosh (h_{\C_j} /T).
\end{eqnarray*}
Therefore, conditioned on a fixed $\omega\in\Xi$, each open cluster $\C_j$ must have the same spin and this spin is plus with probability $\frac{\exp(h_{\C_j}/T)}{2\cosh(h_{\C_j}/T)}$.
\end{enumerate}
With the Edward-Sokal coupling in mind, we may sometimes abuse the notation $\langle\cdot\rangle$ to be the expectation operator for both the Ising and FK-Ising measures.

\subsection{Basic properties of the model(s)}

Here we state some of the standard and well-known properties that will be used repeatedly.

 \begin{enumerate}[(i)]
        \item {\bf FKG inequality.} This was introduced in \cite{FKG71} and named after the three authors.
        
        If $A, B$ are both increasing (or decreasing) events, then we have
$$P(A\cap B)\ge P(A)\times P(B).$$
        
        \item {\bf Comparison of boundary conditions.} If two boundary conditions $\xi_1\le\xi_2$, then for any increasing event $A$ we have
        $$P^{\xi_1}(A)\le P^{\xi_2}(A).$$
        Here $\xi_1\le\xi_2$ stands for a partial order in the set of boundary conditions: for spin configurations, we say $\xi_1\le\xi_2$ if every plus spin in $\xi_1$ is also plus in $\xi_2$; for edge configurations, we say $\xi_1\le\xi_2$ if every open edge in $\xi_1$ is also open in $\xi_2$.

        \item {\bf Domain Markov property.} For two domains $\Gamma_1\subset\Gamma_2$, given the configuration $\xi$ on $\Gamma_2/\Gamma_1$, the influence on the measure in $\Gamma_1$ behaves like a boundary condition:
        $$P_{\Gamma_2}(\,\cdot\,|\xi)=P_{\Gamma_1}^{\xi_{|\partial \Gamma_1}}(\cdot),$$
        where on the right-hand side $\xi_{|\partial \Gamma_1}$ is the restriction on $\pare\Gamma_1$ if $P$ is an Ising measure and $\xi_{|\partial \Gamma_1}$ is a partition on $\pari\Gamma_{ 1}$ if $P$ is an FK-Ising measure.
\item {\bf Dual path.} We review the dual theory for the FK-Ising model initiated in \cite{KW42}, and our presentation follows that in \cite{BD12}. Define $(\mathbb{Z}^2)^{\diamond}=\mathbb{Z}^2+(\frac12,\frac12)$ to be the dual of $\mathbb{Z}^2$, which can be viewed as a translation of $\mathbb{Z}^2$ in $\mathbb{R}^2$. We see that every vertex in $(\mathbb{Z}^2)^{\diamond}$ is the center of a unit square in $\mathbb{Z}^2$ and vice versa. In other words, every edge $e\in\mathbb{Z}^2$ intersects with a unique edge $e^{\diamond}\in(\mathbb{Z}^2)^{\diamond}$. For a configuration $\omega\in\{0,1\}^{E(\mathbb{Z}^2)}$, we define its dual configuration $\omega^{\diamond}$ by setting
$$\omega^{\diamond}(e^{\diamond})=1-\omega(e),~~~\mbox{for~all~}e\in E(\mathbb{Z}^2).$$
Importantly, the measure of $\omega^{\diamond}$ is also an FK-Ising measure with a dual boundary condition and dual parameter $p^{\diamond}$ (see \cite{BD12} for more details), whereby \cite{On44}
$$p_c^{\diamond}=p_c=\dfrac{\sqrt{2}}{1+\sqrt{2}}.$$
Now for a rectangle $R=[a,b]\times[c,d]$, let ${\mcc H}^{\diamond}(R)$ denote the event that there exists an open dual path crossing R horizontally. This is a slight abuse of notation since the dual path lives in $[a-\frac12,b+\frac12]\times[c+\frac12,d-\frac12]\subset(\mathbb{Z}^2)^{\diamond}$
. Similarly, we define the event of vertical dual crossing $\mathcal{V^{\diamond}(R)}$,
and we also define the event $\mathcal{H(R)}$ and $\mathcal{V(R)}$ which correspond to crossings by $\omega$-paths. By duality, we have $\mathcal{V(R)}$ happens if and only if $\mathcal{H^{\diamond}(R)}$ fails, and $\mathcal{H(R)}$ happens if and only if $\mathcal{V^{\diamond}(R)}$ fails.

    \end{enumerate}
    In (i), (ii), and (iii), $P$ may stand for the Ising measure with or without the external field and with or without the boundary condition and may also stand for the FK-Ising measure without the external field. However, the FK-Ising measure with the external field only satisfies (i) and (ii). We will write FKG, CBC, and DMP to represent the properties (i), (ii), and (iii) in the following for convenience.

\subsection{Boundary influence and correlation length}\label{sec: Boundary influence and correlation length}

One of the key quantities of interest for the RFIM is the {\bf boundary influence} defined as follows:
$$m(T,N,\eps ):=\mbb E(\tilde{m}(T,N,\eps h))=\frac{1}{2}\mbb E\left(\langle \sigma_o \rangle^{+,\eps h}_{T,\lamn}-\langle \sigma_o \rangle^{-,\eps h}_{T,\lamn}\right),$$
 where $\sigma\in\{-1,1\}^{\lamn}$ is the spin configuration, $\eps h=\{\eps h_v:v\in\mbb Z^d\}$ is the external field with strength $\eps>0$.
 And we will say the {\bf long-range order} exists for the RFIM if the boundary influence has a positive lower bound as the box size $N$ goes to infinity.
 The absence of long-range order in two dimensions for any temperature $T$ and any disorder strength $\eps>0$ raises an intriguing question: can we still observe a phase transition phenomenon for the 2D RFIM?  The answer is positive, the behavior of boundary influence alters when $\eps$ decays with $N$ in different rates. 
 
 More precisely, Define the {\bf correlation length of two-dimensional RFIM} as follows:
$$ {\psi_\star(T, \epsilon)} = \min\{N: m(T, N, \eps) \leq m(T, N, 0)/2\}\,.$$ By a sequence of works \cite{DW20,DZ21,DHX23}, they prove that the correlation length has order $c e^{c\epsilon^{-4/3}}$ at low temperature $T<T_c$ and has order $c \epsilon^{-8/7}$ at critical temperature $T=T_c$.
Furthermore, \cite{DHX23} proves when $T=T_c(2)$:
 \begin{itemize}
    \item For $\eps\ll N^{-\frac{7}{8}},$ we have $m(T,N,\eps )=(1+o(1))m(T,N,0).$
    \item For $\eps\gg N^{-\frac{7}{8}}$, the boundary influence $m(T,N,\eps )$ decays exponentially in $N$, moreover, the decay rate is of order $\eps^{\frac{8}{7}}.$
\end{itemize}

In summary, at the critical temperature, the boundary influence can ``perceive'' the random external field if and only if $\eps\gg N^{-7/8}$. In contrast to this threshold, our result Theorem~\ref{thm:RFIM} implies that for any statistic of the Ising measure at any temperature, as long as $\eps\ll N^{-1}$, it will behave similarly in both the cases with and without random field. Furthermore, Theorem~\ref{thm-all-temperature} implies that for any statistic of the FK-Ising measure at the critical-temperature, as long as $\eps\ll N^{-\frac{15}{16}}$, it will behave similarly in both the cases with and without random field.

\section{Proof of Theorem~\ref{thm:RFIM}}
\label{sec:warm-up}

We begin by proving Theorem~\ref{thm:RFIM}. While its proof is significantly less complex than those of Theorem~\ref{thm-all-temperature}, it introduces key concepts that we will repeatedly utilize throughout our discussion.

We will refer to inequalities such as \eqref{eq:RFIM-absolute-continuity} as ``absolute continuity" and \eqref{eq:RFIM-singularity} as ``singularity"  for the sake of simplicity, although this terminology is not mathematically rigorous since all the measures discussed here are discrete. Notice that the total variation distance is always between 0 and 1, thus we can always assume $\eps N$ is small enough while proving \eqref{eq:RFIM-absolute-continuity} and large enough while proving \eqref{eq:RFIM-singularity}.

\subsection{Obtaining absolute continuity from decoupling}\label{sec: warm up AC}

It is easy to calculate that 
\begin{equation}
    \frac{\mu^{\xi,\eps h}_{T,\lamn}(\sigma)}{\mu^{\xi,0}_{T,\lamn}(\sigma)}=\frac{\mcc Z^{\xi,0 }_{\mu,T,\lamn}}{\mcc Z^{\xi,\eps h}_{\mu,T,\lamn}}\exp\left(\eps\sum_{x\in\lamn}\sigma_x h_x\right).
\end{equation}
For any $\sigma$ fixed, 
we have that $\eps\sum_{x\in\lamn}\sigma_x h_x$ is a Gaussian variable with variance $\eps^2N^2$, therefore it concentrates well when $\eps N$ is sufficiently small.
But we have to deal with all possible $\sigma$ for a fixed external field $\eps h.$ To address this, we have the following decoupling lemma:

\begin{lem}\label{lem:decoupling}
    Let $(\Omega,\mcc F)$ be a measurable space and $\nu$ is a probability measure on it. Let $\eta$ be a random field with law $\mbb Q$ and $\nu^{\eta}$ is a random probability measure on $(\Omega,\mcc F)$. For any $\omega\in\Xi,$ define $G(\eta,\omega)=\frac{\nu^\eta(\omega)}{\nu(\omega)},$
    then we have for any constants $0<a<b<1$:
    $$\mbb Q(\Vert \nu^{\eta}-\nu \Vert_\mathrm{TV}\geq b)\leq \frac{\mbb Q\otimes\nu\left(G(\eta,\omega)<1-a\right)}{b-a}.$$
\end{lem}

\begin{proof}
    For any fixed external field $\eta$, we can first rewrite the total variation distance as follows:
    \begin{align*}
        \Vert \nu^{\eta}-\nu \Vert_\mathrm{TV}&=\sum_{\omega\in\Xi}\left(\nu(\omega)-\nu^\eta(\omega)\right)_+=\sum_{\omega\in\Xi}\nu(\omega)
        \left(1-G(\eta,\omega)\right)_+\\
    &=\sum_{\omega:G(\eta,\omega)\geq 1-a}\nu(\omega)
        \left(1-G(\eta,\omega)\right)_++\sum_{\omega:G(\eta,\omega)< 1-a}\nu(\omega)
        \left(1-G(\eta,\omega)\right)_+\\&\leq a+\nu\left(G(\eta,\omega)<1-a\right),
    \end{align*}
       where $(x)_+=\max\{x,0\}$.
    Thus, we have 
    \begin{equation*}
        \mbb Q(\Vert \nu^{\eta}-\nu \Vert_\mathrm{TV}\geq b)\leq \mbb Q(\nu(G(\eta,\omega)<1-a)\geq b-a)\leq \frac{\mbb Q\otimes\nu\left(G(\eta,\omega)<1-a\right)}{b-a}.\qedhere
    \end{equation*}
\end{proof}

In a word, by losing acceptable accuracy, we can focus on the product measure $\mbb Q\otimes \nu$ instead of the random measure $\nu^\eta.$ This is sufficient to prove the absolute continuity.

\begin{proof}[Proof of \eqref{eq:RFIM-absolute-continuity}] Taking $\nu=\mu^{\xi,0}_{T,\lamn}$ and $\nu^{\eta}=\mu^{\xi,\eps h}_{T,\lamn}$ in Lemma~\ref{lem:decoupling}, all we need is to upper-bound 
$$G(h,\sigma)=\frac{\mu^{\xi,\eps h}_{T,\lamn}(\sigma)}{\mu^{\xi,0}_{T,\lamn}(\sigma)}=\frac{\mcc Z^{\xi,0 }_{\mu,T,\lamn}}{\mcc Z^{\xi,\eps h}_{\mu,T,\lamn}}\exp\left(\eps\sum_{x\in\lamn}\sigma_x h_x\right).$$
Write $$Z(h)=\frac{\mcc Z^{\xi,0 }_{\mu,T,\lamn}}{\mcc Z^{\xi,\eps h}_{\mu,T,\lamn}}\text{ and }E(h,\sigma)=\exp\left(\eps\sum_{x\in\lamn}\sigma_x h_x\right)$$ for short.
Then by chaos expansion technique (\cite[Theorem 3.14]{FSZ16}) for Ising (see Section~\ref{sec:concentration-of-Z} for a detailed proof), we have
  $$\mbb P\left(|Z(h)-1|\leq \sqrt{\eps N}\right)\geq 1-c_1\exp\left(-c_1\sqrt{\eps^{-1}N^{-1}}\right).$$

 Meanwhile, for any fixed $\sigma,$ we have 
 $\eps \underset{_{x\in\lamn}}{\sum}\sigma_xh_x$ is a Gaussian variable with variance $\eps^2N^2.$
 Therefore, $$\mbb P\left(E(h,\sigma)\leq 1-c_2\sqrt{\eps N}\right)\leq c_2\exp\left(-c_2^{-1}\eps^{-1} N^{-1}\right).$$
 Thus, we have 
 $$\mbb P\otimes\mu^{{\xi ,0}}_{{T,\lamn}}\left(G(h,\sigma)\leq 1-c_3\sqrt{\eps N}\right)\leq c_3\exp\left(-c_3\sqrt{\eps^{-1}N^{-1}}\right),$$
 by choosing $a=c_3\sqrt{\eps N}< b=c_3(\eps N)^{1/4}$ in Lemma~\ref{lem:decoupling} and recalling $\eps N$ is small enough, the proof is completed.
\end{proof}

\subsection{Obtaining singularity from independency}\label{sec:RFIM-singularity}

The idea behind proving singularity for the RFIM  primarily involves leveraging the independence obtained from the DMP. First, we have the following natural result for the total variation distances between product measures.

\begin{defi}
    Let $\{\mu_i:i=1,2,\cdots,n\}$ be a sequence of distributions on measurable spaces $\{(\cE_i,\mcc F_i):i=1,2,\cdots,n\}.$
    Then for any event $\cA_i\in\mcc F_i$, we define $\hat\cA_i=\cE_1\otimes\cdots\otimes\cE_{i-1}\otimes\cA_i\otimes\cE_{i+1}\otimes\cdots\otimes\cE_n$ to be a cylinder set.   
\end{defi}
\begin{lem}\label{lem: product measure total variation distance lower bound}
    Let $\mu_1,\mu_2,\cdots,\mu_n$ and $\nu_1,\nu_2,\cdots,\nu_n$ be two sequences of distributions with $\TV{\mu_i}{\nu_i}\ge a>0$ and $\mu=\otimes_{i=1}^n\mu_i,~ \nu=\otimes_{i=1}^n\nu_i$ be their product measure. Then there exists an absolute constant $c_1>0$ such that $$\TV{\mu}{\nu}\ge 1-c_1\exp\left(-{c_1}^{-1}a^2n\right).$$
\end{lem}
\begin{proof}
    Since $\TV{\mu_i}{\nu_i}\ge a$, there exists a event $\cA_i\in\mcc F_i$ such that $\mu_i(\cA_i)-\nu_i(\cA_i)\ge a$. Since $\cA_i(i=1,2,\cdots,n)$ are independent under $\mu$,  by Lemma~\ref{lem: concentration inequality with finite exponential moment}, we obtain that \begin{align}
    \mu\left(\Big|\sum_{i=1}^n\1_{\hat\cA_i}-\mu(\hat\cA_i)\Big|{\geq}\frac{na}{2}\right)&\le 2\exp\left(-C_1a^2n\right),\label{eq:hoeffding 1}\\ \nu\left(\Big|\sum_{i=1}^n\1_{\hat\cA_i}-\nu(\hat\cA_i)\Big|{\geq}\frac{na}{2}\right)&\le 2\exp\left(-C_1a^2n\right).\label{eq:hoeffding 2}
    \end{align}
    Since $\mu(\hat \cA_i)-\nu(\hat \cA_i)=\mu_i(\cA_i)-\nu_i(\cA_i)\ge a$, we get that $$\left\{\Big|\sum_{i=1}^n\1_{\hat\cA_i}-\mu(\hat\cA_i)\Big|{<}\frac{na}{2}\right\}\cap\left\{\Big|\sum_{i=1}^n\1_{\hat\cA_i}-\nu(\hat\cA_i)\Big|{<}\frac{na}{2}\right\}=\emptyset.$$ Combined with \eqref{eq:hoeffding 1} and \eqref{eq:hoeffding 2}, it yields that $\TV{\mu}{\nu}\ge 1-2\exp(-C_1a^2n)$.
\end{proof}
\begin{cor}\label{cor: product measure total variation distance lower bound}
     Let $\mu_1,\mu_2,\cdots,\mu_n$ and $\nu_1,\nu_2,\cdots,\nu_n$ be two sequences of independent random distributions with joint law $\mbb P$. Assume $\E\TV{\mu_i}{\nu_i}\ge a>0$ for all $1\leq i\leq n$ where $\mbb E$ is the expectation w.r.t. $\mbb P$, and $\mu=\otimes_{i=1}^n\mu_i,~ \nu=\otimes_{i=1}^n\nu_i$ are product measures. Then there exists an absolute constant $c_1>0$ such that $$\P\Big(\TV{\mu}{\nu}\ge 1-c_1\exp(-c_1^{-1}{a^2}n)\Big)\ge 1-c_1\exp\left({-}c_1^{-1}a^2n\right).$$
\end{cor}
\begin{proof}
    Let $I\subset\{1,2,\cdots,n\}$ denote the random set of indices such that $\TV{\mu_i}{\nu_i}\ge \frac{a}{2}$.
    Since $\E\TV{\mu_i}{\nu_i}\ge a$, we get that $\P(\TV{\mu_i}{\nu_i}\ge \frac{a}{2})\ge \frac{a}{2}$. Combined with the fact that $\{\mu_i\}_{1\leq i\leq n}$ and $\{\nu_i\}_{1\leq i\leq n}$ are two independent sequences of distributions, it yields by Lemma~\ref{lem: concentration inequality with finite exponential moment} that \begin{equation}
        \P\left(|I|\ge \frac{an}{4}\right)\ge 1-C_1\exp(-C_1^{-1}a^2n).
    \end{equation} The rest of the proof is the same as that of Lemma~\ref{lem: product measure total variation distance lower bound}.
\end{proof}

Now we are ready to prove the singularity result.

\begin{proof}[Proof of \eqref{eq:RFIM-singularity}]
    Write $$U_N=(2\mbb Z)^2\cap \lamn,\quad V_N=\lamn\setminus U_N; \text{ and } \Sigma_U=\{-1,1\}^{U_N}, \Sigma_V=\{-1,1\}^{V_N}.$$ Thus, for any $x\in U_N,$ all the neighboring points of $x$ are contained in $V_N.$
    For any $\sigma_U\in \Sigma_U$ and $\sigma_V\in\Sigma_V$, we will write $\sigma=\sigma_U\oplus\sigma_V$ for $\sigma\in\Sigma=\{-1,1\}^{\lamn}$ if the restriction of $\sigma$ on $U_N$ and $V_N$ coincides with $\sigma_U$ and $\sigma_V$ respectively.
    Then we can decompose the total variation distance as follows:
    \begin{align}
        &1-\Vert \mu^{\xi,\eps h}_{T,\lamn}-\mu^{\xi,0}_{T,\lamn}\Vert_{\mathrm{TV}}=\sum_{\sigma\in\Sigma}\min\left\{\mu^{\xi,\eps h}_{T,\lamn}(\sigma),\mu^{\xi,0}_{T,\lamn}(\sigma)\right\}\nonumber \\=~&\sum_{\sigma_V\in\Sigma_V}\sum_{\sigma_U\in\Sigma_U}\min\left\{\mu^{\xi,\eps h}_{T,\lamn}(\sigma_U\oplus\sigma_V),\mu^{\xi,0}_{T,\lamn}(\sigma_U\oplus\sigma_V)\right\}.\label{eq: ising singularity decouple}
    \end{align}
    Note that we can first couple the configuration on $V_N$ and throw the cases that the coupling failed, thus we get \begin{align}
        &\min\left\{\mu^{\xi,\eps h}_{T,\lamn}(\sigma_U\oplus\sigma_V),\mu^{\xi,0}_{T,\lamn}(\sigma_U\oplus\sigma_V)\right\}\nonumber\\ \le~& \Big(\mu^{\xi,0}_{T,\lamn}(\sigma_{\mid V_N}=\sigma_V)+\mu^{\xi,\eps h}_{T,\lamn}(\sigma_{\mid V_N}=\sigma_V)\Big)\min\left\{\mu_{T,U_N}^{\sigma_{V},\eps h}(\sigma_U),\mu_{T,U_N}^{\sigma_{V},0}(\sigma_U)\right\}.\nonumber
    \end{align} Combined with \eqref{eq: ising singularity decouple}, it yields that\begin{align}
        &1-\Vert \mu^{\xi,\eps h}_{T,\lamn}-\mu^{\xi,0}_{T,\lamn}\Vert_{\mathrm{TV}}\nonumber \\\leq~& \sum_{\sigma_V\in\Sigma_V}\Big(\mu^{\xi,0}_{T,\lamn}(\sigma_{\mid V_N}=\sigma_V)+\mu^{\xi,\eps h}_{T,\lamn}(\sigma_{\mid V_N}=\sigma_V)\Big)\sum_{\sigma_U\in\Sigma_U}\min\left\{\mu_{T,U_N}^{\sigma_{V},\eps h}(\sigma_U),\mu_{T,U_N}^{\sigma_{V},0}(\sigma_U)\right\}\nonumber\\
        \leq~& 2\max_{\sigma_V\in\Sigma_V}\left(1-\Vert\mu_{T,U_N}^{\sigma_{V},\eps h}-\mu_{T,U_N}^{\sigma_{V},0}\Vert_{\mathrm{TV}}\right).\label{eq: ising singularity final}
    \end{align}
    By DMP, the distribution of points in $U_N$ are independent of each other. Let $V_x$ denote the set of neighbours of $x$, and write $\sigma_{V_x}=(\sigma_V)_{\mid V_x}$
    then we have
    $$\mu_{T,U_N}^{\sigma_{V},\eps h}=\underset{x\in U_N}{\otimes}\mu_{T,\{x\}}^{\sigma_{V_x},\eps h_x}~~\mbox{and}~~\mu_{T,U_N}^{\sigma_{V},0}=\underset{x\in U_N}{\otimes}\mu_{T,\{x\}}^{\sigma_{V_x},0}.$$
    Note that  \begin{align*}
    \Big|\mu_{T,\{x\}}^{\sigma_{V_x},\eps h_x}(\sigma_x=1)-\mu_{T,\{x\}}^{\sigma_{V_x},0}(\sigma_x=1)\Big|
    =\Big|g\left(\frac{N_x}{T}\right)-g\left(\frac{N_x+\eps h_x}{T}\right)\Big|
    =\Big|\frac{\eps h_x}{T}g'\left(\frac{N_x}{T}+\theta\right)\Big|,
    \end{align*}
    where we write $N_x=\sum_{y\sim x}\sigma_{V_x}(y)$,  $g(t)=\frac{e^t}{e^t+e^{-t}},$ and $\theta$ is a constant between 0 and $\frac{\eps h_x}{T}$.
    Since
    $|g'(t)|$ decays in $|t|$ and $|N_x|\leq 4,$ we have 
    $$\Vert\mu_{T,\{x\}}^{\sigma_{V_x},\eps h_x}-\mu_{T,\{x\}}^{\sigma_{V_x},0}\Vert_\mathrm{TV}\geq \frac{\eps |h_x|}{4T}e^{-4-\eps |h_x|}.$$

Combined with Corollary~\ref{cor: product measure total variation distance lower bound} and the fact that $|h_x|\in[1,2]$ with positive probability, it yields the desired result. 
\end{proof}

\section{Absolute continuity in weak-disorder regime}\label{sec: TV exponential decay to 0}
From now on, we focus on the FK-Ising model and the random field FK-Ising model.
In this section, we will prove the absolute continuity results for all temperature regimes. The proof is similar to that in Section \ref{sec: warm up AC}, but some new ideas are required. Since many calculations in high-temperature, low-temperature, and critical-temperature regimes are similar, we introduce the following notation for simplicity:
$$\beta(T)=\left\{\begin{aligned}
    &1~~~&0<T<T_c\,,\\
    &\frac{7}{8}~~~&T=T_c\,,\\
    &\frac{1}{2}~~~&T>T_c\,.
\end{aligned}\right.$$

Our main goal is to show that
\begin{equation}\label{eq:RN-of-RFFK}
    \frac{\phi^{\gamma,\eps h}_{p,\lamn}(\omega)}{\phi^{\gamma,0}_{p,\lamn}(\omega)}=Z(h)\exp\left(F(h,\omega)\right)
    \end{equation} is close to 1 for typical disorder when $\eps\ll N^{-\beta(T)}$ where 
\begin{equation}\label{Def-f(h,omega)}
Z(h):=\frac{\cZ^{\gamma,0}_{\phi,p,\lamn}}{\cZ^{\gamma,\eps h}_{\phi,p,\lamn}},\quad F(h,\omega)=\sum_{\cC\in\fC(\omega)}f{\left(\frac{\eps h_\cC}{T}\right)}.
\end{equation}
In Section \ref{sec:concentration-of-Z}, we show the concentration of $Z(h)$ using a similar technique as in \cite[Theorem 3.14]{FSZ16} and in Section \ref{sec: Concentration for the Radon Nikodym derivative}, we apply Lemma~\ref{lem:decoupling} to show the absolute continuity.

\subsection{Concentration of $Z(h)$}
\label{sec:concentration-of-Z}
We start by controlling the ratio of partition functions with and without external field. We expand the ratio between the partition functions with and without disorder as follows:
\begin{align}
       \frac{1}{Z(h)}=\frac{\cZ^{\gamma,\eps h}_{\phi,p,\lamn}}{\cZ^{\gamma,0}_{\phi,p,\lamn}}=\frac{\cZ^{\gamma,\eps h}_{\mu,T,\lamn}}{\cZ^{\gamma,0}_{\mu,T,\lamn}}
       &= \Bigg\langle ~\exp\left(\sum_{v\in \lamn}\frac{\eps h_v \sigma_v}{T}\right) ~ \Bigg\rangle_{T,\lamn}^{\gamma,0}\nonumber\\ &= \Bigg\langle ~\prod_{v\in \lamn}\left(\cosh\left(\frac{\eps h_v }{T}\right)+ \sigma_v\sinh\left(\frac{\eps h_v }{T}\right)\right)~\Bigg\rangle_{T,\lamn}^{\gamma,0}\nonumber\\ &= \prod_{v\in \lamn} \cosh\left(\frac{\eps h_v }{T}\right) \Bigg\langle ~\prod_{v\in \lamn}\left(1+ \sigma_v\tanh\left(\frac{\eps h_v }{T}\right)\right)~\Bigg\rangle_{T,\lamn}^{\gamma,0}.\label{eq: partition function expansion}
\end{align}
To control the right-hand side of \eqref{eq: partition function expansion}, we recall the following techniques from \cite{BS22} and \cite{DHX23}.

\begin{defi}\label{def: good external field for partition function}
    Let $\Omega$ be a region such that $\lamn\subset\Omega\subset\Lambda_{2N}$ and $\gamma\in \{-1,1\}^{\pari\Omega}$ be a boundary condition on $\pari\Omega$. 
    Let $\cA_0\subset\R^{\Omega}\otimes\Xi$ denote the collection of pairs $(h,\omega)$ such that the property \emph{(P0)} holds, \begin{equation}\label{eq: good external field for partition function}
        (\text{\emph{P}}0):\quad \Bigg|\Bigg\langle ~ \prod_{v\in\Omega}\Bigg[1+ \sigma_v\tanh\left(\frac{\eps h_v}{T}\right)\Bigg]~\Bigg\rangle^{\gamma,0}_{T,\Omega }-1\Bigg|\le \sqrt{\eps N^{\beta(T)}}.
    \end{equation} Note that actually $\cA_0$ depends only on $h$, so we could define $\cA_0\subset\R^{\Omega}$ to be the collection of the external field $h$ such that  the property \emph{(P0)} holds. However, for further convenience, we keep $\omega$ in the definition of $\cA_0$. 
    %Let $\cH^0\subset\R^{\Omega}$ denote the collection of the external field $h$ such that \begin{equation}\label{eq: good external field for partition function}
    %     \Bigg|\Bigg\langle ~ \prod_{v\in\Omega}\Bigg[1+ \sigma_v\tanh\left(\frac{\eps h_v}{T}\right)\Bigg]~\Bigg\rangle^{\gamma,0}_{T,\Omega }-1\Bigg|\le \sqrt{\eps N^{\beta(T)}}.
    %\end{equation} Then $\mcc H^0$ is measurable under the probability measure $\mathbb P$.
\end{defi}
\begin{rmk}
    Of course, the set $\cA_0$ relies on $\gamma,\Omega,T$ and $\eps,$ but we omit them from the notation for concision. The same conventions will be used several times without further notice.
\end{rmk}
\begin{lem}\label{lem: small perturbation for partition function}
   Take $\Omega=\Lambda_N,$
   then there exist constants $c_1,c_2>0$ depending only on $T$ such that $\P\otimes\phi_{p,\Omega}^{\gamma,0}(\cA^0)\ge 1-c_1\exp\left(-c_1^{-1}\sqrt{\eps^{-1}N^{-\beta(T)}}\right)$ for any disorder strength  {$\epsilon\leq c_2N^{-\beta(T)}$}  and any boundary condition $\gamma$  on $\pari\lamn$.
\end{lem}
\begin{proof}
    The case $T=T_c$ is covered by \cite[Lemma 3.4]{DHX23}. The cases that $T<T_c$ and $T>T_c$ follow a similar proof strategy.
    Here, we only highlight the key adaptations required.
    Note that 
    $$\Bigg\langle ~ \prod_{v\in\lamn}\Bigg[1+ \sigma_v\tanh\left(\frac{\eps h_v}{T}\right)\Bigg]~\Bigg\rangle^{\gamma,0}_{T,\lamn }-1=\sum_{k=1}^\infty\sum_{J\subset \lamn, |J|=k}\langle\sigma^J\rangle^{\gamma,0}_{T,\lamn}\prod_{v\in J}\tanh\left(\frac{\eps h_v}{T}\right).$$
To control the right-hand side, using \cite[Lemma~B.1]{DHX23}, the core is to bound 
$$A_k=\eps\Big[\sum_{J\subset \lamn, |J|=k}\left(\langle\sigma^J\rangle^{\gamma,0}_{T,\lamn}\right)^2\Big]^{\frac{1}{2k}}.$$

For $T<T_c$, using $|\sigma_x|=1$, we can obtain
$$A_k\leq \eps \binom{N^2}{k}^{\frac{1}{2k}}\leq c_2N^{-1}\left(\frac{N^{2k}}{k!}\right)^{\frac{1}{2k}}=\frac{c_2}{(k!)^{\frac{1}{2k}}}.$$

For $T>T_c$, by the Edward-Sokal coupling between the Ising model and the FK-Ising model, we have 
$\langle\sigma^J\rangle^{\gamma,0}_{T,\lamn}= \phi^{\gamma,0}_{p,\lamn}(\mathtt{Even}_J),$
where $\mathtt{Even}_J$ stands for the event that
each cluster in $\omega^{ \gamma}$  contains an even number of points in $J$. 

Furthermore, the probability of a cluster being large decays exponentially when $p<p_c,$ i.e., when $T>T_c$, we have for any boundary condition $\gamma$  ({see for example \cite{DT20}}):
\begin{equation}\label{eq: exponential decay for high temperature}
\phi_{p,\Omega}^{\gamma,0}\left(x\longleftrightarrow \pari\Lambda_m(x)\right)\leq e^{-cm}.
\end{equation}
 For any $J \subset\lamn$, let $d_x=\frac{1}{2}\cdot dist(x,\pari\lamn\cup J\setminus\{x\})$ and let $\cR_{x}=\Lambda_{d_x}(x)$. Then it is easy to get that \begin{equation}\label{eq: even event reduction}
     \mathtt{Even}_J\subset\bigcap_{x\in J}\{x\longleftrightarrow\pari\cR_{x}\}.
 \end{equation} Combining \eqref{eq: exponential decay for high temperature} and \eqref{eq: even event reduction}, we get that \begin{equation}
\langle\sigma^J\rangle^{\gamma,0}_{T,\lamn}\le \prod_{x\in J}\phi_{p,\mcc R_x}^{\mathrm{w},0}\left(\{x\longleftrightarrow\pari\cR_{x}\}\right)\le\exp{\Big(-\sum_{x\in J}c\cdot d_x\Big)}.\label{eq: second moment for Ising connecting probability}
    \end{equation} Then we get from the calculation in \cite[Lemmas 8.1 and 8.3]{FSZ16} that \begin{equation}
        \sum_{J\subset \lamn, |J|=k}\left(\langle\sigma^J\rangle^{\gamma,0}_{T,\lamn}\right)^2\le \frac{\prod_{j=1}^k(C_1N+C_1k)}{k!}.\label{eq: second moment for Ising connecting probability 2}
    \end{equation} Thus, we get that
    \begin{align*}
        A_k&\leq \eps \left(\frac{\prod_{j=1}^k(C_1N+C_1k)}{k!}\right)^{\frac{1}{2k}}\leq \eps\left(\frac{(C_1N+C_1k)^k}{k!}\right)^{\frac{1}{2k}}\\&\leq\eps\cdot\frac{\sqrt{C_1N+C_1k}}{\sqrt{k/e}}\leq\frac{\sqrt{C_1N}\eps }{\sqrt{k/e}}+\sqrt{eC_1}\eps.\qedhere
    \end{align*}
With the upper bound for $A_k$, the rest of the proof is the same to that in  \cite[Lemma 3.4]{DHX23}.
\end{proof}
In order to meet further needs, we prove a stronger version of Lemma~\ref{lem: small perturbation for partition function} even though it will not be used in this section.
\begin{lem}\label{lem: small perturbation for partition function 2}
   Let $\Omega$ be a region such that $\lamn\subset\Omega\subset\Lambda_{2N}$ and $\gamma\in \{-1,1\}^{\pari\Omega}$ be a boundary condition on $\pari\Omega$.\begin{enumerate}
       \item If $T\ge T_c$ and $\gamma=\f$ is the free boundary condition, then there exist constants $c_1,c_2>0$ depending only on $T$ such that $\P(\cH^0)\ge 1-c_1\exp\left(-c_1^{-1}\sqrt{\eps^{-1}N^{-\beta(T)}}\right)$ for all disorder strength  {$\epsilon\leq c_2N^{-\beta(T)}$}.
       \item If $T< T_c$, then  there exist constants $c_3,c_4>0$ depending only on $T$ such that $\P(\cH^0)\ge 1-c_3\exp\left(-c_3^{-1}\sqrt{\eps^{-1}N^{-\beta(T)}}\right)$ for all disorder strength  {$\epsilon\leq c_4N^{-\beta(T)}$} and any boundary condition $\gamma$.
   \end{enumerate} 
\end{lem}
\begin{proof}
   The proof is similar to the proof of Lemma~\ref{lem: small perturbation for partition function}, so we follow the notations there. The case
    that $T<T_c$ follows directly from the proof of Lemma~\ref{lem: small perturbation for partition function}.
    
    If $T\ge T_c$ and $\gamma=\f$ is the free boundary, then we get from CBC that $\langle\sigma^J\rangle^{\f,0}_{T,\Omega}\leq \langle\sigma^J\rangle^{\w,0}_{T,\Lambda_{2N}}.$ The desired bound on $A_k$ then follows from the proof of Lemma~\ref{lem: small perturbation for partition function}. 
\end{proof}

Further, in order to control the product of hyperbolic cosine in the right-hand side of \eqref{eq: partition function expansion}, we give the following definition.
\begin{defi}
Let $\cA_1\subset\R^{\Omega}\otimes\Xi$ denote the collection of pairs $(h,\omega)$ such that the property \emph{(P1)} holds \begin{equation}\label{eq: good external field for product of consine hyperbolics}
      (\text{\emph{P}}1):  \quad \Big|\sum_{v\in\lamn}f\left(\frac{\eps h_v }{T}\right)-\frac{2\eps^2N^2}{T^2}\Big|\le \sqrt{\eps^2N}.
    \end{equation}
    %Let  $\cH^1$ denote the set of external field that the following (P1) property holds: \begin{equation}\label{eq: good external field for product of consine hyperbolics}
    %    (\text{\emph{P}}1):\quad \Big|\sum_{v\in\lamn}f\left(\frac{\eps h_v }{T}\right)-\frac{2\eps^2N^2}{T^2}\Big|\le \sqrt{\eps^2N}.
    %\end{equation}
\end{defi} 
\begin{lem}\label{lem: good external field for product of hyperbolic cosine}
    There exist constants $c_1>0$ and $c_2>0$ depending only on $T$ such that for any disorder strength $\eps\le  c_1 N^{-\frac{1}{2}}$, we have $\P\otimes\phi_{p,\Omega}^{\gamma,0}(\cA_1)\ge 1-c_2\exp\left(-c_2^{-1}\eps^{-1}N^{-\frac{1}{2}}\right)$.
\end{lem}\begin{proof}
Recall that $f(x)=\ln\cosh(x)$. Given the inequality $\frac{x^2-x^4}{2}\le f(x)\le \frac{x^2}{2}$, we calculate that \begin{align}
    &\P\left(\Big|\sum_{v\in\lamn}f\left(\frac{\eps h_v}{T}\right)-\frac{2\eps^2N^2}{T^2}\Big|> \sqrt{\eps^2N}\right)\nonumber\\\le~&\P\left(\Big|\sum_{v\in\lamn}\frac{\eps^2h_v^2}{2T^2}-\frac{2\eps^2N^2}{T^2}\Big| >\frac{\sqrt{\eps^2N}}{2}\right)+\P\left(\Big|\sum_{v\in\lamn}\frac{\eps^4h_v^4}{2T^4}\Big| >\frac{\sqrt{\eps^2N}}{2}\right)\nonumber\\\le~&c_2\exp\left(-c_2^{-1}\eps^{-\frac{1}{2}}N^{-\frac{1}{4}}\right)
\end{align} where the last inequality comes from Lemma~\ref{lem: chi fourth power tail bound} and letting $c_1>0$ small enough.
\end{proof}
Combining \eqref{eq: partition function expansion}, \eqref{eq: good external field for partition function} and \eqref{eq: good external field for product of consine hyperbolics}, we obtain that for $(h,\omega)\in\cA_0\cap\cA_1,$ 
\begin{equation}\label{eq: two-sided bound for partition function}
         \begin{aligned}
\left(1-\sqrt{\eps N^{\beta(T)}}\right)\left(1-\sqrt{\eps^2 N}\right)&\le Z(h)\cdot\exp\left(\frac{2\eps^2N^2}{T^2}\right)&\le\left(1+\sqrt{\eps N^{\beta(T)}}\right)\left(1+\sqrt{\eps^2 N}\right).
         \end{aligned}
     \end{equation}

\subsection{Concentration of $F(h,\omega)$}
\label{sec: Concentration for the Radon Nikodym derivative}

Recalling the definition of $F(h,\omega)$ in \eqref{Def-f(h,omega)} and noting that $\frac{x^2-x^4}{2}\leq f(x)\leq \frac{x^2}{2}$, 
we have 
\begin{equation}\label{eq: f two sided bound}
    \sum_{\cC\in\fC}\frac{\eps^2 h_{\cC}^2}{2T^2}-\frac{\eps^4 h_{\cC}^4}{2T^4}\leq F(h,\omega)\leq \sum_{\cC\in\fC}\frac{\eps^2 h_{\cC}^2}{2T^2}.
\end{equation}
Combined with Lemma~\ref{lem:decoupling}, we can consider the ``typical'' event such that $F(h,\omega)$ concentrates well under the measure $\P\otimes\fk\gamma$.
%Therefore, we can consider the ``typical'' external field such that $F(h,\omega)$ concentrates well.
\begin{defi}\label{def:good external field for upper bound}
We denote by $\cA_2$ and $\cA_3$ the collections of pairs $(h,\omega)$ such that the following properties \emph{(P2)} and \emph{(P3)} hold respectively.
    \begin{equation}\label{eq:good external field and clusters}
    (\text{\emph{P}}2):\quad F(h,\omega)-\frac{2\eps^2N^2}{T^2}\le \eps N^{\alpha(T)},
    \end{equation}
    \begin{equation}(\text{\emph{P}}3):\quad\label{eq:perfect external field and clusters}
        F(h,\omega)-\frac{2\eps^2N^2}{T^2}\ge -\eps N^{\alpha(T)}.
    \end{equation}

%We denote by $\mcc H^2$ and $\mcc H^3$ the sets of the external field $\{h_v:v\in \lamn\}$ that the following properties \emph{(P2)} and \emph{(P3)} hold respectively.
%    \begin{equation}\label{eq:good external field and clusters}
%    (\text{\emph{P}}2):\quad \fk\gamma\left(F(h,\omega)-\frac{2\eps^2N^2}{T^2}\le \eps N^{\alpha(T)}\right) \geq 1-\exp\left(-\eps^{-\frac{1}{4}} N^{-\frac{\alpha(T)}{4}}\right)\,,
%    \end{equation}
    %\begin{equation}(\text{\emph{P}}3):\quad\label{eq:perfect external field and clusters}
    %    \fk\gamma\left(F(h,\omega)-\frac{2\eps^2N^2}{T^2}\ge -\eps N^{\alpha(T)}\right) \geq 1-\exp\left(-\eps^{-\frac{1}{4}} N^{-\frac{\alpha(T)}{4}}\right)\,.
    %\end{equation}

\end{defi}

\begin{lem}\label{lem:good external field probability}
There exist constants $c_1=c_1(T)>0$ and $c_2=c_2(T)>0$ such that for any disorder strength $\eps<c_1N^{-\alpha(T)}$ we have
    $$\P\otimes\fk\gamma(\cA_2\cap \cA_3)\ge 1-c_2\exp\big(-c_2^{-1}\eps^{-\frac{1}{2}} N^{-\frac{\alpha(T)}{2}}\big).$$
\end{lem}

\begin{proof} 
%The proof is essentially an application of the Markov's inequality. 
%We use the notation $\P\otimes\fk\gamma$ to denote the product measure of the external field and the FK-Ising measure, where we emphasize that the second component is simply the FK-Ising measure without disorder. Let $\cA_2, \cA_3\subset \mbb R^{\Lambda_N}\otimes \{0,1\}^{\edge(\Lambda_N)}$ denote the events that 
%$$\sum_{\cC\in\mathfrak{C}} f\left({\frac{\eps h_{\cC}}{T}}\right)-\frac{\eps^2N^2}{2T^2}\le \eps N^{\alpha(T)}~~\mbox{and}~~\sum_{\cC\in\mathfrak{C}} f\left({\frac{\eps h_{\cC}}{T}}\right)-\frac{\eps^2N^2}{2T^2}\geq -\eps N^{\alpha(T)}.$$
%Then we have (for $i=2,3$): 
%\begin{equation}\label{eq-l-level-good-bound-2}
%\P\otimes\fk\gamma(\cA_i^{c})  \,\geq \,\P\otimes\fk\gamma\left(\cA_i^c|(\cH^i)^c\right)\times\P\left((\cH^i)^c\right)\ge \exp\left(-\eps^{-\frac{1}{4}} N^{-\frac{\alpha(T)}{4}}\right)\times\P\left((\cH^{i})^c\right)\,.
%\end{equation} 
%Therefore, all we need is to show that $\cA_i^c$ has a small probability under $\P\otimes\fk\gamma$.

For any $\omega$ fixed, we write $X_\mcc C=\frac{h_\mcc C}{\sqrt{|\mcc C|}}$, then we get that $X_\mcc C~(\mcc C\in \mathfrak{C})$'s are i.i.d. Gaussian variables with mean 0 and variance 1 and we denote the law as $\mbb P_X$. Now we can rewrite the second moment term in \eqref{eq: f two sided bound} as 
$$\sum_{\cC\in\fC}\frac{\eps^2 h_{\cC}^2}{2T^2}=\sum_{\cC\in \fC(\omega)}\frac{\eps^2|\cC|}{2T^2}(X_{\cC}^2-1)+\frac{2\eps^2 N^2}{T^2}.$$
The first term $\sum_{\cC\in \fC(\omega)}\frac{\eps^2|\cC|}{2T^2}(X_{\cC}^2-1)$ is a centralized summation of $\chi$-square distributions which can be estimated  using Lemma~\ref{lem: chi fourth power tail bound}.
%\cite[Lemma 1]{LM00}. 
Let $\mss S=\{\omega:\sum_{\cC\in\fC{(\omega)}}|\cC|^2\le \eps^{-1}N^{3\alpha(T)}\}.$ 
Then for any $\omega\in\mss S$, we have by \eqref{eq: f two sided bound}
\begin{align}
        &\P\left(\sum_{\cC\in\fC} f\left({\frac{\eps h_{\cC}}{T}}\right)-\frac{2\eps^2N^2}{T^2}> \eps N^{\alpha(T)}\right)\nonumber\\
        \leq~& {\mbb \P_X}\left(\sum_{\cC\in\fC{(\omega)}} \frac{\eps^2|\cC|(X_{\cC}^2-1)}{2T^2}> \eps N^{\alpha(T)}\right)
        \leq C_1\exp\left(-C_1^{-1}\eps^{-\frac{1}{2}} N^{-\frac{\alpha(T)}{2}}\right)\,.\label{eq-l-level-good-bound}
\end{align} 
The last inequality follows from Lemma~\ref{lem: chi fourth power tail bound}.
%the exponential tail bound for $\chi$-square distributions (\cite[Lemma 1]{LM00}). 
Similarly, we have by \eqref{eq: f two sided bound}
 \begin{align}
~&\P\left(\sum_{\cC\in\fC} f\left({\frac{\eps h_{\cC}}{T}}\right)-\frac{2\eps^2N^2}{T^2}< -\eps N^{\alpha(T)}\right)\nonumber\\
\le~&{\P}\left(\sum_{\cC\in\fC}\frac{\eps^2 h_{\cC}^2}{2T^2}-\frac{\eps^4 h_{\cC}^4}{2T^4}-\frac{2\eps^2N^2}{T^2}< -\eps N^{\alpha(T)}\right)\nonumber\\
        \le~& {\P}\left(\sum_{\cC\in\fC}\frac{\eps^2 h_{\cC}^2-\eps^2|\cC|}{2T^2}< -\frac{1}{2}\eps N^{\alpha(T)}\right)+{\P}\left(\sum_{\cC\in\fC}\frac{\eps^4 h_{\cC}^4}{2T^4}> \frac{1}{2}\eps N^{\alpha(T)}\right)\nonumber\\
        =~& {\P_X}\left(\sum_{\cC\in\fC} \frac{\eps^2|\cC|}{2T^2}(X_{\cC}^2-1)< -\frac{1}{2}\eps N^{\alpha(T)}\right)+{\P_X}\left(\sum_{\cC\in\fC}\frac{\eps^4|\cC|^2}{2T^4}X_{\cC}^4> \frac{1}{2}\eps N^{\alpha(T)}\right).\label{eq-perfect-bound 0}
\end{align}
%Again, by the exponential tail bound for $\chi$-square distributions (\cite[Lemma 1]{LM00}), 
Again by Lemma~\ref{lem: chi fourth power tail bound}, we have \begin{align}
    {\P_X}\left(\sum_{\cC\in\fC} \frac{\eps^2|\cC|}{2T^2}(X_{\cC}^2-1)< -\frac{1}{2}\eps N^{\alpha(T)}\right)&\le C_1\exp\left(-C_1^{-1}\eps^{-\frac{1}{2}} N^{-\frac{\alpha(T)}{2}} \right),\label{eq-perfect-bound 1}\\
    \mbox{and }~~~ {\P_X}\left(\sum_{\cC\in\fC}\frac{\eps^4|\cC|^2}{2T^4}X_{\cC}^4> \frac{1}{2}\eps N^{\alpha(T)}\right)&\le C_1\exp\left(-C_1^{-1}\eps^{-\frac{1}{2}} N^{-\frac{\alpha(T)}{2}} \right).\label{eq-perfect-bound 2}
\end{align}  Combining \eqref{eq-perfect-bound 1} and \eqref{eq-perfect-bound 2} with \eqref{eq-perfect-bound 0}, we get that \begin{equation}
    \P\left(\sum_{\cC\in\fC} f\left({\frac{\eps h_{\cC}}{T}}\right)-\frac{2\eps^2N^2}{T^2}< -\eps N^{\alpha(T)}\right)\le 2C_1\exp\left(-C_1^{-1}\eps^{-\frac{1}{2}} N^{-\frac{\alpha(T)}{2}} \right).\label{eq-}
\end{equation}

In addition, we have by Corollary~\ref{cor: sum of squares of clusters all temperature} that
\begin{equation}\label{eq-l-level-good-bound-for-cluster-deviation}
    \fk\gamma(\sS)\ge 1-C_2\exp\left(-C_2^{-1}\eps^{-1}N^{-\alpha(T)}\right).
\end{equation}
Combining \eqref{eq-l-level-good-bound}, \eqref{eq-} with \eqref{eq-l-level-good-bound-for-cluster-deviation} shows that \begin{equation}\label{eq-l-level-bound-typical}
    \P\otimes\fk\gamma(\cA_i^c)\le C_3\exp\left(-C_3^{-1}\eps^{-\frac{1}{2}} N^{-\frac{\alpha(T)}{2}}\right)
\end{equation} for $i=2,3,$ some $C_3>0$ and sufficiently large $N$. Thus we complete the proof.
%Combining \eqref{eq-l-level-bound-typical} with \eqref{eq-l-level-good-bound-2}, we complete the proof of the lemma.
\end{proof}

\begin{proof}[Proof of Theorem~\ref{thm-all-temperature}: part one]
       Let $\cA=\cA_0\cap\cA_1\cap\cA_2\cap\cA_3$, then combining Lemmas~\ref{lem: small perturbation for partition function}, \ref{lem: good external field for product of hyperbolic cosine} and \ref{lem:good external field probability}  shows $$\P\otimes\fk\gamma(\cA)\ge 1-C_1\exp\left(-C_1^{-1}\eps^{-\frac{1}{2}} N^{-\frac{\alpha(T)}{2}}\right)$$ for some constant $C_1>0$.
    For $(h,\omega)\in\cA$, by  \eqref{eq:good external field and clusters} and \eqref{eq:perfect external field and clusters}, we compute that \begin{equation}\label{eq: two-sided bound for energy function}
         \exp(-\eps N^{\alpha(T)})\le\frac{\prod_{\cC\in\fC}\cosh(\frac{\eps h_{\cC}}{T})}{\exp(\frac{2\eps^2N^2}{T^2})}\le \exp\left(\eps N^{\alpha(T)}\right).
    \end{equation} 
    Plugging \eqref{eq: two-sided bound for partition function} and \eqref{eq: two-sided bound for energy function} into Lemma~\ref{lem:decoupling} gives the desired result.
\end{proof}

\section{Near-critical behaviour}\label{sec: near-critical behavior}

The crucial input for Section~\ref{sec:singularity} is to lower-bound the influence of external field on an $M$ box (with $M\sim \eps^{-\frac{1}{\alpha(T)}}$) and we will do it in this section for all temperature regimes. To be precise, we will show that as long as $\eps > cN^{-\alpha(T)}$ for some constant $c>0$,  $\E\TV{\phi_{p,\lamn}^{\gamma,0}}{\phi_{p,\lamn}^{\gamma,\eps h}}$ is bounded away from $0$ for $T\geq T_c$
in Sections~\ref{sec: anti concentration critical} and \ref{sec: anti concentration high}. A similar lower bound in the low-temperature regime will be obtained in
Section \ref{sec: anti concentration low} with different emphasis according to the ingredients that will be used in Section~\ref{sec: coarse graining at low temperature}.
Finally, as a byproduct, we will show that
$\alpha(T)$ is the critical threshold of the choice of exponent
by proving the upper bound for total variation distance on an $M$ box in all temperature regimes in Section~\ref{sec: upper bound at near critical disorder}. 

We point out that the primary computational effort of this work is concentrated in this section. Readers less interested in technical derivations may proceed directly to Section~\ref{sec:singularity}, as the continuity of the exposition remains intact when assuming the core results from Section~\ref{sec: near-critical behavior}.

\subsection{The critical-temperature case}\label{sec: anti concentration critical}
We first consider the critical-temperature case.   We prove the following result on general domains that have been used in Section~\ref{sec: crit+high coarse-graining}.
\begin{prop}\label{prop: anti-absolute continuity at critical temperature}
    Fix $T=T_c(2)$ and thus $p=p_c$. For any constant $\theta>0$, there exists a constant $c=c(\theta)>0$ such that $ \E\TV{\fkco\f}{\fkhco\f}\ge c$ for any disorder strength $\theta N^{-\frac{15}{16}}\le \eps\le 2\theta N^{-\frac{15}{16}}$, and any domain $\lamn\subset\Omega\subset\Lambda_{2N}$.
\end{prop}

To obtain the lower bound, recall that in Section~\ref{sec:concentration-of-Z} we have shown  that the ratio between the partition functions of the FK-Ising model on the domain $\Omega$ with and without disorder concentrates well.
Therefore, 
it suffices to show that the product of $\cosh(\frac{\eps h_{\cC}}{T_c})$ does not always concentrate at the product of $\exp(\eps^2|\cC|/2T_c^2)$. Moreover, let $\cC_{\diamond}$ denote the maximal cluster and note that it has typical volume $cN^{\frac{15}{8}}$, thus we obtain from $\eps\ge \theta N^{-\frac{15}{16}}$ that $\eps h_{\cC_{\diamond}}$ is typically $O(1)$.  Hence $\cosh{(\frac{\eps h_{\mcc C_\diamond}}{T_c})}$ will have a positive probability to be much smaller than $\exp(\eps^2|\cC_{\diamond}|/2T_c^2)$, which leads to the anti-concentration result. To carry this out, we introduce the following definition with idea similar to 
Definition \ref{def:good external field for upper bound}.
\begin{defi}\label{def: anticoncentrate external field at Tc}
    For any constant $c>0$, we use the notation $\cH_{\mathrm{cri}}(c)$ to denote the set of external field such that 
    \begin{equation*}
        \fkco\f\left(\sum_{\cC\in\mathfrak{C}} f\left({\frac{\eps h_{\cC}}{T_c}}\right)-\frac{\eps^2|\Omega|}{2T_c^2}\le -1\right) \geq c\,.
    \end{equation*}
\end{defi}
\begin{lem}\label{lem: anticoncentrate external field at Tc}
There exist constants $c_1,c_2>0$ depending only on $\theta$ such that for any disorder strength $\theta N^{-\frac{15}{16}}\le \eps\le 2\theta N^{-\frac{15}{16}}$, domain $\lamn\subset\Omega\subset\Lambda_{2N}$ and boundary condition $\gamma$ we have $\P(\cH_{\mathrm{cri}}(c_1))\ge c_2$.
\end{lem}
\begin{proof}
We still consider the product measure $\mbb P\otimes \phi_{p_c,\Omega}^{f,0}.$
    Let $\cB\subset \mbb R^{\Omega}\otimes\{0,1\}^{\mathrm{E}(\Omega)}$ denote the event that \begin{equation}\label{eq: def of cB}
        \sum\limits_{\cC\in\mathfrak{C}} f\left({\frac{\eps h_{\cC}}{T_c}}\right)-\frac{\eps^2|\Omega|}{2T_c^2}\le -1.
    \end{equation} Further, we define $\cB_1$ and $\cB_2$ to be the following events respectively: $$\sum\limits_{\cC\in\mathfrak{C}\setminus\{\cC_{\diamond}\}} \left(f\left({\frac{\eps h_{\cC}}{T_c}}\right)-\frac{\eps^2|\cC|}{2T_c^2}\right)\le 1\mbox{  and  } f\left({\frac{\eps h_{\cC_{\diamond}}}{T_c}}\right)-\frac{\eps^2|\cC_{\diamond}|}{2T_c^2}\le -2.$$ Thus, $\cB_1\cap\cB_2\subset\cB$. Now we are going to lower-bound the probability of $\cB_1$ and $\cB_2$.
    Applying the exponential Markov inequality, we obtain that for any configuration $\omega$ \begin{equation}\label{eq: fluctuate at Tc 1}
        \P\otimes\fkco\f(\cB_1\mid\omega)\ge 1-e^{-1}\prod_{\cC\in\mathfrak{C}\setminus\{\cC_{\diamond}\}}\E\cosh\left(\frac{\eps h_{\cC}}{T_c}\right)\exp\left(\frac{-\eps^2|\cC|}{2T_c^2}\right)=1-e^{-1}.
    \end{equation} To lower-bound $\cB_2$, for any configuration $\omega\in\sS_{\diamond}=\left\{|\cC_{\diamond}|\ge \frac{100T_c^2}{\eps^2}\right\}$, we apply the inequality $\cosh(x)\le \exp(|x|)$ to get that \begin{equation}\label{eq: fluctuate at Tc 2}
        \P\otimes\fkco\f(\cB_2\mid\omega)\ge \P_X\left(\sqrt{\frac{\eps^2|\cC_{\diamond}|}{T_c^2}}|X|-\frac{\eps^2|\cC_{\diamond}|}{2T_c^2}\le -2\right)\ge 0.9
    \end{equation} where $X$ denotes a normal random variable with mean $0$ and variance $1$ with law $\mbb P_X$. Combining \eqref{eq: fluctuate at Tc 1} and \eqref{eq: fluctuate at Tc 2} shows that \begin{equation*}
        \P\otimes\fkco\f(\cB)\ge \sum_{\omega\in\sS_\diamond}\P\otimes\fkco\f(\cB\mid\omega)\cdot\fkco\f(\omega)\ge (1-e^{-1}-0.1)\fkco\f(\sS_{\diamond}).
    \end{equation*} The desired result comes from Theorem~\ref{thm:LDP lower bound for maiximal cluster} and Markov inequality.
\end{proof}
With Lemma~\ref{lem: anticoncentrate external field at Tc}, we are ready to prove Proposition~\ref{prop: anti-absolute continuity at critical temperature}.
\begin{proof}[Proof of Proposition~\ref{prop: anti-absolute continuity at critical temperature}]
    Let $c_1,c_2$ be defined in Lemma~\ref{lem: anticoncentrate external field at Tc}. For $h\in \cH^0\cap\cH_{\mathrm{cri}}(c_1)$ (recall $\mcc H^0$ from Definition~\ref{def: good external field for partition function})  and $(h,\omega)\in \cB$ (recall $\cB$ from \eqref{eq: def of cB}), we have by Lemma~\ref{lem: small perturbation for partition function 2}
    \begin{equation*}
        \frac{\fkhco\f}{\fkco\f}=\frac{\cZ^{\f,0}_{\phi,p_c,\Omega}}{\cZ^{\f,\eps h}_{\phi,p_c,\Omega}}\cdot \prod_{\cC\in\mathfrak{C}}\cosh\left(\frac{\eps h_{\cC}}{T_c}\right)\le e^{-1}\left(1+\sqrt{\eps N^{{\frac{7}{8}}}}\right)\le \frac{1}{2}
    \end{equation*} for sufficiently large $N$. Thus, we obtain that \begin{equation*}
        \TV{\fkhco\f}{\fkco\f}\ge \frac{1}{2}\fkco\f(\cB\mid h)=\frac{1}{2}\fkco\f\left(\sum_{\cC\in\mathfrak{C}} f\left({\frac{\eps h_{\cC}}{T_c}}\right)-\frac{\eps^2|\Omega|}{2T_c^2}\le -1\right)\ge\frac{c_1}{2}.
    \end{equation*}
    Combining Lemmas \ref{lem: small perturbation for partition function} and \ref{lem: anticoncentrate external field at Tc} and integrating over $\cH^0\cap\cH_{\mathrm{cri}}(c_1)$, we obtain that \begin{equation*}
        \E\TV{\fkhco\f}{\fkco\f}\ge \frac{c_1}{2}\P\big(\cH^0\cap\cH_{\mathrm{cri}}(c_1)\big)\ge \frac{c_1}{2}\left(c_2-c\exp\left(-c^{-1}\sqrt{\eps^{-1}N^{-\frac{7}{8}}}\right)\right)\ge \frac{c_1c_2}{3}
    \end{equation*} for sufficiently large $N$.
\end{proof}

\subsection{The high-temperature case}\label{sec: anti concentration high}
In this subsection, we consider the high-temperature case.% and fix $p<p_c$. 
\begin{prop}\label{prop: anti-absolute continuity at high temperature}
    Fix $T>T_c(2)$ and thus $p<p_c$. Then there exists a constant $\theta_0=\theta_0(p)>0$ such that  for any constant $\theta\in (0,\theta_0)$, there exists a constant $c=c(\theta,p)>0$ such that $ \E\TV{\fkhhof}{\fkhof}\ge c$ for any disorder strength $\theta N^{-\frac{1}{2}}\le \eps\le 2\theta N^{-\frac{1}{2}}$.
\end{prop}
% \begin{rmk}

% In fact, the proof presented in this subsection is valid only for values of $\theta\in(0,\theta_0)$ with $\theta_0$ small enough. Regarding the comprehensive proof, Lemma~\ref{lem: good external field for product measure} establishes a coarse-graining framework that generalizes to all positive $\theta$, using the case that $\theta\in(0,\theta_0)$ as an input. While this arrangement may seem somewhat confusing, particularly for readers focusing solely on the proof of Proposition~\ref{prop: anti-absolute continuity at high temperature}, it is more appropriately integrated within the overall structure and organization of the paper.

% \end{rmk}

\begin{rmk}
One could easily remove the restriction of $\theta\in(0,\theta_0)$ via a coarse graining method, see  
Section \ref{sec: crit+high coarse-graining} for a detailed coarse graining framework (but with a different purpose). 
\end{rmk}

Unlike the critical-temperature case, we do not have a good concentration bound for the ratio between partition functions of the FK-Ising model with and without disorder in the high-temperature case. So we focus on the fluctuation coming from all the clusters and show that this will lead to an anticoncentration of the measure. We start with some definitions.
\begin{defi}\label{def: anticoncentrate external field at T big0}
    Let $\cF\subset\{0,1\}^{\mathrm{E}(\Omega)}$ denote the set of edge configurations that $\max_{\cC\in\mathfrak{C}}|\cC|\le N^{0.1}$ and $\sum_{\cC\in\mathfrak{C}}|\cC|^4\le 2c_1N^2$ where $c_1$ comes from Lemma~\ref{lem: LDP for high temp}. We use the notation $\cH^0_{\mathrm{hig}}$ to denote the set of external field such that \begin{equation}\label{eq: good external field for partition function at big T}
        \Bigg\langle\left(\sum_{\cC\in\mathfrak{C}}\frac{\eps^2h_{\cC}^2}{2T^2}-\frac{\eps^2|\Omega|}{2T^2}\right)^4\Bigg|\cF\Bigg\rangle\le \exp\left(\eps^{-1}N^{-1/2}\right).
    \end{equation}
    Here and in the following, we write 
    $\left\langle  X\mid \mcc F\right\rangle:=\frac{\left\langle X\cdot\1_{\mcc F}\right\rangle_{p,\Omega}^{\f,0}}{\left\langle  \1_{\mcc F}\right\rangle_{p,\Omega}^{\f,0}}$
    for any variable $X$ and event $\mcc F$.
\end{defi}

\begin{defi}\label{def: anticoncentrate external field at T big}
    For any constant $c>0$, we use the notation $\cH_{\mathrm{hig}}^1(c)$ to denote the set of external field such that 
    \begin{equation}\label{eq: fluctuation of all clusters} 
        \fkhof\left(\sum_{\cC\in\mathfrak{C}}f\left(\frac{\eps h_{\cC}}{T}\right)-\frac{\eps^2h_{\cC}^2}{2T^2}\ge -c\eps^4N^2\right)\ge 1-N^{-1}.
    \end{equation}
    
    Further, for any constant $c>0$, let $\cH_{\mathrm{hig}}^2(c)$ denote the collection of external field such that \begin{equation*}
    \var_{\phi}\left(\sum_{\cC\in\mathfrak{C}}\frac{\eps^2h_{\cC}^2}{2T^2}\Bigg|\cF\right)\ge c
    \end{equation*} where $\var_{\phi}$ denotes the variance operator under $\fkhof$ and $\var_{\phi}(\cdot|\cF)$ denotes the conditional variance operator with respect to $\cF$.
\end{defi}
\begin{lem}\label{lem: anticoncentrate external field at T big0}
    There exists a constant $c>0$ depending only on $p$ such that for any disorder strength $\eps>0$ and domain $\lamn\subset\Omega\subset\Lambda_{2N}$, we have \begin{equation*}
        \P\left(\cH^0_{\mathrm{hig}}\right)\ge 1 - c^{-1}\eps^8N^4\exp\left(-c{\eps^{-1}N^{-1/2}}\right).
    \end{equation*}
\end{lem}

\begin{lem}\label{lem: anticoncentrate external field at T big1}
    There exist constants $c_1,c_2>0$ depending only on $p$ such that for any disorder strength $\eps>0$ and domain $\lamn\subset\Omega\subset\Lambda_{2N}$, we have \begin{equation*}
        \P\left(\cH_{\mathrm{hig}}^1(c_1)\right){\ge 1-\frac{c_2}{N}}.
    \end{equation*}
\end{lem}
\begin{lem}\label{lem: anticoncentrate external field at T big2}
    There exists a constant $\theta_0<1$ such that the following holds. For any $\theta\le \theta_0$, there exist constants $c_3=c_3(p),c_4=c_4(p)>0$ such that for any disorder strength $\theta N^{-1/2}\le \eps\le 2\theta N^{-1/2}$, domain $\lamn\subset\Omega\subset\Lambda_{2N}$, we have \begin{equation*}
        \P\left(\cH_{\mathrm{hig}}^2(c_3\theta^4)\right)\ge c_4.
    \end{equation*}
\end{lem}

With Lemmas~\ref{lem: anticoncentrate external field at T big0}, \ref{lem: anticoncentrate external field at T big1} and \ref{lem: anticoncentrate external field at T big2}, we are ready to prove Proposition~\ref{prop: anti-absolute continuity at high temperature}.
\begin{proof}[Proof of Proposition~\ref{prop: anti-absolute continuity at high temperature}]
To show the lower bound, we consider $h\in \cH^0_{\mathrm{hig}}\cap\cH_{\mathrm{hig}}^1(c_1)\cap\cH_{\mathrm{hig}}^2(c_3\theta^4)$ where $c_1$ and $c_3$ are the constants defined in Lemmas~\ref{lem: anticoncentrate external field at T big1} and \ref{lem: anticoncentrate external field at T big2}. Applying Lemmas \ref{lem: anticoncentrate external field at T big0}, \ref{lem: anticoncentrate external field at T big1} and \ref{lem: anticoncentrate external field at T big2}, it suffices to show that $\TV{\fkhhof}{\fkhof}\ge C$ for any  $h\in\cH_{\mathrm{hig}}^0\cap\cH_{\mathrm{hig}}^1(c_1)\cap\cH_{\mathrm{hig}}^2(c_3\theta^4)$ and some constant $C>0$ independent of $h$. Let $X(\omega)=\sum_{\cC\in\mathfrak{C}}\frac{\eps^2h_{\cC}^2}{2T^2}-\frac{\eps^2|\Omega|}{2T^2}$. Applying \eqref{eq: good external field for partition function at big T}, we have 
\begin{equation}\label{eq: B3 property 1 high}
        \langle X^4\mid\cF\rangle\le C_1.
    \end{equation} Since $h\in \cH_{\mathrm{hig}}^2(c_3\theta^4)$, we have that \begin{equation}\label{eq: B3 property 2 high}
        \var_{\phi}(X\mid\cF)\ge c_3\theta^4.
    \end{equation} 
    Combining \eqref{eq: B3 property 1 high}, \eqref{eq: B3 property 2 high} and Corollary \ref{cor: dis concentration using variance}, we get that there exists a constant $C_2>0$ such that for any real number $A$ there exists a set of configurations $\sS_1=\sS_1(A)$ with $\fkhof(\sS_1\mid\cF)\ge C_2$ such that for any $\omega\in\sS_1(A)$, we have $|X(\omega)-A|\ge \frac{\sqrt{c_3}\theta^2}{4}$. Applying Lemma~\ref{lem: LDP for high temp}, we get that $\fkhof(\cF)\ge \frac{1}{2}$ and thus $\fkhof(\sS_1)\ge \frac{C_2}{2}$. Let $\sS_2$ denote the collection of configurations such that $0\ge \sum_{\cC\in\mathfrak{C}}f(\frac{\eps h_{\cC}}{T})-\frac{\eps^2|h_{\cC}^2|}{2T^2}\ge -c_1\eps^4N^2$. Then for any $\omega\in\sS_2$, we have \begin{align}
        &\Bigg|\ln\left(\frac{\fkhhof(\omega)}{\fkhof(\omega)}\right)\Bigg|= \Bigg|\sum_{\cC\in\mathfrak{C}}f\left(\frac{\eps h_{\cC}}{T}\right)+\ln\left(\frac{\cZ^{\f,0}_{\phi,p,\Omega}}{\cZ^{\f,\eps h}_{\phi,p,\Omega}}\right)\Bigg|\nonumber\\&\ge-\Bigg|\sum_{\cC\in\mathfrak{C}}f\left(\frac{\eps h_{\cC}}{T}\right)-\frac{\eps^2h_{\cC}^2}{2T^2}\Bigg|+\Bigg|X(\omega)+\ln\left(\frac{\cZ^{\f,0}_{\phi,p,\Omega}}{\cZ^{\f,\eps h}_{\phi,p,\Omega}}\right)+\frac{\eps^2|\Omega|}{2T^2}\Bigg|.\nonumber
    \end{align}
    Let    $A=-\ln\left(\frac{\cZ^{\f,0}_{\phi,p,\Omega}}{\cZ^{\f,\epsilon h}_{\phi,p,\Omega}}\right)-\frac{\eps^2|\Omega|}{2T^2}$, then for any $\omega\in \sS_1(A)\cap\sS_2$, we have  $$\Bigg|\ln\left(\frac{\fkhhof(\omega)}{\phi^{\f,0}_{p,\Omega}
        (\omega)}\right)\Bigg|\ge \frac{\sqrt{c_3}\theta^2}{4}-c_1\eps^4N^2\ge \frac{\sqrt{c_3}\theta^2}{4}-16c_1\theta^4.$$
     Note that $f(x)\le \frac{x^2}{2}$, then \eqref{eq: fluctuation of all clusters} shows that ${\phi_{p,\Omega}^{\f,0}}\left(\sS_2\right)\ge 1-N^{-1}$ and thus
     ${\phi_{p,\Omega}^{\f,0}}\left(\sS_1(A)\cap\sS_2\right)\ge C_3>0$ for $N$ large enough. Hence we complete the proof of Proposition~{\ref{prop: anti-absolute continuity at high temperature}} by letting $\theta_0>0$ small enough.    
\end{proof}

Now we turn to the proofs of Lemmas~\ref{lem: anticoncentrate external field at T big0}, \ref{lem: anticoncentrate external field at T big1} and \ref{lem: anticoncentrate external field at T big2}.
\begin{proof}[Proof of Lemma~\ref{lem: anticoncentrate external field at T big0}]
    For any configuration $\omega\in \cF$, we obtain that \begin{align*}
        &\E\left(\sum_{\cC\in\mathfrak{C}}\frac{\eps^2h_{\cC}^2}{2T^2}-\frac{\eps^2|\Omega|}{2T^2}\right)^4\\=~&\E\sum_{\cC\in\fC}\left(\frac{\eps^2(h_{\cC}^2-|\cC|)}{2T^2}\right)^4+\E\sum_{\cC_1\neq\cC_2\in\fC}\left(\frac{\eps^2(h_{\cC_1}^2-|\cC_1|)}{2T^2}\right)^2\cdot\left(\frac{\eps^2(h_{\cC_2}^2-|\cC_2|)}{2T^2}\right)^2\\=~&\sum_{\cC\in\fC}\frac{C_1\eps^8|\cC|^4}{16T^8}+\sum_{\cC_1\neq\cC_2\in\fC}\frac{\eps^8|\cC_1|^2|\cC_2|^2}{4T^8}\le C_2\eps^8N^4.
    \end{align*} Summing over all configurations over $\cF$ shows that \begin{equation*}
\E~\Bigg\langle\left(\sum_{\cC\in\mathfrak{C}}\frac{\eps^2h_{\cC}^2}{2T^2}-\frac{\eps^2|\Omega|}{2T^2}\right)^4\Bigg|\cF\Bigg\rangle\le C_2\eps^8N^4,
    \end{equation*} and the desired result comes from Markov inequality.
\end{proof}
\begin{proof}[Proof of Lemma~\ref{lem: anticoncentrate external field at T big1}]
    Let $\cB\subset \mbb R^{\Omega}\otimes \{0,1\}^{\mathrm{E}(\Omega)}$ denote the event that $\sum_{\cC\in\mathfrak{C}}f(\frac{\eps h_{\cC}}{T})-\frac{\eps^2h_{\cC}^2}{2T^2}\ge -C_1\eps^4N^2$ where $C_1>0$ is some constant to be determined.
    Since $f(x)\ge \frac{x^2-x^4}{2}$, we obtain that for any $\omega\in\mcc F$, 
    \begin{align}
\P\otimes\phi_{p,\Omega}^{\f,0}(\cB ^c\mid\omega)&\le \P\left(\sum_{\cC\in\mathfrak{C}}\frac{\eps^4h_{\cC}^4}{2T^4}>C_1\eps^4N^2\right)=\P_X\left(\sum_{\cC\in\mathfrak{C}}\frac{\eps^4|\cC|^2X_{\cC}^4}{2T^4}>C_1\eps^4N^2\right)\label{eq: split of cB1 into cB2 and error term}
    \end{align} where $\{X_{\cC}\}_{\cC\in\mathfrak{C}}$ are independent Gaussian variable with mean $0$ and variance $1$.  To compute the right-hand side of \eqref{eq: split of cB1 into cB2 and error term}, we compute its first moment and variance. Since $\E_X X_{\cC}^4=3$ and recall $\omega\in \mcc F$, we get that \begin{equation}\label{eq: first moment bound on error term}
        \E_X\sum_{\cC\in\mathfrak{C}}\frac{\eps^4|\cC|^2X_{\cC}^4}{2T^4}=\sum_{\cC\in\mathfrak{C}}\frac{3\eps^4|\cC|^2}{2T^4}\le C_2\eps^4N^2.
    \end{equation} Moreover, we have \begin{equation}\label{eq: variance bound on error term}
        \var_X\left(\sum_{\cC\in\mathfrak{C}}\frac{\eps^4|\cC|^2X_{\cC}^4}{2T^4}\right)=\sum_{\cC\in\mathfrak{C}}\E_X
        ~\frac{\eps^8|\cC|^4(X_{\cC}^4-3)^2}{4T^8}\le C_3\eps^8N^2.
    \end{equation} Combining \eqref{eq: first moment bound on error term}, \eqref{eq: variance bound on error term} with \eqref{eq: split of cB1 into cB2 and error term} and letting $C_1>2C_2$, we get that \begin{align*}
        \P\otimes\phi_{p,\Omega}^{\f,0}(\cB^c\mid\omega)\le \frac{C_4}{N^2}.
    \end{align*}
Integrating over $\omega\in\mcc F$ and applying Lemma~\ref{lem: LDP for high temp} shows that $\P\otimes\phi_{p,\Omega}^{\f,0}(\cB^c)\le C_5\exp\left(-C_5^{-1}N^{0.05}\right)+\frac{C_4}{N^2}.$ The desired result on $\P(\cH_{\mathrm{hig}}^1(C_1))$ thus comes from the Markov inequality and letting $N$ big enough. 
\end{proof}
\begin{proof}[Proof of Lemma~\ref{lem: anticoncentrate external field at T big2}]
    Let $c_1>0$ be a constant to be defined later. By Lemma~\ref{lem: LDP for high temp}, we  get that \begin{equation}
        \fkhof(\cF)\ge 1-\exp(-CN^{0.05})\label{eq: cF bound in high temp}
    \end{equation} for some constant $C>0.$ To lower-bound $\P(\cH_{\mathrm{hig}}^2(c))$, we proceed by lower-bounding the first moment of $\var_{\phi}\left(\sum_{\cC\in\mathfrak{C}}\frac{\eps^2h_{\cC}^2}{2T^2}\mid\cF\right)$ and also upper-bounding its second moment.  We first expand the variance into \begin{align*}
        &\var_{\phi}\left(\sum_{\cC\in\mathfrak{C}}\frac{\eps^2h_{\cC}^2}{2T^2}\Bigg|\cF\right)=\Bigg\langle \left(\sum_{\cC\in\mathfrak{C}}\frac{\eps^2h_{\cC}^2}{2T^2}\right)^2\Bigg|\cF\Bigg\rangle-\Bigg\langle \left(\sum_{\cC\in\mathfrak{C}}\frac{\eps^2h_{\cC}^2}{2T^2}\right)\cdot \left(\sum_{\tilde{\cC}\in\tilde{\mathfrak{C}}}\frac{\eps^2h_{\tilde{\cC}}^2}{2T^2}\right)\Bigg|\cF\Bigg\rangle
    \end{align*} where ${\mathfrak{C}},\tilde{\mathfrak{C}}$ are the collections of all clusters of two independently sampled configurations according to $\fkhof(\cdot\mid\cF)$. Thus, we compute \begin{align}
    &\E~\var_{\phi}\left(\sum_{\cC\in\mathfrak{C}}\frac{\eps^2h_{\cC}^2}{2T^2}\Bigg|\cF\right)
    =\E~\Bigg\langle \left(\sum_{\cC\in\mathfrak{C}}\frac{\eps^2h_{\cC}^2}{2T^2}\right)^2\Bigg|\cF\Bigg\rangle-\E\Bigg\langle \left(\sum_{\cC\in\mathfrak{C}}\frac{\eps^2h_{\cC}^2}{2T^2}\right)\cdot \left(\sum_{\tilde{\cC}\in\tilde{\mathfrak{C}}}\frac{\eps^2h_{\tilde{\cC}}^2}{2T^2}\right)\Bigg|\cF\Bigg\rangle\nonumber\\=~&\Bigg\langle\left(\sum_{\cC\in\mathfrak{C}}\frac{\eps^2|\cC|}{2T^2}\right)^2+\sum_{\cC\in\mathfrak{C}}\frac{\eps^4|\cC|^2}{2T^4}\Bigg|\cF\Bigg\rangle-\left(\Bigg\langle\sum_{\cC\in\mathfrak{C}}\frac{\eps^2|\cC|}{2T^2}\Bigg|\cF\Bigg\rangle\right)^2-\Bigg\langle\sum_{\cC\in\mathfrak{C},\tilde{\cC}\in\tilde{\mathfrak{C}}}\frac{\eps^4}{2T^4}|\cC\cap\tilde{\cC}|^2\Bigg|\cF\Bigg\rangle\nonumber
    \end{align} where the last equation comes from the fact that $\E h_{A}^2=|A|$ and $\E h_A^2h_B^2=|A|\cdot|B|+2|A\cap B|^2$. Combined with Cauchy inequality, it yields that \begin{align}
\E~\var_{\phi}\left(\sum_{\cC\in\mathfrak{C}}\frac{\eps^2h_{\cC}^2}{2T^2}\Bigg|\cF\right)&\ge C_1\eps^4\left\langle\sum_{\cC\in\mathfrak{C}}|\cC|^2-\sum_{\cC\in\mathfrak{C},\tilde{\cC}\in\tilde{\mathfrak{C}}}|\cC\cap\tilde\cC|^2\Big|\cF\right\rangle\nonumber\\&\ge C_1\eps^4\left\langle\left[\sum_{\cC\in\mathfrak{C}}|\cC|^2-\sum_{\cC\in\mathfrak{C},\tilde{\cC}\in\tilde{\mathfrak{C}}}|\cC\cap\tilde\cC|^2\right]\cdot \1_{\cF}\right\rangle\label{eq: variance first moment bound high 0}
    \end{align}
    Note that for $\omega\in\cF^{c}$, we have the bound $\sum_{\cC\in\mathfrak{C}}|\cC|^2-\sum_{\cC\in\mathfrak{C},\tilde{\cC}\in\tilde{\mathfrak{C}}}|\cC\cap\tilde\cC|^2\le \sum_{\cC\in\mathfrak{C}}|\cC|^2\le (2N)^4.$ Combined with \eqref{eq: variance first moment bound high 0}, it yields that \begin{equation}
        \E~\var_{\phi}\left(\sum_{\cC\in\mathfrak{C}}\frac{\eps^2h_{\cC}^2}{2T^2}\Bigg|\cF\right)\ge  C_1\eps^4\left\langle\sum_{\cC\in\mathfrak{C}}|\cC|^2-\sum_{\cC\in\mathfrak{C},\tilde{\cC}\in\tilde{\mathfrak{C}}}|\cC\cap\tilde\cC|^2\right\rangle-C_1\eps^4(2N)^4\fkhof(\cF^c).\label{eq: variance first moment bound high 1}
    \end{equation} Plugging \eqref{eq: cF bound in high temp} and  Lemma~\ref{lem: high temperature all cluster deviation bound} into \eqref{eq: variance first moment bound high 1}, we get that \begin{equation}
        \E~\var_{\phi}\left(\sum_{\cC\in\mathfrak{C}}\frac{\eps^2h_{\cC}^2}{2T^2}\Bigg|\cF\right)\ge C_2\eps^4N^2.\label{eq: variance first moment bound high}
    \end{equation}
    Next, we compute the upper bound of the second moment. We upper-bound the variance by $$\var_{\phi}\left(\sum_{\cC\in\mathfrak{C}}\frac{\eps^2h_{\cC}^2}{2T^2}\Bigg|\cF\right)\le\Bigg\langle \left(\sum_{\cC\in\mathfrak{C}}\frac{\eps^2h_{\cC}^2}{2T^2}- \frac{\eps^2|\Omega|}{2T^2}\right)^2\Bigg|\cF\Bigg\rangle.$$ Applying Cauchy inequality, we obtain that \begin{align}
    &\E~\left(\var_{\phi}\left(\sum_{\cC\in\mathfrak{C}}\frac{\eps^2h_{\cC}^2}{2T^2}\Big|\cF\right)\right)^2\le\Bigg\langle \E~\left(\sum_{\cC\in\mathfrak{C}}\frac{\eps^2h_{\cC}^2}{2T^2}- \frac{\eps^2|\Omega|}{2T^2}\right)^4\Big|\cF\Bigg\rangle \nonumber\\=~& \Bigg\langle\sum_{\cC\in\mathfrak{C}}\E~\left(\frac{\eps^2(h_{\cC}^2-|\cC|)}{2T^2}\right)^4+\sum_{\substack{\cC_1,\cC_2\in\mathfrak{C},\\\cC_1\neq\cC_2}}\E~\left(\frac{\eps^2h_{\cC_1}^2}{2T^2}- \frac{\eps^2|\cC_1|}{2T^2}\right)^2\cdot\left(\frac{\eps^2h_{\cC_1}^2}{2T^2}- \frac{\eps^2|\cC_2|}{2T^2}\right)^2\Bigg|\cF\Bigg\rangle\nonumber
    \end{align}
    Note that $\E~(h_{A}^{2}-|A|)^4=60|A|^4$ and $\E~(h_{A}^{2}-|A|)^2=2|A|^2$. Thus, we get that \begin{align}
        \E~\left(\var_{\phi}\left(\sum_{\cC\in\mathfrak{C}}\frac{\eps^2h_{\cC}^2}{2T^2}\right)\Bigg|\cF\right)^2&\le C_3\eps^8\Bigg\langle\left(\sum_{\cC\in\mathfrak{C}}|\cC|^2\right)^2\Bigg|\cF\Bigg\rangle\label{eq: variance second moment bound high 0}
    \end{align}Recall the definition of $\cF$ in Definition~\ref{def: anticoncentrate external field at T big0} , we get for any $\omega\in\cF$ that $\sum_{\cC\in\mathfrak{C}}|\cC|^2\le \sum_{\cC\in\mathfrak{C}}|\cC|^4\\\le cN^2.$ Combined with \eqref{eq: variance second moment bound high 0}, it yields that \begin{equation}\label{eq: variance second moment bound high}
\E~\left(\var_{\phi}\left(\sum_{\cC\in\mathfrak{C}}\frac{\eps^2h_{\cC}^2}{2T^2}\Bigg|\cF\right)\right)^2\le C_4\eps^8N^4.
    \end{equation}
    Combining \eqref{eq: variance first moment bound high}, \eqref{eq: variance second moment bound high} and Paley-Zygmund inequality, we obtain that \begin{align}
    &\P\left(\var_{\phi}\left(\sum_{\cC\in\mathfrak{C}}\frac{\eps^2h_{\cC}^2}{2T^2}\Bigg|\cF\right)\ge \frac{1}{2}\E~\var_{\phi}\left(\sum_{\cC\in\mathfrak{C}}\frac{\eps^2h_{\cC}^2}{2T^2}\Bigg|\cF\right)\right)\nonumber\ge~\frac{\frac{1}{4}\left[\E~\var_{\phi}\left(\sum_{\cC\in\mathfrak{C}}\frac{\eps^2h_{\cC}^2}{2T^2}\Bigg|\cF\right)\right]^2}{\E~\left(\var_{\phi}\left(\sum_{\cC\in\mathfrak{C}}\frac{\eps^2h_{\cC}^2}{2T^2}\Bigg|\cF\right)\right)^2}\ge\frac{C_2^2}{4C_4}.\nonumber
    \end{align}The desired comes from \eqref{eq: variance first moment bound high} and the fact that $\theta N^{-\frac{1}{2}}\le \eps\le 2\theta N^{-\frac{1}{2}}$.
\end{proof}

\subsection{The low-temperature case}\label{sec: anti concentration low}

In this section, we consider the low-temperature case and fix $p>p_c$.
The main goal of this subsection is to prove that the good external field in Definition~\ref{def: good external field for product measure low decomposition} has a positive $\mbb P$ probability not depending on its size. 
The strategy we use in  the low-temperature regime is similar to that in the high-temperature regime. The difference is that in the high-temperature case, all the clusters are small such that none of them has a typical fluctuation, so they will contribute together and lead to the fluctuation. However, in the low-temperature case, the maximal cluster itself can lead to a big fluctuation. We start with some definitions as in Definitions~\ref{def: anticoncentrate external field at T big0} and \ref{def: anticoncentrate external field at T big}, which was hinted in Definition \ref{def: good external field for product measure low decomposition}. 
For the sake of convenience, we divide the definition of a good external field into the following three sets.

\begin{defi}\label{def: anticoncentrate external field at T small0}

    We use the notation $\cH_{\mathrm{low}}^0$ to denote the set of external field such that \begin{align}
        &\Big\langle\exp\left(\frac{\eps|h_{\cC_*}|}{T}\right)\Big\rangle\le \exp(\eps^{-1}N^{-1})~\text{and}~\label{eq: good external field for partition function at small T}\\ &\Big\langle\exp\left(\frac{\eps|h_{\cC_*}|}{T}\right)\cdot\prod_{\cC\in\fC\setminus\{\cC_*\}}\cosh\left(\frac{\eps h_{\cC}}{T}\right)\Big\rangle\le \exp(\eps^{-1}N^{-1}).\label{eq: good external field for partition function at small T2}
    \end{align}
    Recall that $f(x)=\ln\cosh(x)$ and $\mcc C_*$ is the boundary cluster. We use the notation $\cH_{\mathrm{low}}^1$ to denote the set of external field such that 
    \begin{equation}\label{eq: fluctuation of small clusters}
        \fklow\left(\Bigg|\sum_{\cC\in\mathfrak{C}\setminus\{\cC_*\}}f\left(\frac{\eps h_{\cC}}{T}\right)-\frac{\eps^2(|\Omega|-|\cC_*|)}{2T^2}~\Bigg|\le N^{-0.1}\right)\ge 1-N^{-0.1}.
    \end{equation}   
    For any constant $c>0$, let $\cH_{\mathrm{low}}^2(c)$ denote the set of external field such that $\var_{\phi}\big(\eps h_{\cC_*}\big)\ge c$ where $\var_{\phi}$ denotes the variance operator under $\fklow$.
\end{defi}
\begin{lem}\label{lem: anticoncentrate external field at T small0}
    There exists a constant $c>0$ depending only on $p$ such that for any disorder strength $\eps>0$ and domain $\lamn\subset\Omega\subset\Lambda_{2N}$, we have \begin{equation*}
        \P(\cH^0_{\mathrm{low}})\ge 1 - c\exp\left(-\eps^{-1}N^{-1}+c\eps^2N^2\right).
    \end{equation*}
\end{lem}

\begin{proof}
    For any configuration $\omega$, let $X$ denote a standard Gaussian variable and we have
    \begin{align*}
        &\E~\exp\left(\frac{\eps |h_{\cC_*}|}{T}\right)=\E_X~\exp\left(\frac{\eps\sqrt{|\cC_*|}|X|}{T}\right)\le \exp\left( C_1\eps^2|\cC_*|\right)\mbox{  and  }\\
        &\E~\cosh\left(\frac{\eps h_{\cC}}{T}\right) =\E_X~\cosh\left(\frac{\eps\sqrt{|\cC|}X}{T}\right)\le \exp( C_1\eps^2|\cC|)
    \end{align*} for some absolute constant $C_1>0.$ Summing over all configurations shows that \begin{align*}
        \E~\Bigg\langle\exp\left(\frac{\eps|h_{\cC_*}|}{T}\right)\Bigg\rangle&\leq
        \E~\Bigg\langle\exp\left(\frac{\eps|h_{\cC_*}|}{T}\right)\cdot\prod_{\cC\in\fC\setminus\{\cC_*\}}\cosh\left(\frac{\eps h_{\cC}}{T}\right)\Bigg\rangle\\
        &\le \Big\langle\exp( C_1\eps^2|\cC_*|)\cdot\prod_{\cC\in\fC\setminus\{\cC_*\}}\exp(C_1\eps^2|\cC|)\Big\rangle \le \exp(4C_1\eps^2N^2).
    \end{align*} The desired results come from an application of the Markov inequality.
\end{proof}
\begin{lem}\label{lem: anticoncentrate external field at T small1}
There exist constants $\theta_0>0$ and $c=c(p)>0$ such that the following holds. For any $0<\theta\le \theta_0$, any disorder strength $\theta N^{-1}\le \eps\le 2\theta N^{-1}$ and domain $\lamn\subset\Omega\subset\Lambda_{2N}$, we have \begin{equation*}
        \P(\cH_{\mathrm{low}}^1)\le {1-c^{-1}\exp(-cN^{c})}.
    \end{equation*}
\end{lem}
\begin{proof}
We choose $\theta_0=1$.
    Let $\cB^+$ denote the event that $\sum_{\cC\in\mathfrak{C}\setminus\{\cC_*\}}f(\frac{\eps h_{\cC}}{T})-\frac{\eps^2|\Omega|-|\cC_*|}{2T^2}\le N^{-0.1}$ and $\cB^-$ denote the event that $\sum_{\cC\in\mathfrak{C}\setminus\{\cC_*\}}f(\frac{\eps h_{\cC}}{T})-\frac{\eps^2|\Omega|-|\cC_*|}{2T^2}\ge -N^{-0.1}$. We first control the probability of $\cB^+$ under $\P\otimes\fklow$. For any configuration $\omega\in\sS_{\mathrm{low},*}:=\{\max_{\cC\in\mathfrak{C}\setminus\{\cC_*\}}|\cC|\le N^{0.1}\}$, 
    we have \begin{align*}
\P\otimes\fklow\big((\cB^+)^c\mid\omega\big)&=\P_X\left(\sum_{\cC\in\mathfrak{C}\setminus\{\cC_*\}}\left(f\left(\frac{\eps\sqrt{|\cC|}}{T}X_{\cC}\right)-\frac{\eps^2|\cC|}{2T^2}\right)>N^{-0.1}\right)\\&\le \P_X\left(\sum_{\cC\in\mathfrak{C}\setminus\{\cC_*\}}\frac{\eps^2|\cC|}{2T^2}(X_{\cC}^2-1)>N^{-0.1}\right)
    \end{align*} where the inequality comes from $f(x)\le \frac{x^2}{2}$. Combined with Lemma~\ref{lem: chi fourth power tail bound}, it yields that \begin{align*}
        \P\otimes\fklow\left((\cB^+)^c\mid\omega\right)\le c_1^{-1}\exp(-C_1N^{0.1}).
    \end{align*} Integrating over $\omega$ and applying Lemma~\ref{lem: low temperature cluster bound} 
    shows that $$\P\otimes\fklow\big((\cB^+)^c\big)\le C_1^{-1}\exp(-C_1N^{0.1})+\fklow(\sS_{\mathrm{low},*}^c)\le C_2^{-1}\exp(-C_2N^{C_2}).$$Next we control $\cB^-$. For any configuration $\omega\in\sS_{\mathrm{low},*}$, we have  \begin{align*}
&\P\otimes\fklow\big((\cB^-)^c\mid\omega\big)=\P_X\left(\sum_{\cC\in\mathfrak{C}\setminus\{\cC_*\}}f\left(\frac{\eps\sqrt{|\cC|}}{T}X_{\cC}\right)-\frac{\eps^2|\cC|}{2T^2}<-N^{-0.1}\right)\\\le~& \P_X\left(\sum_{\cC\in\mathfrak{C}\setminus\{\cC_*\}}\frac{\eps^2|\cC|}{2T^2}(X_{\cC}^2-1)<-\frac{N^{-0.1}}{2}\right)+\P_X\left(\sum_{\cC\in\mathfrak{C}\setminus\{\cC_*\}}\frac{\eps^4|\cC|^2}{2T^4}X_{\cC}^4>\frac{N^{-0.1}}{2}\right)
    \end{align*} where the inequality comes from $f(x)\ge \frac{x^2-x^4}{2}$. Combined with  Lemma~\ref{lem: chi fourth power tail bound}, it yields that \begin{align*}
        \P\otimes\fklow\big((\cB^-)^c\mid\omega\big)\le C_3^{-1}\exp(-C_3N^{0.1}).
    \end{align*} Integrating over $\omega$ and applying Lemma \ref{lem: low temperature cluster bound} shows that $$\P\otimes\fklow\big((\cB^-)^c\big)\le C_3^{-1}\exp(-C_3N^{0.1})+\fklow(\sS_{\mathrm{low},*}^c)\le C_4^{-1}\exp(-C_4N^{C_2}).$$
    The desired result on $\P(\cH_{\mathrm{low}}^1)$ thus comes from Markov inequality. 
\end{proof}
\begin{lem}\label{lem: anticoncentrate external field at T small2}
     For any $\theta>0$, there exist constants $c_1=c_1(p),c_2=c_2(p)>0$ such that for any disorder strength $\theta N^{-1}\le \eps\le 2\theta N^{-1}$, domain $\lamn\subset\Omega\subset\Lambda_{2N}$, we have \begin{equation*}
        \P(\cH_{\mathrm{low}}^2(c_1\theta^{ 2}))\ge c_2.
    \end{equation*}
\end{lem}

\begin{proof}[Proof of Lemma~\ref{lem: anticoncentrate external field at T small2}]
    Let $c>0$ be a constant to be defined later. To lower-bound $\P(\cH_{\mathrm{low}}^2(c_1\theta))$, we proceed by lower-bounding the first moment of $\var_{\phi}\big(\eps h_{\cC_*}\big)$ and also upper-bounding its second moment.  We first expand the variance into \begin{align*}
        &\var_{\phi}\big(\eps h_{\cC_*}\big)=\big\langle \eps^2 h_{\cC_*}^2\big\rangle-\big\langle \eps h_{\cC_*}\cdot \eps h_{\tilde{\cC_*}}\big\rangle
    \end{align*} where $\cC_*,\tilde\cC_*$ are the maximal clusters of two independently sampled configurations according to $\fklow$. Thus, we compute \begin{align}
    \E~\var_{\phi}\big(\eps h_{\cC_*}\big)&=\big\langle \E~\eps^2 h_{\cC_*}^2\big\rangle-\big\langle \E~\eps h_{\cC_*}\cdot \eps h_{\tilde{\cC_*}}\big\rangle=\big\langle \eps^2 |\cC_*|\big\rangle-\big\langle \eps^2 |\cC_*\cap\tilde{\cC_*}|\big\rangle\ge C_1\eps^2N^2.\label{eq: variance first moment bound}
    \end{align} where the last inequality comes from Lemma~\ref{lem: low temperature maximal cluster deviation bound}.
    Next, we compute the upper bound of the second moment
    \begin{align}
        \E~\Big[\var_{\phi}\big(\eps h_{\cC_*}\big)\Big]^2&\le 4\E~\big\langle \eps^2 h_{\cC_*}^2\big\rangle^2\le 4\E~\big\langle \eps^4 h_{\cC_*}^4\big\rangle=12\big\langle \eps^4 |\cC_*|^2\big\rangle\le C_2\eps^4N^4.\label{eq: variance second moment bound}
    \end{align}
    Combining \eqref{eq: variance first moment bound}, \eqref{eq: variance second moment bound} and Paley-Zygmund inequality, we obtain that \begin{align}
        &\P\Big(\var_{\phi}\big(\eps h_{\cC_*}\big)\ge \frac{1}{2}\E~\var_{\phi}\big(\eps h_{\cC_*}\big)\Big)\nonumber\ge\frac{\frac{1}{4}\Big[\E~\var_{\phi}\big(\eps h_{\cC_*}\big)\Big]^2}{\E~\Big[\var_{\phi}\big(\eps h_{\cC_*}\big)\Big]^2}\ge\frac{C_1^2}{4C_2}.\nonumber
    \end{align}The desired comes since $\theta N^{-1}\le \eps\le 2\theta N^{-1}$.
\end{proof}
Combining Lemmas~\ref{lem: anticoncentrate external field at T small0}, \ref{lem: anticoncentrate external field at T small2} and \ref{lem: dis concentration using variance},  a similar version of Propositions~\ref{prop: anti-absolute continuity at critical temperature} and \ref{prop: anti-absolute continuity at high temperature} could be proved, and the proof is the same as that in Section~\ref{sec: coarse graining at low temperature}. Thus, we just state the Proposition here and omit further details:

\begin{prop}\label{prop: anti-absolute continuity at low temperature}
    Fix $T<T_c(2)$ and thus that $p>p_c$.
    %Then there exists a constant $\theta_0>0$ only relies on $p$ such that for any $\theta\in(0,\theta_0)$, 
   For any $\theta>0$, there exists a constant $c=c(\theta,p)>0$ such that $ \E\TV{\fkhlow}{\fklow}\ge c$ for any disorder strength $\theta N^{-1}\le \eps\le 2\theta N^{-1}$.
\end{prop}

% \begin{rmk}
% One could easily remove the restriction of $\theta\in(0,\theta_0)$ via a coarse graining method, see  
% Section \ref{sec: coarse graining at low temperature} for a detailed coarse graining framework (but with a different purpose). 
% \end{rmk}

\subsection{Criticality of $\alpha(T)$}\label{sec: upper bound at near critical disorder}
In this subsection, we will show that in order to make $\E\TV{\phi_{p,\Omega}^{\gamma,0}}{\phi_{p,\Omega}^{\gamma,\eps h}}$ converging to $1$, a necessary condition is that $\eps \gg N^{-\alpha(T)}$. More precisely, we will prove:
\begin{prop}\label{prop: upper bound near critical disorder}
    Fix $d=2$ and $T>0$. For any constant $\theta>0$, there exists a constant $c=c(\theta)>0$ such that $\E\TV{\phi_{p,\lamn}^{\gamma,0}}{\phi_{p,\lamn}^{\gamma,\eps h}}\le 1-c$ for any disorder strength $\eps\le \theta N^{-\alpha(T)}$ and boundary condition $\gamma$.
\end{prop} We remark here that Proposition~\ref{prop: upper bound near critical disorder} is not a straightforward consequence of Section~\ref{sec: TV exponential decay to 0}, because in Section~\ref{sec: TV exponential decay to 0}, we focus on the region $\eps \ll N^{-\alpha(T)}$. Nevertheless, the proof idea of Proposition~\ref{prop: upper bound near critical disorder} is similar to that in Section~\ref{sec: TV exponential decay to 0}. In order to overcome the difference that $\eps N^{\alpha(T)}$ might be big, we need to modify the definitions in Section~\ref{sec: TV exponential decay to 0}. 
\begin{defi}\label{def: good external field for partition function near}
    Let $\gamma\in \{-1,1\}^{\pari\lamn}$ be a boundary condition on $\pari\lamn$. Let $\cH^i_{\diamond}(c)\subset\R^{\lamn}~(i=1,2,3)$ denote the collection of the  external field $h$ such that the followings hold \begin{enumerate}
        \item \begin{equation}\label{eq: good external field for partition function near}
         c^{-1}\le \frac{\cZ^{\gamma,\eps h}_{\phi,p,\Omega}}{\cZ^{\gamma,0}_{\phi,p,\Omega}}\cdot \frac{1}{\exp(\eps^2|\Omega|/2T^2)}\le c.
    \end{equation} 
    \item \begin{equation}\label{eq:good external field and clusters near}
    \fk\gamma\left(\sum_{\cC\in\mathfrak{C}} f\left({\frac{\eps h_{\cC}}{T}}\right)-\frac{2\eps^2N^2}{T^2}\le c\right) \geq 1-c^{-1}\,.
    \end{equation}
    \item \begin{equation}\label{eq:perfect external field and clusters near}
        \fk\gamma\left(\sum_{\cC\in\fC} f\left({\frac{\eps h_{\cC}}{T}}\right)-\frac{2\eps^2N^2}{T^2}\ge -c\right) \geq 1-c^{-1}\,.
    \end{equation}
    \end{enumerate}
\end{defi}
Following the idea in Section~\ref{sec: Concentration for the Radon Nikodym derivative}, it is easy to prove that there exists a constant $c>0$ depending on $\theta$ and $T$ such that for $i=2,3$ we have $$\P\big(\cH^i_\diamond(c)\big)\ge 0.9\,.$$For $\cH^1_\diamond$, it is a little more complicated since \eqref{eq: good external field for partition function} no longer provides a lower bound on the ratio of the partition function when $\eps N^{\beta(T)}$ is bigger than $1$. The idea is that although we cannot directly control the ratio of the partition function, we can divide $\lamn$ into small areas and show that with high $\P$-probability adding an external field on each small area perturbs the ratio by at most a factor of $2$. We summarize as the following Lemma.
\begin{lem}\cite[Lemma 3.4]{DHX23}\label{lem: small perturbation for partition function new}
    Let $\Omega$ be an $M$-box such that $\Omega\subset\Lambda_{N}$, $\gamma\in\{-1,1\}^{\pari\lamn}$ be the boundary condition and $g\in \mathbb{R}^{\lamn\setminus\Omega}$ be the external field on $\lamn\setminus\Omega$. Let $\cH^{\gamma,g}_{\emptyset}\subset\R^{\Omega}$ denote the collection of the external field $h$ such that \begin{equation}\label{eq: good external field for partition function new}
         \Bigg|\Bigg\langle ~ \prod_{v\in\Omega}\left[1+ \sigma_v\tanh\left(\frac{\eps h_v}{T}\right)\right]~\Bigg\rangle^{\gamma,g}_{p,\lamn}-1\Bigg|\le \sqrt{\eps M^{\beta(T)}}.
    \end{equation} 
There exist constants $c_1,c_2>0$ depending only on $T$ such that $$\P(\cH^{\gamma,g}_{\emptyset})\ge 1-c_1\exp\left(-c_1^{-1}\sqrt{\eps^{-1}M^{-\beta(T)}}\right)$$ for all disorder strength {$\epsilon\leq c_2M^{-\beta(T)}$}.
\end{lem}
For any disorder strength such that $\eps\le \theta N^{-\alpha(T)}$, let $M=\lfloor C_1\eps^{-\frac{1}{\alpha(T)}}\rfloor$ with $C_1>0$ small enough. Let $\{B_i\}_{i\in I}$ be a partition of $\lamn$ by $M$-boxes, then we can conduct an induction by Lemma~\ref{lem: small perturbation for partition function new} and obtain that $$\P(\cH_\diamond^1(2^{|I|}))\ge1-c_1|I|\exp\left(-c_1^{-1}\sqrt{\eps^{-1} M^{-\beta(T)}}\right).$$ Note that $|I|=\frac{N^2}{M^2}\le M^{-2}(\theta\eps^{-1})^{\frac{2}{\alpha(T)}}$, letting $C_1$ small enough shows that $\P\left({\cH_\diamond^1(2^{|I|})}\right)\ge 0.9$. The rest of the proof is similar to that in Section \ref{sec: TV exponential decay to 0} and thus we omit further details.

\section{Singularity in strong-disorder regime}\label{sec:singularity}
This section is devoted to the proof of singularity results in Theorem~\ref{thm-all-temperature}. It combines a coarse-graining argument with some near-critical results that have be shown in Section~\ref{sec: near-critical behavior}. The proof idea is quite similar to Section~\ref{sec:RFIM-singularity}, however, careful analysis of the random cluster configuration is involved since the RFFKIM lacks the nice domain Markov property (DMP) possessed by the RFIM.

\subsection{Critical-temperature and high-temperature cases}\label{sec: crit+high coarse-graining}

Critical-temperature and high-temperature regimes can be analyzed jointly. The core intuition is that in both cases, random clusters remain bounded in size – equivalently, dual paths are abundant.
%The critical-temperature and high-temperature cases could be solved together, and the rough intuition is that in both cases the random clusters are not likely to grow too large, or equivalently speaking, there will be plenty of dual paths.

More precisely, let us consider the boxes inside $\lamn$ that are surrounded by closed edges.
Let $M$ be an integer with order  $O(\eps^{-\frac{1}{\alpha(T)}})$  whose value will be defined later and without loss of generality, we assume $N$ is divisible by $2M$. Let $B_i~(i=1,\cdots,n)$ be a partition of $\lamn$ by $2M$-boxes and let $u_i$ denote the center of $B_i$ (i.e. $B_i=\Lambda_{2M}(u_i)$ {and $n=\Big(\frac{N}{2M}\Big)^2$}). 
\begin{defi}\label{def: outmost close region}
    For any configuration $\omega\in \{0,1\}^{\edge}$, we say $\calO=(\Omega_1,\Omega_2,\cdots,\Omega_n)$ is the outmost close region of $\omega$ if \begin{enumerate}[(i)]
        \item $\Omega_i\subset B_i$ and either $\Lambda_M(u_i)\subset\Omega_i$ or $\Omega_i=\emptyset$ $(i=1,\cdots,n)$.
        \item $\forall e=\{x,y\}\in\edge$, $x\in\Omega_i,y\notin\Omega_i\,$, $e$ is close in $\omega$.\label{item: condition 2}
        \item There does not exists a subset $\Lambda_M(u_i)\subset\Omega_i'\subset \Lambda_{2M}(u_i)$ satisfying \eqref{item: condition 2} such that $\Omega_i\subsetneqq\Omega_i'$.
    \end{enumerate} See Figure~\ref{fig: outmost close region} for  illustration.
    We use the notation $\calO=\mathrm{Out}(\omega)$ to denote the event that $\calO=(\Omega_1,\Omega_2,\cdots,\Omega_n)$ is the outmost close region of $\omega$. Let $\eta(\calO)$ denote the number of nonempty regions $\Omega_i$, i.e. $\eta(\calO)=\sum_{i=1}^n\1_{\Omega_i\neq\emptyset}$ and $\Omega(\calO)=\bigcup_{i=1}^n\Omega_i$ denote the union of the regions.
\end{defi}

\begin{figure}[htbp]
\vspace{-10pt}
\centering
\includegraphics[scale=0.12]{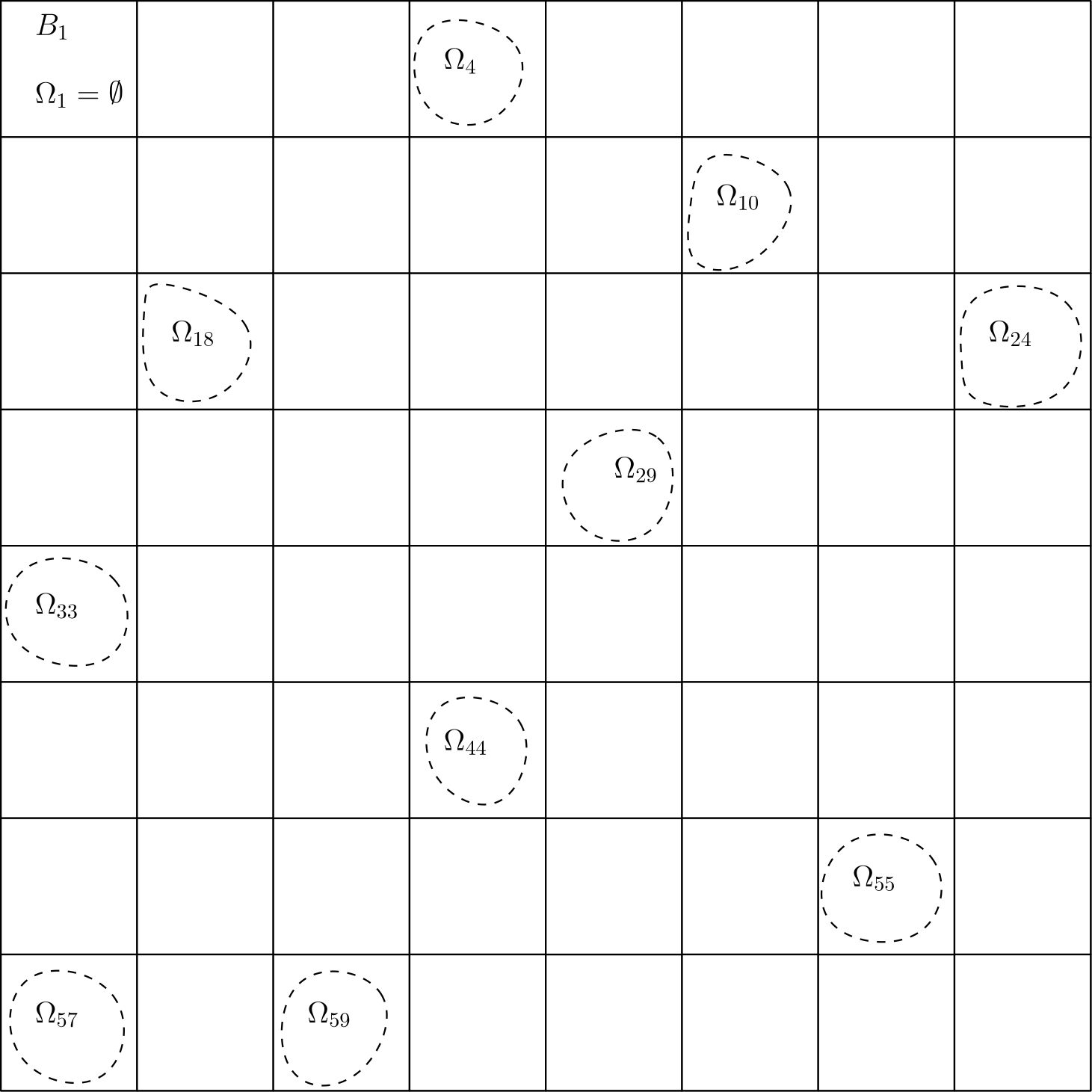} % 插入图片
\vspace{-10pt} % 缩小图片与文字之间的距离
\caption{An illustration of the outmost close region. Dotted lines: the dual circuits.} % 在图片下方添加说明
\label{fig: outmost close region}
\vspace{-5pt} % 进一步缩小距离（可选）
\end{figure}

For any $\calO=(\Omega_1,\Omega_2,\cdots,\Omega_n)$, let $\fkh\gamma(\calO)=\sum_{\omega:\calO=\mathrm{Out}(\omega)}\fkh\gamma(\omega)$ be the marginal distribution on the outmost close region. We emphasize that the event $\mcc O=\mathrm{Out}(\omega)$ is measurable under $\omega_{|E(\lamn)\setminus E(\Omega(\calO))}$. We further define $\Phi_{p,\calO}^h$ to be the product measure of $\phi^{\f,h}_{p,\Omega_i}~(i=1,2,\cdots,n)$. More precisely, let $\omega^\calO\in\{0,1\}^{E(\Omega(\calO))}$ be an edge configuration and $\omega_i^\calO=\omega^\calO_{|\Omega_i},$ then we define
$\Phi_{p,\calO}^h(\omega^\calO)=\prod_{i=1}^n\phi_{p,\Omega_i}^{\f,h}(\omega_i^\calO).$ By applying DMP, we obtain that for any configuration $\omega$ and $\calO=\mathrm{Out}(\omega)$, \begin{align}\label{eq: decomposition to xi and small region}
\fk\gamma(\omega)&=\fk\gamma(\calO)\cdot\Phi_{p,\calO}^0(\omega^\calO),~~~~
\fkh\gamma(\omega)=\fkh\gamma(\calO)\cdot\Phi_{p,\calO}^{\eps h}(\omega^\calO).
\end{align} 
We first claim that in order to show $\TV{\fkh\gamma}{\fk\gamma}$ is close to $1$, it suffices to show that $\TV{\Phi_{p,\calO}^{\eps h}}{\Phi_{p,\calO}^0}$ is close to $1$ with {high probability under the measure $\fk\gamma$ for overwhelming proportion of
external field $\eps h$. To see this,  for any fixed external field $\eps h,$
let $\mathscr{O}=\mathscr{O}(\eps h)$ denote the set of $\calO$ such that \begin{equation}
        \label{eq: good xi condition}\TV{\Phi^0_{p,\calO}}{\Phi^{\eps h}_{p,\calO}}\ge 1-c_1\exp(-c_1^{-1}n).
    \end{equation} Similar to \eqref{eq: ising singularity decouple} and \eqref{eq: ising singularity final}, we apply \eqref{eq: decomposition to xi and small region} to obtain  that \begin{align}
        &1-\TV{\fkh\gamma}{\fk\gamma}\,{=}\sum_{\mcc O}\sum_{\omega:\calO=\mathrm{Out}(\omega)}\min\left\{\fk\gamma(\calO)\cdot\Phi_{p,\calO}^0(\omega),\fkh\gamma(\calO)\cdot\Phi_{p,\calO}^{\eps h}(\omega)\right\}\nonumber\\
        \le& \sum_{\mcc O\in\mss O}\big(\fk\gamma(\calO)+\fkh\gamma(\calO)\big)\cdot\sum_{\omega:\calO=\mathrm{Out}(\omega)}\min\left\{\Phi_{p,\calO}^0(\omega),\Phi_{p,\calO}^{\eps h}(\omega)\right\}\nonumber+\nonumber\\&\sum_{\mcc O\in\mss O^c}\sum_{\omega:\calO=\mathrm{Out}(\omega)}\fk\gamma(\calO)\cdot\Phi_{p,\calO}^0(\omega)\nonumber\\=~&\sum_{\mcc O\in\mss O}\big(\fk\gamma(\calO)+\fkh\gamma(\calO)\big)\cdot\big(1-\TV{\Phi^0_{p,\calO}}{\Phi^{\eps h}_{p,\calO}}\big)+\fk\gamma(\mss O^c).\label{eq: total variation distance expansion}
    \end{align}

This naturally leads to the following definition:
\begin{defi}\label{def: good external field for product measure}
    We use the definition $\cH_{\diamond}(c)$ to denote the set of external field such that \begin{equation}\label{eq: good xi under good external field}
        \fk\gamma\Big(\big\{\calO : \TV{\Phi^0_{p,\calO}}{\Phi^{\eps h}_{p,\calO}}\ge 1-c\exp(-c^{-1}n)\big\}\Big)\ge 1-c\exp(-c^{-1}n).
    \end{equation}
\end{defi}
\begin{lem}\label{lem: good external field for product measure}
    Fix temperature $T\ge T_c$ and $\theta>0$,
    there exist constants $c_1,c_2>0$ only relies on $T,\theta$ such that for any disorder strength $\theta M^{-\alpha(T)}\le \eps\le 2\theta M^{-\alpha(T)}$, we have $$\P\big(\cH_{\diamond}(c_1)\big)\ge 1-c_2\exp\left(-c_2^{-1}\left(\frac{N}{2M}\right)^2\right).$$
\end{lem}
\begin{proof}
    Recall that $\{B_i\}_{1\le i \le n}$ is a partition of $\lamn$ by $2M$-boxes. Applying the RSW theory {(see \cite{DT20} and also e.g. \cite{DHN11, Tassion16, KT23} for other remarkable progress on the RSW theory)}, there exists a constant $C_1>0$ such that $\phi^{\xi,0}_{p,B_i}(\Omega_i\neq\emptyset)\ge C_1$ where $\Omega_i$ is the region satisfying Definition \ref{def: outmost close region} and $\xi$ is an arbitrary boundary condition on $\pari B_i$. Thus, we obtain by DMP and Lemma~\ref{lem: concentration inequality with finite exponential moment} that \begin{equation}\label{eq: good xi}
        \fk\gamma\left(\left\{\eta(\calO)\ge \frac{C_1n}{2}\right\}\right)\ge 1-C_2\exp({-}C_2^{-1}n).
    \end{equation}
    Let $\cA$ denote the event that $1-\TV{\Phi^0_{p,\calO}}{\Phi^{\eps h}_{p,\calO}}\le C_3\exp({-}C_3^{-1}n)$  with some $C_3>0$ to be determined later.
Note that $\Phi^0_{p,\calO}$ and $\Phi^{\eps h}_{p,\calO}$ are product measures of $\phi^{\f,0}_{p,\Omega_i}$ and $\phi^{\f,\eps h}_{p,\Omega_i}$ and Propositions~\ref{prop: anti-absolute continuity at critical temperature} and \ref{prop: anti-absolute continuity at high temperature} show that $\E\TV{\phi^{\f,0}_{p,\Omega_i}}{\phi^{\f,\eps h}_{p,\Omega_i}}\ge C_4>0$. Thus, applying Corollary~\ref{cor: product measure total variation distance lower bound} gives that for any $\mcc O$ such that $\eta(\calO)\ge \frac{C_1n}{2}$ \begin{equation}\label{eq: good event under good xi}
    \P\otimes\fk\gamma(\cA\mid \{\mcc O=\mathrm{Out}(\omega)\})\ge 1-C_5\exp({-}C_5^{-1}n).
\end{equation}Note that we also fix $C_3$ here to be the constant given by Corollary~\ref{cor: product measure total variation distance lower bound}.

    Integrating \eqref{eq: good event under good xi} over $\{\eta(\calO)\ge \frac{C_1n}{2}\}$ and combining with \eqref{eq: good xi}, we get that 
    \begin{equation}\label{eq: bad event probability}
        \P\otimes\fk\gamma(\cA^c)\le C_6\exp({-}C_6^{-1}n).
    \end{equation} Thus, the desired result comes from Markov inequality and letting $c_1>\max\{C_3,C_6\}$.
\end{proof}
With Lemma~\ref{lem: good external field for product measure}, we are ready to prove the second criterion of Theorem~\ref{thm-all-temperature}.
\begin{proof}[Proof of Theorem~\ref{thm-all-temperature} (high-temperature and critical-temperature)~part~two]
    Let $M$ be an integer such that $\theta M^{-\alpha(T)}\le \eps \le 2\theta M^{-\alpha(T)}$  and we can simply choose $\theta=1$.
    We first consider $h\in\cH_{\diamond}(c_1)$ where $c_1$ was defined in Lemma~\ref{lem: good external field for product measure}. 
    Combining \eqref{eq: total variation distance expansion} with \eqref{eq: good xi under good external field} and \eqref{eq: good xi condition}, we get that \begin{align}
        1-\TV{\fkh\gamma}{\fk\gamma}&\le \sum_{\mcc O\in\mss O}\big(\fk\gamma(\calO)+\fkh\gamma(\calO)\big)\cdot c_1\exp(-c_1^{-1}n)+c_1\exp(-c_1^{-1}n)\nonumber\\&\le 3c_1\exp(-c_1^{-1}n).\nonumber
    \end{align}The desired result comes from Lemma~\ref{lem: good external field for product measure}.
   \end{proof} 
    
\subsection{The low-temperature case}\label{sec: coarse graining at low temperature}
Now we consider the low-temperature case. Let $M<N$ be a fixed integer with order $O(\eps^{-1})$ to be defined later.  The main difference from Section \ref{sec: crit+high coarse-graining} is that \eqref{eq: good xi} does not hold since with high probability there does not exist a dual circuit in $\Lambda_{2M}(u_i)\setminus\Lambda_M(u_i)$. Thus, we need to consider an open circuit instead of a dual circuit, which provides a wired boundary condition instead of a free boundary condition. Further, we emphasize that the second equality in \eqref{eq: decomposition to xi and small region} fails since the external field outside $\Omega_i$ also influences the measure. Nevertheless, the influence of outside external field is manageable if the open circuit is connected to a unique large cluster.

We start with the construction of the desired unique large cluster.

\begin{defi}
\label{def:unique-large-cluster}
    Recall that $\{B_i=\Lambda_{2M}(u_i):i=1,2,\cdots,n\}$ is a partition of $\lamn$ by $2M$-boxes. Without loss of generality, we assume $M$ is an even number and $N$ is divisible by $2M$.  Define $ \Lambda:=\bigcup_{i=1}^n(\Lambda_{2M}(u_i)\setminus\Lambda_{M}(u_i))$ to be the union of the annuli in each $B_i$ and we point out that $\Lambda$ is connected.
    
    An edge configuration $\omega$ is called \textbf{well-connected} if there exists a unique connected cluster in $\omega_{|\Lambda}$ such that its diameter is not less than $N/2.$ And we will call this cluster the \textbf{main} cluster in $\omega$ for convenience. Besides, we write $\mcc W\subset\{0,1\}^{\mbb E(\Lambda_N)}$ for the set of well-connected configurations.
\end{defi}

\begin{rmk}
  We highlight that in the definition of a well-connected configuration $\omega$, the uniqueness of the main cluster is confined to $\omega_{\mid\Lambda}$. In other words, despite this restriction, there remains the possibility that certain small clusters within $\omega_{\mid\Lambda}$ may be connected to a large cluster in $\omega$ that differs from the main cluster. This clarification is solely intended to enhance readers' understanding and does not impact the proof in any way.
\end{rmk}

Then we will explore the open circuits in each $B_i$ that is connected to the main cluster. The following definition comes from a similar idea as Definition~\ref{def: outmost close region}.
\begin{defi}\label{def: outmost open region}
For any well-connected configuration $\omega\in\mcc W$ and any  external field $h$, we say $\calO=(\Omega_1,\Omega_2,\cdots,\Omega_n,\tilde\omega)$ is the outmost good open region of $\omega$ if
   \begin{enumerate}[(i)] 

 \item  For all $i=1,2,\cdots,n,$
 either $\Omega_i=\emptyset$ or $\Lambda_{M/2}(u_i)\subset \Omega_i\subset\Lambda_M(u_i)$ such that
the circuit boundary of $\Omega_i$ is open and it is connected to the main cluster of $\omega.$
%$\pari\Omega_i$ is an open circuit in $\omega$ and it is connected to the main cluster of $\omega.$

 \label{item: condition 2'}

 \item If $\Omega_i\neq\emptyset$, then $\eps h$ is \textbf{$c$-good} with respect to $\Omega_i$ where the $c$-good external field will be defined in the Definition~\ref{def: good external field for product measure low decomposition} below.
 \label{item: condition 4}
\item There does not exists a subset $\Lambda_{M/2}(u_i)\subset\Omega'_i\subset \Lambda_{M}(u_i)$ satisfying {\eqref{item: condition 2'} } and {\eqref{item: condition 4} }such that $\Omega_i\subsetneqq\Omega'_i$.
 \item Let $\Gamma=\cup_{i=1}^n(B_i\setminus\Omega_i)$, then $\tilde\omega=\omega_{|\Gamma}$.

   \end{enumerate}
    See Figure~\ref{fig: outmost good open region} for illustration. Let $\cC$ be the main cluster of $\tilde\omega$ in $\Lambda$. Let $\cC'$ be the cluster that contains $\cC$ of $\tilde\omega$ in $\Gamma$. Note that $\cC\subset\cC'$. We slightly abuse the notation and call $\cC'$ the main cluster in $\tilde\omega$.
\end{defi}

\begin{figure}[htbp]
\vspace{-10pt}
\centering
\includegraphics[scale=0.1]{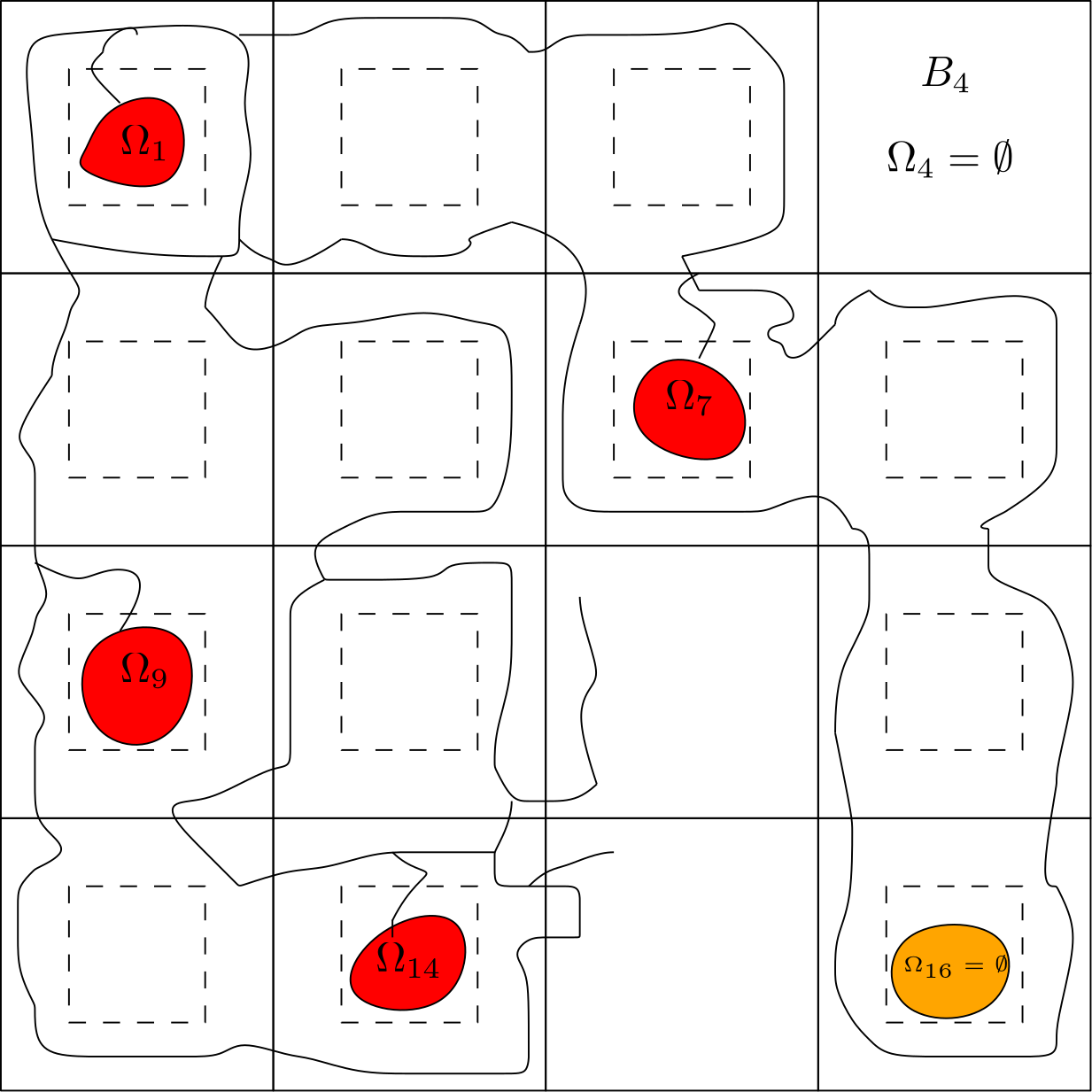} % 插入图片
\vspace{-10pt} % 缩小图片与文字之间的距离
\caption{An illustration of the outmost good open region. The dotted squares are $\Lambda_M(u_i)$ where $\Lambda_{2M}(u_i)=B_i$. We do not draw the box $\Lambda_{M/2}(u_i)$, but each red and orange region contains a box of such size. The red color denotes that $h$ is $c$-good with respect to the region, and  the orange color denotes that $h$ is not $c$-good with respect to the region.} % 在图片下方添加说明
\label{fig: outmost good open region}
\vspace{-5pt} % 进一步缩小距离（可选）
\end{figure}

 We use the notation $\calO=\mathrm{Out}(\omega,h)$ to denote that $\calO=(\Omega_1,\Omega_2,\cdots,\Omega_n,\tilde\omega)$ is the outmost good open configuration of $\omega$. Note that for $\calO=\mathrm{Out}(\omega,h)$, $\omega$ is well-connected if and only if $\tilde{\omega}$ is well-connected. Thus, we get that $\{\omega: \mathrm{Out}(\omega,h)=(\Omega_1,\Omega_2,\cdots,\Omega_n,\tilde\omega)\}=\tilde{\omega}\otimes\{0,1\}^{E(\cup_{i=1}^{n}\Omega_i)}$.
 Let $\eta(\calO)$ denote the number of nonempty regions $\Omega_i$. We define $\Phi_{p,\calO}^{\eps h}$ to be FK measure $\fkh\gamma(\cdot\mid\tilde\omega)$ conditioned on $\omega_{|\Gamma}=\tilde\omega$.  By calculations in \eqref{eq: total variation distance expansion}, it suffices to prove a lower bound for $\TV{\Phi^{\eps h}_{p,\calO}}{\Phi^0_{p,\calO}}$ as in \eqref{eq: good xi condition}. Unlike Section~\ref{sec: crit+high coarse-graining}, $\Phi^{\eps h}_{p,\calO}$ is not a product measure of measures supported on $\Omega_i$ because the external field outside $\Omega_i$ will influence the measure on $\Omega_i$. Therefore, we introduce the following definition of good external field.
\begin{defi}\label{def: good external field for product measure low decomposition} For any region $\Lambda_{M/2}\subset\Omega\subset\Lambda_{M}$, consider the FK-Ising measure with wired boundary condition on $\Omega$. Let  $\cC_*$ be the boundary cluster, we say the external field $\eps h$ is $c$-good with respect to $\Omega$ if the following holds.\begin{enumerate}[(i)]
        \item $\Big\langle\exp\left(\frac{\eps|h_{\cC_*}|}{T}\right)\Big\rangle_{p,\Omega}^{\w,0}\le \exp(\eps^{-1}M^{-1}).$\label{item: good external field 1}
        \item $\Big\langle\exp\left(\frac{\eps|h_{\cC_*}|}{T}\right)\cdot\prod_{\cC\in\fC\setminus\{\cC_*\}}\cosh\left(\frac{\eps h_{\cC}}{T}\right)\Big\rangle_{p,\Omega}^{\w,0}\le \exp(\eps^{-1}M^{-1}).$\label{item: good external field 2}
        \item $\phi_{p,\Omega}^{\w,0}\left(\sum_{\cC\in\mathfrak{C}\setminus\{\cC_*\}}f(\frac{\eps h_{\cC}}{T})-\frac{\eps^2|\Omega|-|\cC_*|}{2T^2}\le M^{-0.1}\right)\ge 1-M^{-0.1}.$\label{item: good external field 3}
        \item $\var_{\phi}\big(\eps h_{\cC_*}\big)\ge c$ where $\var_{\phi}$ denotes the variance operator under $\phi_{p,\Omega}^{\w,0}$.\label{item: good external field 4}
\end{enumerate}
\end{defi}
The following lemma provides a useful anticoncentration fact about the good external field with the proof postponed to the end of this section.

\begin{lem}\label{lem: anticoncentration for good external field}
    There exists a constant $\theta_0>0$ such that for any $0<\theta<\theta_0$ and $\theta M^{-1}\le \eps\le 2\theta M^{-1}$, we have the following. For any region $\Lambda_M\subset\Omega\subset\Lambda_{2M}$ and any $c$-good external field $\eps h$ with respect to $\Omega$, let $X^\pm=\sum_{\cC\in\fC\setminus\{\cC_*\}}f\left(\frac{\eps h_{\cC}}{T}\right)\pm\frac{\eps h_{\cC_*}}{T}$. Then we have  $$\langle X^\pm\rangle^{\w,0}_{p,\Omega}\le \ln\big(\langle \exp(X^\pm)\rangle^{\w,0}_{p,\Omega}\big)-c_1$$ for some $c_1>0$ depending only on $\theta$  and $c$.
\end{lem}

Then we claim that ${\eta(\calO)}$ is large for high $\P$-probability and also postpone the proof {to the end of this section.}
\begin{defi}\label{def: good external field for number of xi}
    We use the definition $\cH_{\Box}(c)$ to denote the set of external field such that \begin{equation}\label{eq: good xi under good external field low}
        \fk\gamma\left(\left\{\calO=\mathrm{Out}(\omega,h):{\eta(\calO)}\ge \frac{cN^2}{M^2}\right\}\right)\ge 1-c^{-1}\exp\left(-\frac{cN}{M}\right).
    \end{equation}
\end{defi}
\begin{lem}\label{lem: good external field for number of xi}
    Fix $T< T_C$ and $\theta<\theta_0$ where $\theta_0$ is defined in Lemma~\ref{lem: anticoncentration for good external field}. Then there exist constants $c_1,c_2>0$ such that for any disorder strength $\theta M^{-1}\le \eps\le 2\theta M^{-1}$, we have $\P\big(\cH_{\Box}(c_1)\big)\ge 1-c_2\exp\big(-c_2^{-1}\frac{N}{M}\big).$
\end{lem}
Now we are ready to give a similar bound as \eqref{eq: good xi condition}, as incorporated in the following Lemma~\ref{prop: exponential decay through coarse graining low}.
\begin{lem}\label{prop: exponential decay through coarse graining low}
    For any $\calO$ such that $\eta(\calO)\ge c_1n$, we have $$\TV{\Phi^{0}_{p,\calO}}{\Phi^{\eps h}_{p,\calO}}\ge 1-c_2\exp(-c_2^{-1}n)$$ where $c_2>0$ does not depend on $\mcc O$ or $\eps h$.
\end{lem}

\begin{proof} Fix any external field $\eps h$ and $\calO,$
    let $I$ denote the set of indices such that $\Omega_i\neq\emptyset$ (thus $\eps h$ is good with respect to $\Omega_i$). For any $\mathrm{Out}(\omega,h)=\calO$ and $i\in I$, let $\fC^i$ be the collection of clusters in $\omega_{\mid\Omega_i}$ and let $\cC_*^i$ be the boundary cluster in $\Omega_i$. Let $$X_i^\pm(\omega)=\sum_{\cC\in\fC^i\setminus\{\cC_*^i\}}f\left(\frac{\eps h_{\cC}}{T}\right)\pm\frac{\eps h_{\cC_*^i}}{T}.$$ Recall Definitions \ref{def:unique-large-cluster} and \ref{def: outmost open region}, we get that $\cC_*^i$ are connected by the main cluster of $\omega$. Then we can compute the Radon-Nikodym derivative between $\Phi^{0}_{p,\calO}$ and $\Phi^{\eps h}_{p,\calO}$ \begin{align}
        \frac{\Phi^{\eps h}_{p,\calO}(\omega)}{\Phi^{0}_{p,\calO}(\omega)}&=\frac{\prod_{i\in I}\prod_{\cC\in\fC^i\setminus\{\cC_*^i\}}\cosh(\frac{\eps h_{\cC}}{T})\cdot\cosh\left(\sum_{i\in I}\frac{\eps h_{\cC_*^i}}{T}+a\right)}{\Big\langle\prod_{i\in I}\prod_{\cC\in\fC^i\setminus\{\cC_*^i\}}\cosh\left(\frac{\eps h_{\cC}}{T}\right)\cdot\cosh\left(\sum_{i\in I}\frac{\eps h_{\cC_*^i}}{T}+a\right)\Big\rangle^{0}_{p,\calO}}\nonumber\\&=\frac{\exp\left(\sum_{i\in I}X_i^++a\right)+\exp\left(\sum_{i\in I}X_i^--a\right)}{\Big\langle\exp\left(\sum_{i\in I}X_i^++a\right)+\exp\left(\sum_{i\in I}X_i^--a\right)\Big\rangle^{0}_{p,\calO}}\nonumber\\&=\frac{\exp\left(\sum_{i\in I}X_i^++a\right)+\exp\left(\sum_{i\in I}X_i^--a\right)}{\exp(a)\cdot\prod_{i\in I}\big\langle\exp(X_i^+)\big\rangle^{\w,0}_{p,\Omega_i}+\exp(-a)\cdot\prod_{i\in I}\big\langle\exp(X_i^-)\big\rangle^{\w,0}_{p,\Omega_i}}\label{eq: radon derivative expansion}
    \end{align} where $a$ is the sum of the external field on the main cluster of $\tilde\omega$. Recalling Definition \ref{def: good external field for number of xi} and combining with Lemma~\ref{lem: anticoncentration for good external field}, we get that $$\langle X_i^\pm\rangle^0_{p,\Omega_i}\le \ln\left(\langle \exp(X_i^\pm)\rangle^0_{p,\Omega_i}\right)-C_1$$ for some $C_1>0$. Let $\sS^\pm$ denote the collection of configurations such that $$\big|\sum_{i\in I}X_i^\pm-\langle X_i^\pm\rangle^0_{p,\Omega_i}\big|\ge \frac{C_1n}{2}.$$ Noticing that $\Phi^0_{p,\calO}$ is the product measure of $\phi^{\w,0}_{p,\Omega_i}$,  by Lemma~\ref{lem: concentration inequality with finite exponential moment} and $|I|\ge c_1n$ we obtain that \begin{equation}\label{eq: concentration for Xi}
        \Phi^0_{p,\calO}\Big(\sS^+\cup \sS^-\Big)\le C_2^{-1}\exp(-C_2n).
    \end{equation}
     Further for any $\omega\in (\sS^+)^c\cap(\sS^-)^c$, we have that \begin{align}
         \exp\left(\sum_{i\in I}X_i^\pm\pm a\right)&\le\exp\left(\sum_{i\in I}\left\langle X_i^\pm\right\rangle^{\w,0}_{p,\Omega_i}\pm a+\frac{C_1n}{2}\right)\nonumber\\&\le \prod_{i\in I}\left\langle \exp(X_i^\pm)\right\rangle^{\w,0}_{p,\Omega_i}\cdot\exp\left(\pm a-\frac{C_1n}{2}\right).\label{eq: concentration into radon}
     \end{align} Combining \eqref{eq: concentration into radon} with \eqref{eq: radon derivative expansion} shows that for any $\omega\in (\sS^+)^{c}\cap(\sS^-)^{c}$
     \begin{align}\label{eq: tv exponential decay low}
        \frac{\Phi^{\eps h}_{p,\calO}(\omega)}{\Phi^{0}_{p,\calO}(\omega)}\le \exp\left(-\frac{C_1n}{2}\right).
    \end{align}The desired result comes from combining \eqref{eq: tv exponential decay low} with \eqref{eq: concentration for Xi}.
\end{proof}

With Lemma \ref{lem: good external field for number of xi} and Lemma \ref{prop: exponential decay through coarse graining low} in place of Lemma~\ref{lem: good external field for product measure}, the proof of the singularity in {the} low-temperature regime
is the same to that of high-temperature and critical-temperature regimes. Thus, we omit the details here.

Finally, we come back to the proof of Lemmas~\ref{lem: anticoncentration for good external field} and \ref{lem: good external field for number of xi}.

\begin{proof}[Proof of Lemma~\ref{lem: anticoncentration for good external field}]
Combining \eqref{item: good external field 1} and \eqref{item: good external field 4} in Definition~\ref{def: good external field for product measure low decomposition} with Lemma~\ref{lem: dis concentration using variance}, we get that there exist two collections of configurations $\sS_1$ and $\sS_2$, such that $\phi_{p,\Omega}^{\w,0}(\sS_i)\ge C_1$ and for any $\omega_1\in\sS_1,\omega_2\in\sS_2$ we have $|\frac{\eps h_{\cC_*(\omega_1)}}{T}-\frac{\eps h_{\cC_*(\omega_2)}}{T}|\ge C_2$ where $C_1,C_2>0$ are constants depending only on $\theta$ and $c$. Let $\sS_*$ denote the collection of configurations such that 
$$\sum_{\cC\in\mathfrak{C}\setminus\{\cC_*\}}f\left(\frac{\eps h_{\cC}}{T}\right)-\frac{\eps^2|\Omega|-|\cC_*|}{2T^2}\le M^{-0.1}\text{ and }\big|~|\cC_*|-\langle|\cC_*|\rangle~\big|\le M^{1.5}.$$ 
Then letting $M$ big enough and combining \eqref{item: good external field 3} in Definition \ref{def: good external field for product measure low decomposition} with Lemma \ref{lem: MDP for boundary cluster at low temperature} shows that $\phi_{p,\Omega}^{\w,0}(\sS_i\cap\sS_*)\ge C_3$. Thus, for any $\omega_1\in\sS_1,\omega_2\in\sS_2$ we have \begin{align*}
    &\big|X^+(\omega_1)-X^+(\omega_2)\big|\ge \Big|\frac{\eps h_{\cC_*(\omega_1)}}{T}-\frac{\eps h_{\cC_*(\omega_2)}}{T}\Big|-\Bigg|\sum_{\cC\in\mathfrak{C}(\omega_1)\setminus\{\cC_*(\omega_1)\}}f\left(\frac{\eps h_{\cC(\omega_1)}}{T}\right)-\frac{\eps^2|\Omega|-|\cC_*(\omega_1)|}{2T^2}\Bigg|\\ &\hspace{5em}-\Bigg|\sum_{\cC\in\mathfrak{C}(\omega_2)\setminus\{\cC_*(\omega_2)\}}f\left(\frac{\eps h_{\cC}}{T}\right)-\frac{\eps^2|\Omega|-|\cC_*(\omega_2)|}{2T^2}\Bigg|-\Bigg|\frac{\eps^2 (\big||\cC_*(\omega_1)|-|\cC_*(\omega_2)|\big|)}{2T^2}\Bigg|\ge \frac{C_2}{2}.
\end{align*}
Therefore, we get that $\var_{\phi}(X^+)\ge C_4$ for some $C_4>0$ depending only on $\theta$ and $c$. Combined with \eqref{item: good external field 2} in Definition \ref{def: good external field for product measure low decomposition} and Lemma \ref{lem: inverse Jenson}, it yields that $$\langle X^+\rangle^0_{p,\Omega}\le \ln\big(\langle \exp(X^+)\rangle^0_{p,\Omega}\big)-C_5.$$ By symmetry, the same result holds for $X^-$.
\end{proof}
\begin{proof}[Proof of Lemma~\ref{lem: good external field for number of xi}]
    We first prove that with high $\phi^{\gamma,0}_{p,\Lambda_N}$ probability, an edge configuration $\omega$ is well-connected. It is a typical result for the low-temperature FK-Ising model thus we just sketch the proof here. 
    
    For a box $B_i=\Lambda_{2M}(u_i),$ let $\mcc G(u)$ denote the event that there exist two open circuits in the annuli $\Lambda_{3M}(u)\setminus\Lambda_{2M}(u)$ and $\Lambda_{2M}(u)\setminus\Lambda_{M}(u)$ respectively, and these two circuits are connected in $\omega{_{\mid\Lambda_{3M}(u)\setminus\Lambda_M(u)}}.$ Then by the RSW theory (see \cite{DT20} for example), there exists a constant $C_1>0$ relying only on $p$ such that for any boundary condition $\gamma$ on $\pari\Lambda_{3M}(u_i),$ $\phi^{\gamma,0}_{p,\Lambda_{3M}(u_i)}(\mcc G(u))\geq 1-\exp(-C_1M).$
    Therefore, we define 
    $\mcc Q(\omega)=\left(Q_x=\1_{\mcc G(2Mx)}:x\in\Lambda_{\frac{N}{2M}}\right)$ to be a random vertex configuration on the box $\Lambda_{\frac{N}{2M}},$ and then by \cite{LSS97}, we get that $\mcc Q$ dominates a Bernoulli site percolation $\mbb P_{q},$ where $q$ can be arbitrary close to 1 as $M$ goes to infinity.
    We define $\cA$ to be the event that the following holds:

\textbf{(1)} There exists a unique cluster $C_{\mcc Q}$ in $\mcc Q$ with diameter at least $\frac{N}{4M}$ and volume at least $\frac{C_2N^2}{M^2}$ in $\mcc Q(\omega)$.
%Then the cluster induces a cluster $\cC$ in $ \Lambda:=\bigcup_{i=1}^n(\Lambda_{2M}(u_i)\setminus\Lambda_{M}(u_i))$.}

\textbf{(2)} For any connected subset $S$ in $\Lambda_{\frac{N}{2M}}$ with diameter at least $\frac{N}{100M}$, we have $S\cap C_{\mcc Q}\neq \emptyset.$

Next we will prove for any $\omega\in \cA$, we have $\omega$ is well-connected. It is easy to see that $C_{\mcc Q}$ induces a cluster $\cC$ with diameter at least $N/2$ in $ \Lambda=\bigcup_{i=1}^n(\Lambda_{2M}(u_i)\setminus\Lambda_{M}(u_i))$. So it suffices to prove that this is the unique cluster in $\Lambda$ with diameter at least $N/2$. For any cluster $\cC^\prime$ in $\Lambda$ with diameter at least $N/2$, let $S=\{x\in\Lambda_{\frac{N}{2M}}: \Lambda_{M}(2Mx){\cap\mcc C'\neq\emptyset} \}$. By \textbf{(2)}, we get that there exists $x\in S\cap C_{\mcc Q}$. Since $Q_x=1$ and $x\in{S}$, $\cC^\prime$ must intersect with the circuit in $\Lambda_{2M}(u)\setminus\Lambda_{M}(u)$. Thus we get $\cC^\prime$ intersects with $\cC$ and $\omega$ is well-connected.
    By standard percolation result, under the measure $\mbb P_p,$ $\cA$ has probability $1-\exp(-C_3\frac{N}{M}).$   
    Thus, we get $\phi^{\gamma,0}_{p,\Lambda_N}(\omega \text{ is well-connected})\geq 1-\exp\left (-C_4\frac{N}{M}\right)$ by noticing that a connected component in $\mcc Q$ induces an open cluster in $\omega.$

\begin{figure}[htbp]
\vspace{-5pt}
\centering
\includegraphics[scale=0.3   ]{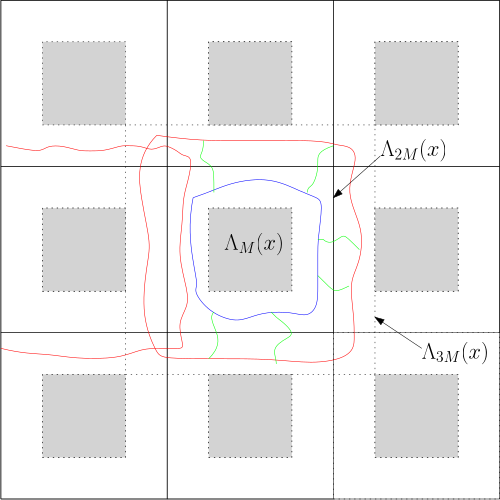} % 插入图片
\vspace{-5pt} % 缩小图片与文字之间的距离
\caption{Illustration of the event $\mcc{A}$. The blue circuit lies in $\Lambda_{2M}(x)\setminus\Lambda_{M}(x)$,  the red circuit in $\Lambda_{3M}(x)\setminus\Lambda_{2M}(x)$, and the green path(s) connect these two circuits. Neighbouring red circuits intersect with each other.} % 在图片下方添加说明
\label{}
\vspace{-5pt} % 进一步缩小距离（可选）
\end{figure}

We choose $\theta=1$ and $M$ such that $\theta M^{-1}\le \eps\le 2\theta M^{-1}$. Furthermore, we choose $c$ in Definition~\ref{def: good external field for product measure low decomposition} to be the first constant in Lemma \ref{lem: anticoncentrate external field at T small2}. Then by the results in Section~\ref{sec: anti concentration low} (see Lemmas~\ref{lem: anticoncentrate external field at T small0}, \ref{lem: anticoncentrate external field at T small1} and \ref{lem: anticoncentrate external field at T small2}), we get that the external field on the region $\Omega$ is good with a positive probability.
   Hence by Lemma~\ref{lem: concentration inequality with finite exponential moment}, we get that   $$\P\otimes\fk\gamma\big(\eta(\calO)\ge C_5n\big)\ge 1- C_5^{-1}\exp\left(-C_5\frac{N}{M}\right)$$ and the desired result comes from Markov's inequality.
\end{proof}

\appendix
\renewcommand{\appendixname}{Appendix~\Alph{section}}
\section{Some LDP results for FK-Ising model at criticality}\label{sec: LDP at critical}

In this section,  we fix $T=T_c$ and omit $p_c$ from the notation.

The following theorem will be the starting point of this section, which extends the large deviation bound of the largest cluster (\cite{Kiss14}) for the Bernoulli percolation model to the FK-Ising model. 
\begin{thm}\label{thm:LDP for maiximal cluster}
There exists a constant $c_1>0$ such that    $$\phi_{\lamn}^{\w,0}\left(\max_{\mcc C\in\fC}|\cC|\ge N^{\frac{15}{8}}x\right)\le c_1\exp\left(-c_1^{-1}x^{16}\right)$$
holds for all positive integers $N$.
\end{thm}

As we remarked above, a continuous version of Theorem~\ref{thm:LDP for maiximal cluster} has been shown in \cite{CJN20appendix}, and  their proof cannot be applied to the discrete settings. So we follow the framework of \cite{Kiss14} to prove Theorem~\ref{thm:LDP for maiximal cluster}.

    First, for any positive integers $m<n$, we define $\mcc A(m,n)$ to be the annulus $\Lambda_n\setminus \Lambda_m,$ and define $C(m,n)$ to be the event that there is an open path around $\mcc A(m,n),$ that is, this open path separates $\pare\Lambda_m$ and $\pari\Lambda_n.$ Also we define $D(m,n)$ to be the event that there exists an open path crossing $\mcc A(m,n),$ or equivalently speaking, connecting $\pari\Lambda_n$ and $\pare\Lambda_m.$ Finally, define
    $$\pi(m,n)=\phi^{\w,0}_{\mcc A(m,n)}(D(m,n)).$$ We point out that here the wired boundary condition on $\mcc A(m,n)$ means that $\pare \Lambda_m$ and $\pari\Lambda_n$ are wired into one point respectively, but they are not wired together. And we write $\pi(n)=\pi(1,n)$ for short.

    Given this notation, we want to check Assumptions 1.1 and 1.2 in \cite{Kiss14}. That is 
    \begin{enumerate}[(i)]
        \item There exists a constant $C_1>0$ such that for any  $0<k<l<m$, 
        $$\pi(k,l)\pi(l,m)\leq C_1\pi(k,m).$$
        \item There exists  positive constants $C_2,\alpha>0$ such that for all $n>m\geq 1,$
        $$\frac{\pi(n)}{\pi(m)} \geq C_2\left(\frac{n}{m}\right)^{-\alpha}.$$
    \end{enumerate}

Assumption (ii) is for true for $\alpha=\frac{1}{8},$ since in \cite{On44} (see also \cite{CDH16,DHN11,GW20}) they have shown that there exists a constant $c>0$ such that for any $n>0$
    \begin{eqnarray}\label{estimate-pi(n)}
         c^{-1}n^{-\frac{1}{8}}\leq \pi(n)\leq cn^{-\frac{1}{8}}.
    \end{eqnarray}
Then we check assumption (i).

First, if $m\geq 2l\geq 4k,$ then we get that
\begin{equation}\label{assumption1-1}
      \pi(k,m)=\phi^{\w,0}_{\mcc A(k,m)}\left(D(k,m)\right)\geq \phi^{\w,0}_{\mcc A(k,m)}\left(D(k,m)\mid C(l,2l)\right)\times \phi^{\w,0}_{\mcc A(k,m)}\left(C(l,2l)\right).  
\end{equation}

\begin{figure}[htbp]
\vspace{-10pt}
\centering
\includegraphics[scale=0.1]{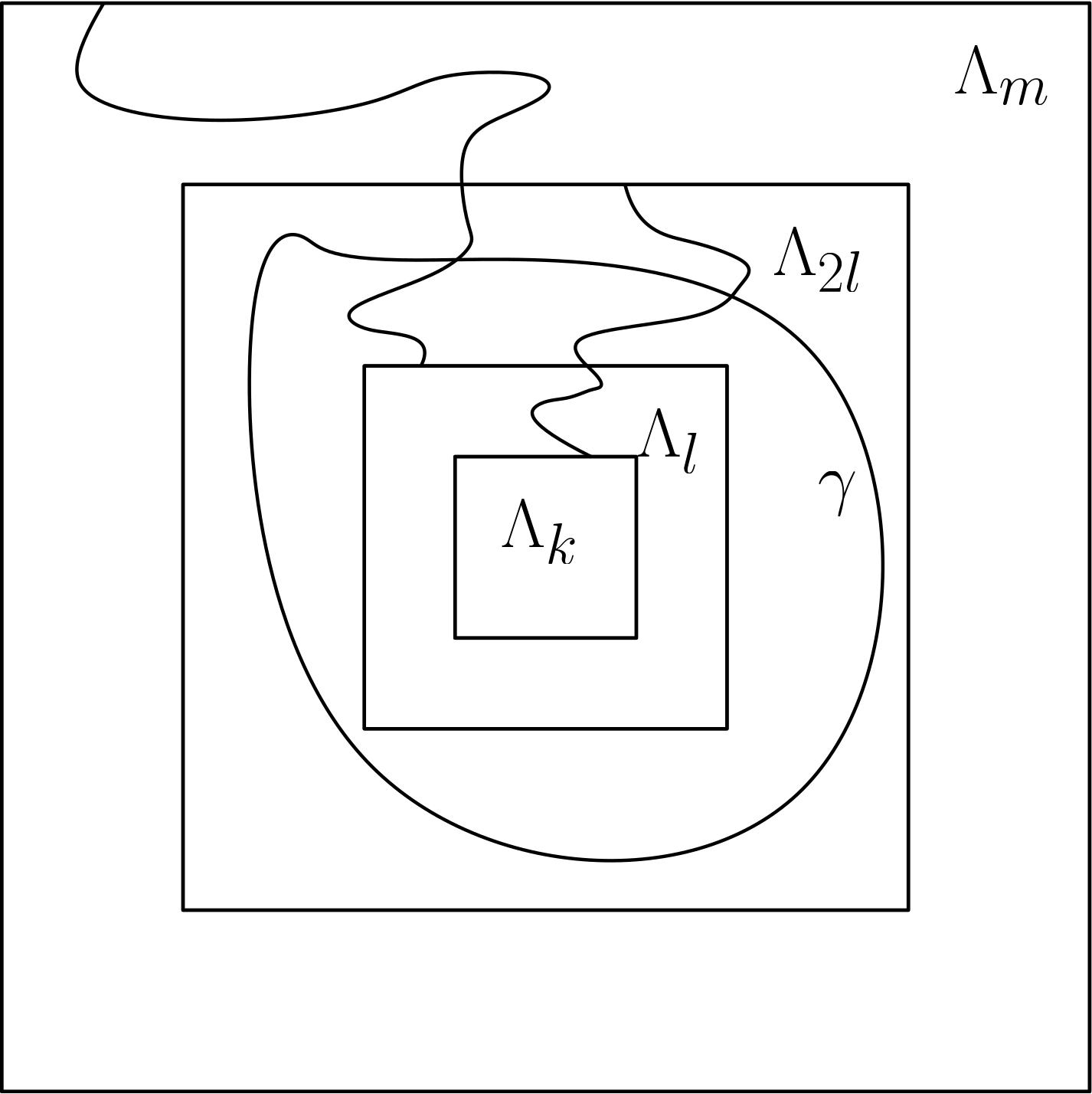} % 插入图片
\vspace{-10pt} % 缩小图片与文字之间的距离
\caption{An illustration of how $\gamma$ connects the two crossing.} % 在图片下方添加说明
\label{fig: box crossing}
\vspace{-5pt} % 进一步缩小距离（可选）
\end{figure}

Given the event $C(l,2l)$ happens, there is an open path $\gamma$ in $\mcc A(l,2l)$ separating its inner and outer boundaries, thus $\gamma$ also separates $\pare\Lambda_k$ and $\pari\Lambda_m.$ As a result, if both $D(k,2l)$ and $D(l,m)$ happen, this two crossing should be connected together by $\gamma$ (See Figure~\ref{fig: box crossing} for illustration). Combined with CBC, we have
\begin{equation}\label{assumption1-2}
    \phi^{\w,0}_{\mcc A(k,m)}(D(k,m)\mid C(l,2l))\geq \pi(k,2l)\pi(l,m).
\end{equation}
Similarly, we have 
\begin{equation}\label{assumption1-3}
    \pi(k,2l)\geq \pi(k,l)\pi\left(\Big\lfloor\frac{l}{2}\Big\rfloor,2l\right)\phi_{\mcc A(k,2l)}^{\w,0}\left(C\left(\Big\lfloor\frac{l}{2}\Big\rfloor,2l\right)\right).
\end{equation}

Finally, by the RSW theory (see \cite[Theorem 1.1]{DHN11}), we know there is a positive constant $c'>0$ such that
\begin{align}\label{assumption1-4}
    \phi_{\mcc A(m,n)}^{\w,0}\left(C(m,n)\right)&\geq c' \text{ if } n\geq 2m;\nonumber\\
    \phi_{\mcc A(m,n)}^{\w,0}\left(D(m,n)\right)&\geq c' \text{ if } n\leq 8m.    
\end{align}

Now putting \eqref{assumption1-1}, \eqref{assumption1-2}, \eqref{assumption1-3} and \eqref{assumption1-4} together, we have
$$\pi(k,m)\geq (c')^3\pi(k,l)\pi(l,m).$$

Furthermore, \eqref{assumption1-3} actually proves that $\pi(k,2l)\geq (c')^2\pi(k,l)$. Combining with the fact that $\pi(a,b)$ increases with $a$ and decreases with $b$, we can solve the case $m<2l$ and $l<2k$ in a similar but simpler way. 

Define
\begin{equation*}
  \cV_n := \left\{v \in \Lambda_n\, : \, v\leftrightarrow
\pari \Lambda_{2n}\right\} \label{eq:def_V_n}
\end{equation*}
to be the set of vertices in $\Lambda_n$ which are connected to $\pari\Lambda_{2n}$.
Then \cite[Theorem 1.5]{Kiss14} and \eqref{estimate-pi(n)} show 
\begin{lem}\label{lem: moment bound}
   There is a constant $c_2$ such that for all $n>0$
\begin{equation}\label{eq: moment bound in LDP}
  \left\langle\binom{\left|\cV_n\right|}{k}\right\rangle_{\Lambda_{2n}}^{\w}  \leq (\frac{c_2n^{\frac{15}{8}}}{k^{\frac{15}{16}}}
)^k. 
\end{equation}
\end{lem}
\begin{rmk}
    Although \cite{Kiss14} only considers the percolation model, their proof on Theorem 1.5 (in \cite{Kiss14}) can be naturally extended to the FK-Ising model.
\end{rmk}
\begin{cor}\label{cor: exponential moment bound}
For $t>n^{-\frac{15}{8}}$, we have 
    $$ \left\langle(1+t)^{|\cV_n|}\right\rangle_{\Lambda_{2n}}^{\w}\le c_3\exp(c_2t^{\frac{16}{15}}n^2).$$
\end{cor}
\begin{proof}
We expand $(1+t)^{|\cV_n|}$ and divide it into two parts at $k=t^{\frac{16}{15}}n^2$
    \begin{align}
  \left\langle(1+t)^{|\cV_n|}\right\rangle_{\Lambda_{2n}}^{\w}
  & = \sum_{k=0}^\infty t^k \left\langle\binom{\left|\cV_n\right|}{k}\right\rangle_{\Lambda_{2n}}^{\w} \stackrel{\eqref{eq: moment bound in LDP}}{\leq} \sum_{k=0}^\infty \left(t\cdot \frac{c_2n^{\frac{15}{8}}}{k^{\frac{15}{16}}}\right)^k\nonumber\\
 &= \sum_{k=0}^{t^{\frac{16}{15}}n^2} \left(t\cdot \frac{c_2n^{\frac{15}{8}}}{k^{\frac{15}{16}}}\right)^k + \sum_{k=t^{\frac{16}{15}}n^2+1}^\infty \left(t\cdot \frac{c_2n^{\frac{15}{8}}}{k^{\frac{15}{16}}} \right)^k.\label{eq:exponential expansion for t^Vn}
\end{align}
We now compute the two terms separately. For $k\le t^{\frac{16}{15}}n^2$, we have $t\cdot \frac{c_2n^{\frac{15}{8}}}{k^{\frac{15}{16}}}\le \frac{c_2t^{\frac{16}{15}}n^2}{k}$, thus 
\begin{equation}\label{eq:exponential expansion term 1}
    \sum_{k=0}^{t^{\frac{16}{15}}n^2} \left(t\cdot \frac{c_2n^{\frac{15}{8}}}{k^{\frac{15}{16}}}\right)^k\le \sum_{k=0}^{t^{\frac{16}{15}}n^2} \left(\frac{c_2t^{\frac{16}{15}}n^2}{k}\right)^k\le \sum_{k=0}^{t^{\frac{16}{15}}n^2} \frac{(c_2t^{\frac{16}{15}}n^2)^k}{k!}\le \exp(c_2t^{\frac{16}{15}}n^2).
\end{equation}
For the other term, we group each consecutive $ t^{\frac{16}{15}}n^2$ elements and apply a union bound since $\left(t\cdot \frac{c_2n^{\frac{15}{8}}}{k^{\frac{15}{16}}} \right)^k$ is decreasing, then we obtain
\begin{equation}\label{eq:exponential expansion term 2}
    \sum_{k=t^{\frac{16}{15}}n^2+1}^\infty \left(t\cdot \frac{c_2n^{\frac{15}{8}}}{k^{\frac{15}{16}}} \right)^k\le t^{\frac{16}{15}}n^2\sum_{l=1}^\infty (l/c_3)^{-\frac{15 }{16}t^{\frac{16}{15}}n^2l}\le t^{\frac{16}{15}}n^2\cdot C_1.
\end{equation}
Combining \eqref{eq:exponential expansion term 1} and \eqref{eq:exponential expansion term 2} into \eqref{eq:exponential expansion for t^Vn} gives
\begin{equation*}
    \left\langle(1+t)^{|\cV_n|}\right\rangle_{\Lambda_{2n}}^{\w}\le C_2\exp\left(c_2t^{\frac{16}{15}}n^2\right).\qedhere
\end{equation*}
\end{proof}
\begin{proof}[Proof of Theorem~\ref{thm:LDP for maiximal cluster}]
Without loss of generality we only consider the case that $N^{\frac{15}{8}}x\leq N^2,$ i.e. $x\leq N^{\frac{1}{8}}$ since the volume of any cluster inside $\lamn$ is no more than $N^2$ . Furthermore, we assume $x>\max\{K_1,K_2\}$ where $K_1,K_2>1$ are two large constants depending only on the constants $c_2$ and $c_3$ in Corollary~\ref{cor: exponential moment bound} to be determined later.

We divide the proof into two parts. 

\vspace{2em}
\textit{PART 1: Controlling the bulk cluster}

In the first part, we show that all the clusters except for the boundary cluster cannot be too large, 
the proof is really similar to the proof of \cite[Proposition 6.3]{BCKS99}. In order to be self-contained, we give a short proof here. 

Without loss of generality, we may assume that $N=2^{n}$ and define 
$\mss B_l:=\{ \Lambda_{2^{l-1}}(u): u\in (2^{l}\mbb Z)^2 \}$ to be a box partition of $\lamn$ with side length $2^l.$ Then we write $\mss B:=\bigcup_{l=1}^n \mss B_l. $ Now for a box $\Lambda_{2^{l-1}(u)}\in\mss B_l,$ we write $L=2^{l-1}$ for short and define
$\mcc V_{L}(u):=\{v\in \Lambda_{2^{l-1}(u)}: v\leftrightarrow \Lambda_{2^{l}(u)}\}.$
We will call the box $\Lambda_{2^{l-1}(u)}$ is {\bf crowded} if $|\mcc V_L(u)|\geq \frac{1}{289}N^{\frac{15}{8}}x$ and we have the following claim:

\textit{ If there exists a cluster in $\lamn$ with volume at least $N^{\frac{15}{8}}x$ and it is not the boundary cluster, then
at least one box in $\mss B$ is { crowded}.}
The proof is relatively straightforward, use $\mcc C$ to denote this cluster then there must exist an integer $l$ such that $2^{l+1}\le diam(\cC)<2^{l+2}$. Then $\cC$ will intersect with at most 289 boxes in $\mss B_l$ (this explains why we require that $\cC$ is not the boundary cluster).  Therefore, at least one of them is crowded since $diam(\cC)<2^{l+2}$ implies for any $x\in \cC$, $x$ connects to some vertex $2^l$ away. 
Therefore, all we need to do is to upper-bound the probability that there is a crowded box in $\mss B$.
Using Corollary~\ref{cor: exponential moment bound} and Markov inequality shows \begin{align}
    &\phi_{\Lambda_{2L}(u)}^{\w,0}\Big(|\cV_{L}(u)|\ge \frac{1}{289}N^{\frac{15}{8}}x\Big)\le (1+t)^{-\frac{1}{289}N^{\frac{15}{8}}x}\left\langle(1+t)^{|\cV_L|}\right\rangle^{\w}_{\Lambda_{2L}}\nonumber\\&\le (1+t)^{-\frac{1}{289}N^{\frac{15}{8}}x} c_3\exp\left(c_2t^{\frac{16}{15}}{L}^2\right).\label{eq:  LDP for maximal cluster single choice bound 0}
\end{align} Choosing $t=\min\{\frac{N^{\frac{225}{8}}x^{15}}{K_1L^{30}},1\}$ (recall that $1<K_1<x$ is the constant defined at the beginning of the proof), then $t>N^{-\frac{15}{8}}$ since $L=2^{l-1}\leq N$. Applying a simple bound $(1+t)^{-\frac{1}{289}N^{\frac{15}{8}}x}\le \exp(-\frac{t}{2\cdot289}N^{\frac{15}{8}}x)$ to \eqref{eq:  LDP for maximal cluster single choice bound 0}, we get that for $\frac{N^{\frac{225}{8}}x^{15}}{K_1L^{30}}<1$, \begin{align*}
    &\phi_{\Lambda_{2L}(u)}^{\w,0}\Big(|\cV_{L}(u)|\ge \frac{1}{289}N^{\frac{15}{8}}x\Big)\le c_3\exp\left(-\frac{1}{2\cdot289K_1}\frac{N^{30}x^{16}}{L^{30}}+\frac{c_2}{K_1^{\frac{16}{15}}}\frac{N^{30}}{L^{30}}\right).
\end{align*}For $\frac{N^{\frac{225}{8}}x^{15}}{K_1L^{30}}>1$, we have \begin{align*}
    \phi_{\Lambda_{2L}(u)}^{\w,0}\Big(|\cV_{L}(u)|\ge \frac{1}{289}N^{\frac{15}{8}}x\Big)&\le c_3\exp\left(-\frac{1}{2\cdot289}N^{\frac{15}{8}}x+c_2{L}^2\right)\\&\le c_3\exp\left(-\frac{1}{2\cdot289}N^{\frac{15}{8}}x+\frac{c_2}{K_1^{\frac{1}{15}}}N^{\frac{15}{8}}x\right).
\end{align*}
Let $l_0$ denote the maximal integer such that $\frac{N^{\frac{225}{8}}x^{15}}{K_1}>2^{30l_0}$.
Letting $K_1$ big enough depending on $c_2$ and $c_3$, we get that \begin{align}
    \phi_{\Lambda_{2L}(u)}^{\w,0}\Big(|\cV_{L}(u)|\ge \frac{1}{289}N^{\frac{15}{8}}x\Big)\le C_1\exp\left(-\frac{1}{C_1}\frac{N^{30}x^{16}}{L^{30}}\right)~~\mbox{if }L= 2^{l-1},~l>l_0;\label{eq: LDP for maximal cluster single choice bound 1} \\
    \phi_{\Lambda_{2L}(u)}^{\w,0}\Big(|\cV_{L}(u)|\ge \frac{1}{289}N^{\frac{15}{8}}x\Big)\le C_1\exp\left(-\frac{1}{C_1}N^{\frac{15}{8}}x\right)~~\mbox{if }L= 2^{l-1},~l\le l_0.\label{eq: LDP for maximal cluster single choice bound 2}
\end{align} At the same time, there are at most $(\frac{N}{L})^2$ elements in $\mss B_l$. The desired result follows from taking a union bound, summing \eqref{eq: LDP for maximal cluster single choice bound 1} and \eqref{eq: LDP for maximal cluster single choice bound 2} over all boxes in $\mss B$ and recalling $L=2^{l-1}\leq N,$ $x\leq N^{\frac{1}{8}}$:
\begin{align}
    \sum_{l=2}^{l_0}C_1\exp\left(-\frac{1}{C_1}N^{\frac{15}{8}}x\right)\cdot \left(\frac{N}{L}\right)^2+\sum_{l=l_0+1}^{\lfloor\log_2N\rfloor}C_1\exp\left(-\frac{N^{30}x^{16}}{C_1L^{30}}\right)\cdot \left(\frac{N}{L}\right)^2\le C_2\exp\left(-C_2^{-1}x^{16}\right).\nonumber
\end{align}

\vspace{2em}
\textit{PART 2: Controlling the boundary cluster}

For the second part, we are going to upper-bound the probability that the boundary cluster is too large. We are not able to repeat the proof above since with the wired boundary, the boundary cluster with a small diameter can intersect with a large number of boxes. As a result, we may exploit a multilevel partition of $\lamn$ instead of a uniform partition. Without loss of generality, we assume $N=2^{8k}$ and $\lfloor\log_2(x)\rfloor=a$. Let $\sT_{8k-2}$ denote the set of boxes with side length $2^{8k-1}$ with one corner at the origin $o$. Let $\sT_{l}(l=7k+a-1,7k+a,\cdots,8k-3)$ be the collection of boxes with side length $2^{l+1}$ and forms a disjoint contour surrounding boxes in $\sT_{l+1}$ (see Figure \ref{fig:Multilevel partition} for illustration).

For convenience, we use $\cT_l(l=7k+a-1,\cdots,8k-2)$ to denote the union of boxes in $\sT_l.$ Since the number of points not contained in $\cup_{l=7k+a-1}^{8k-2}\cT_{l}$ is $N^{\frac{15}{8}}2^{a-1}<\frac{N^{\frac{15}{8}}x}{2}$, we only need to consider the intersection of $\cC_*$ and $\cup_{l=7k+a-1}^{8k-2}\cT_{l}$. Let $C_3$ be an absolute constant such that $\sum_{l=7k+a-1}^{8k-2}C_32^{\frac{-8k+l}{2}}\le 1$. Then we have \begin{equation}
    \phi_{\Lambda_{N}}^{\w,0}\left(|\cC_*|\ge N^{\frac{15}{8}}x\right)\le \sum_{l=7k+a-1}^{8k-3}\fkc{\w}(|\cC_*\cap\cT_l|\ge C_3N^{\frac{15}{8}}x\cdot 2^{\frac{-8k+l}{2}}).\label{eq: boundary cluster LDP multilevel bound}
\end{equation}
Next we want to control the probability $\fkc{\w}(|\cC_*\cap\cT_l|\ge C_3N^{\frac{15}{8}}x\cdot 2^{\frac{-8k+l}{2}})$. Here we remark that the factor $2^{\frac{-8k+l}{2}}$ is used to ensure the sequence is summable.

\begin{figure}
    \centering
    \begin{tikzpicture}
    \draw (-4,-4) rectangle (4,0);
    \draw (-2,-2) rectangle (0,0);
    \draw (0,0) rectangle (2,-2);
    \node at (-3.75,-1) {\tiny$\cdots$};
    \node at (3.75,-1) {\tiny$\cdots$};
    \node at (0,-3.75) {\tiny$\vdots$};
    %\node at (0,3.75) {\tiny$\vdots$};
    %\node at (1,1) {$\sT_{8k-2}^3$};
    %\node at (-1,1) {$\sT_{8k-2}^2$};
    \node at (1,-1) {$\sT_{8k-2}^4$};
    \node at (-1,-1) {$\sT_{8k-2}^1$};
    \node at (0.15,-0.15) {$o$};
    %\draw (-2,-2) rectangle (2,0);
    \draw (-3,-3) rectangle (-3+1,-3+1) node at (-3+0.5,-3+0.5) {$\sT_{8k-3}^1$};
    \draw (-3,-2) rectangle (-3+1,-2+1) node at (-3+0.5,-2+0.5) {$\sT_{8k-3}^4$};
    \draw (-3,-1) rectangle (-3+1,-1+1) node at (-3+0.5,-1+0.5) {$\sT_{8k-3}^3$};
    %\draw (-3,0) rectangle (-3+1,+1) node at (-3+0.5,+0.5) {$\sT_{8k-3}^4$};
    %\draw (-3,1) rectangle (-3+1,1+1) node at (-3+0.5,1+0.5) {$\sT_{8k-3}^1$};
    %\draw (-3,2) rectangle (-3+1,2+1) node at (-3+0.5,2+0.5) {$\sT_{8k-3}^2$};
     \draw (2,-3) rectangle (2+1,-3+1) node at (2+0.5,-3+0.5) {$\sT_{8k-3}^2$};
    \draw (2,-2) rectangle (2+1,-2+1) node at (2+0.5,-2+0.5) {$\sT_{8k-3}^3$};
    \draw (2,-1) rectangle (2+1,-1+1) node at (2+0.5,-1+0.5) {$\sT_{8k-3}^4$};
    %\draw (2,0) rectangle (2+1,+1) node at (2+0.5,+0.5) {$\sT_{8k-3}^1$};
    %\draw (2,1) rectangle (2+1,1+1) node at (2+0.5,1+0.5) {$\sT_{8k-3}^4$};
    %\draw (2,2) rectangle (2+1,2+1) node at (2+0.5,2+0.5) {$\sT_{8k-3}^3$};
     \draw (-2,-3) rectangle (-2+1,-3+1)node at (-2+0.5,-3+0.5) {$\sT_{8k-3}^2$};
     \draw (-1,-3) rectangle (-1+1,-3+1)node at (-1+0.5,-3+0.5) {$\sT_{8k-3}^3$};
     \draw (0,-3) rectangle (+1,-3+1)node at (+0.5,-3+0.5) {$\sT_{8k-3}^4$};
     \draw (1,-3) rectangle (1+1,-3+1)node at (1+0.5,-3+0.5) {$\sT_{8k-3}^1$};
     %\draw (-2,2) rectangle (-2+1,2+1)node at (-2+0.5,2+0.5) {$\sT_{8k-3}^3$};
     %\draw (-1,2) rectangle (-1+1,2+1)node at (-1+0.5,2+0.5) {$\sT_{8k-3}^4$};
     %\draw (0,2) rectangle (+1,2+1)node at (+0.5,2+0.5) {$\sT_{8k-3}^1$};
     %\draw (1,2) rectangle (1+1,2+1)node at (1+0.5,2+0.5) {$\sT_{8k-3}^2$};
    \foreach \i in {-3.5,3}{
    \foreach \j in {-3.5,-3,-2.5,-2,-1.5,-1,-0.5}{
    \draw (\i,\j) rectangle (\i+0.5,\j+0.5);
    }
    }
    \foreach \j in {-3.5}{
    \foreach \i in {-3,-2.5,-2,-1.5,-1,-0.5,0,0.5,1,1.5,2,2.5}{
    \draw (\i,\j) rectangle (\i+0.5,\j+0.5);
    }
    }
\end{tikzpicture}
    \caption{Multilevel partition of $\lamn$}
    \label{fig:Multilevel partition}
\end{figure}
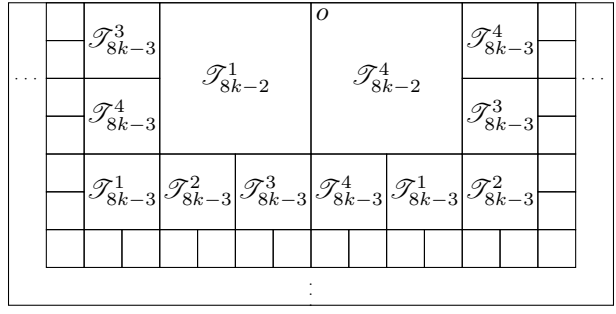

Since $\cC_*$ is the boundary cluster, we get that for each box $\Lambda_{2^l}(u)$ in $\sT_l$, if it contains a point in $\cC_*$, then that point is connected to $\Lambda_{2^{l+1}}(u)$. Since the number of boxes in $\mss T_l$ is divisible by $4,$  we can partition the boxes in $\sT_l$ into $4$ groups $\sT_l^1,\cdots\sT_l^{4}$ (also use $\cT_l^1,\cdots,\cT_l^{4}$ to denote the union of the boxes) such that for any  two different boxes $\Lambda_{2^l}(u),\Lambda_{2^l}(v)$ in a group, we have $\Lambda_{2^{l+1}}(u)\cap \Lambda_{2^{l+1}}(v)=\emptyset$ (see the superscripts in Figure \ref{fig:Multilevel partition} above). Now we can control the number of points connected to the boundary in each group by applying CBC and Corollary~\ref{cor: exponential moment bound}~:
\begin{align}
        &\phi_{\Lambda_{N}}^{\w,0}\left(|\cC_*\cap\cT_l^i|\ge C_3N^{\frac{15}{8}}x\cdot 2^{\frac{-8k+l}{2}}\right)\nonumber\\
        \le ~&\phi_{\Lambda_{N}}^{\w,0}\left(\sum_{\Lambda_{2^l}(u)\in \sT_l^i}|\mcc V_{2^l}(u)|\ge C_3N^{\frac{15}{8}}x\cdot 2^{\frac{-8k+l}{2}}\right)\nonumber\\
        \le ~& (1+t)^{-C_3N^{\frac{15}{8}}x\cdot 2^{\frac{-8k+l}{2}}}\times\prod_{\Lambda_{2^l}(u)\in \sT_l^i}\left\langle(1+t)^{|\cV_{2^l}(u)|}\right\rangle_{\Lambda_{2^{l+1}}(u)}^{\w}\nonumber\\
        \le ~& (1+t)^{-C_3N^{\frac{15}{8}}x\cdot 2^{\frac{-8k+l}{2}}}\times c_{ 3}\exp\left(c_{ 2}|\cT_l^i|t^{\frac{16}{15}}2^{2l}\right).\label{eq:l_level boundary cluster deviation control}
\end{align}
Letting $t=N^{-\frac{15}{8}}x^{15}2^{\frac{15(8k-l)}{16}}/K_2<1$ (recall that $1<K_2<x$ is the constant in the beginning of the proof and thus, $t>n^{-\frac{15}{8}}$) in \eqref{eq:l_level boundary cluster deviation control} and noting that $|\cT_l^i|\le 2^{8k-l}$, we obtain that $$\phi_{\Lambda_{N}}^{\w,0}\left(|\cC_*\cap\cT_l^i|\ge C_3N^{\frac{15}{8}}x\cdot 2^{\frac{-8k+l}{2}}\right)\le \exp\left(-C_3x^{16}2^{\frac{7(8k-l)}{16}}/K_2\right)\times c_3\exp\left(c_2x^{16}/K_2^{\frac{16}{15}}\right).$$
Letting $K_2$ big enough depending on $c_2$ and $c_3$ gives $$\phi_{\Lambda_{N}}^{\w,0}(|\cC_*\cap\cT_l^i|\ge C_3N^{\frac{15}{8}}x\cdot 2^{\frac{-8k+l}{2}})\le C_4\exp(-C_4^{-1}x^{16}\cdot2^{\frac{7(8k-l)}{16}}).$$ Combined with \eqref{eq: boundary cluster LDP multilevel bound}, it completes the proof:
\begin{equation*}
    \phi_{\Lambda_{N}}^{\w,0}\left(|\cC_*|\ge N^{\frac{15}{8}}x\right)\le \sum_{l=7k+a-1}^{8k-3} C_4\exp\left(-C_4^{-1}x^{16}\cdot 2^{\frac{7(8k-l)}{16}}\right)\le C_5\exp\left(-C_5^{-1}x^{16}\right).\qedhere
\end{equation*}
\end{proof}

For a configuration $\omega\in\Xi$, define
$\mss M(\omega):=\sum_{\mcc C\in\fC(\omega)} |\cC|^2,$
to be the summation of the squares of the number of clusters in $\omega.$

Our main goal here is to show the following LDP result:

\begin{thm}\label{thm:LDP for sum of cluster squares}
 There exists a positive constant $c_4>0$ such that for any $N$ we have
 $$\phi_{\lamn}^{\w,0}\left(\mss M(\omega)\geq N^{\frac{15}{4}}x\right)\leq c_4\exp\left(-c_4^{-1}x^{8}\right).$$
\end{thm}

\begin{proof}
    By a simple computation, we have
\begin{align*}
   \mss M(\omega)=\sum_{l=1}^{N^2}\sum_{\mcc C\in\fC,|\mcc C|\geq l}|\mcc C| 
   \leq \sum_{k=0}^{\lfloor 2\log_2 N\rfloor} 2^{k+1} \sum_{\mcc C\in\fC, |\mcc C|\geq 2^k}|\mcc C|.
\end{align*}

We may divide the summation into three pieces with respect to the level $k$. First for $2^k< N^{\frac{7}{4}}x/32$, since $\sum_{|\mcc C|\geq 2^k}|\mcc C|\le 4N^2$, we get that \begin{equation}\label{eq:LDP case 1}
    \sum_{2^{k+5}< N^{\frac{7}{4}}}2^{k+1} \sum_{\mcc C\in\fC, |\mcc C|\geq 2^k}|\mcc C|\le N^{\frac{15}{4}}x/4.
\end{equation}
Secondly, for $2^k\ge N^{\frac{15}{8}}\sqrt{x}$, applying Theorem \ref{thm:LDP for maiximal cluster} gives \begin{equation}\label{eq:LDP case 2}
\begin{aligned}
        &\phi_{\lamn}^{\w,0}\left(\sum_{2^k\ge N^{\frac{15}{8}}\sqrt{x}}2^{k+1}\sum_{\mcc C\in\fC,|\mcc C|\geq 2^k}|\mcc C|\ge N^{\frac{15}{4}}x/4\right)\\
        \le ~&\phi_{\lamn}^{\w,0}\left(\max\left(|\cC_j|\ge N^{\frac{15}{8}}\sqrt{x}\right)\right)
        \leq  c_1\exp\left(-c_1^{-1}x^8\right).
\end{aligned}
\end{equation}
The third part is much more complicated. For $N^{\frac{7}{4}}x/32\le 2^k< N^{\frac{15}{8}}\sqrt{x}$, we want to prove that 
\begin{equation}\label{eq:LDP case 3}
    \phi_{\lamn}^{\w,0}\left(\sum_{\mcc C\in\fC,|\mcc C|\geq 2^k}|\mcc C|\ge C_0N^{\frac{15}{4}}x2^{-k}2^{\frac{k-k_0}{2}}\right)\le C\exp\left(-C^{-1}x^8\cdot 2^{k_0-k}\right),
\end{equation}
where $k_0=\max\{k\in\mathbb{N}:2^k<N^{\frac{15}{8}}\sqrt{x}\}$. As in the proof of Theorem~\ref{thm:LDP for maiximal cluster}, the factor $2^{\frac{k-k_0}{2}}$ is introduced to ensure the sequence is summable.
We will first consider a dividing procedure of $\lamn$. To be precise, we fix an integer $M=M(k)>0$ and consider the following horizontal and vertical line segments inside $\lamn$:
$$\mss L_M:=\Big\{[-N,N]\times \{(2i+1)M\}, \{(2i+1)M\}\times [-N,N]: \text{for all } i \text{ with } |(2i+1)M|\leq N\Big\}.$$

Without loss of generality, we can assume that $(2L+1)M=N$ for some positive integer $L.$ Then $\lamn$ is divided into $(2L+1)^2$ boxes of size $M\times M$. We will use $\mss B_M$ to denote the set of all these boxes. Furthermore, we define $B(i,j):=[(2i-1)M,(2i+1)M]\times [(2j-1)M,(2j+1)M]$ to denote the box in $\mss B_M.$

Now we consider a new measure, where we artificially restrict all segments in $\sL$ to be open. Then the new measure will be a product measure of FK-Ising measure $\phi_{B(i,j)}^{\w,0}(\cdot)$ on $\lamn,$ for each $B(i,j)$, denoted as $\Psi$.

So every cluster $\mcc C$ with $|\mcc C|\geq 2^k$ will either be contained in some $B(i,j)$ or intersect with at least one line segment $\mss L_M$. 
Thus, we can define $\mss C_{(i,j)}$ to be all the clusters in $B(i,j)$ which are not connected to $\pare B(i,j),$ and define $\mcc C^*_{(i,j)}$ to be the cluster connected to the boundary in $B(i,j)$. Furthermore, we define $$\mcc S_{(i,j)}(k,M)=\sum_{\mcc C\in\mss C_{(i,j)}, |\mcc C|\geq 2^k}|\mcc C|+|\mcc C^*_{(i,j)}|.$$

By monotonicity, we know $\sum_{\mcc C\geq 2^k}|\mcc C|$ is stochastically dominated by 
$$\mss S(k,M):=\sum_{B(i,j)\in\mss B_M}\mcc S_{(i,j)}(k,M).$$ 
Now define the event $A_{(i,j)}$ to be the event that the largest cluster in $B(i,j)$ is less than $2^k,$ by Theorem~\ref{thm:LDP for maiximal cluster} we have 
$$\Psi(A_{(i,j)})\geq 1-c_1\exp\left(-\frac{2^{16k}}{c_1M^{30}}\right).$$
Taking a union bound then gives 
$$\Psi(A^*)\geq 1-c_1(2L+1)^2\exp\left(-\frac{2^{16k}}{c_1M^{30}}\right),$$
where $A^*$ is the intersection of all $A_{(i,j)}.$ Thus, we may choose $M(k)=\lfloor2^{\frac{8k}{15}-\frac{k_0-k}{30}}x^{-\frac{4}{15}}\rfloor$ such that $\frac{2^{16k}}{M^{30}}\ge x^82^{k_0-k}$. Note here we restrict $N^{\frac{7}{4}}x/32\le 2^k< N^{\frac{15}{8}}\sqrt{x}$ and thus $M>1$ is well defined. Recall the definition of $k_0$, we have $2L+1=\frac{N}{M}\le C_12^{\frac{17(k_0-k)}{30}}$, thus we obtain that \begin{equation*}
    \Psi(A^*)\geq 1-C_2\exp\left(-C_2^{-1}x^82^{k_0-k}\right).
\end{equation*}

Now it suffices to consider events in $A^*$. Given $A_{(i,j)}$ happens, we have that $\mcc S_{(i,j)}(k,M)$ equals to $|\mcc C_{(i,j)}^*|$. Furthermore, $A_{(i,j)}$ is a decreasing event, so for any real number $S$, we have by FKG $$\Psi(\{\sT(k,M)>S\}\cap A^*)\leq \Psi(\{\sum |\mcc C^*_{(i,j)}|>S\}\cap A^*).$$ Therefore, it suffices to prove that \begin{equation*}
    \Psi(\{\sum |\mcc C^*_{(i,j)}|\geq C_0N^{\frac{15}{4}}x2^{-k}2^{\frac{k-k_0}{2}}\}\cap A^*)\leq C\exp\left(-C^{-1}x^8\right).
\end{equation*} 
We say a box $B(i,j)$ to be nice if $|\mcc C^*_{(i,j)}|\ge V=N^{\frac{7}{4}}x2^{-k}2^{\frac{k-k_0}{2}}M^2$, denoted as $\cN(i,j)$. then applying Theorem \ref{thm:LDP for maiximal cluster} shows \begin{equation}\label{eq:LDP nice probability}
    \phi_{B(i,j)}^{\w,0}(\cN(i,j))\le c_1\exp(-c_1^{-1}[VM^{-\frac{15}{8}}]^{16})=c_1\exp(-c_1^{-1}V^{16}M^{-30}).
\end{equation}
Recall the definition of $A^*$, we get that  $|\mcc C^*_{(i,j)}|\le 2^k$ for any configuration $\omega\in A^*$. Thus, we have \begin{equation}\label{eq:LDP nice boxes number}
    \Psi\left(\left\{\sum |\mcc C^*_{(i,j)}|\geq C_0\frac{VN^2}{M^2}\right\}\cap A^*\right)\le \Phi\left(\text{at least }C_0\frac{VN^2}{2^kM^2}~ \cN(i,j)\text{ happens}\right).
\end{equation}
Combining \eqref{eq:LDP nice probability} and \eqref{eq:LDP nice boxes number} shows \begin{align}
        &\Psi\left(\left\{\sum |\mcc C^*_{(i,j)}|\geq C_0\frac{VN^2}{M^2}\right\}\cap A^*\right)\nonumber\\
        \le~&\Big[c_1\exp\left(-c_1^{-1}V^{16}M^{-30}\right)\Big]^{C_0\frac{VN^2}{2^kM^2}}\times\Big(\frac{N^2}{M^2}\Big)^{C_0\frac{VN^2}{2^kM^2}}.\label{eq:LDP Bernoulli bound}
\end{align}
Since $M\ge2^{\frac{8k}{15}-\frac{k_0-k}{30}}x^{-\frac{4}{15}}/2$ and $2^k\le 2^{k_0}\le N^{\frac{15}{8}}\sqrt{x}$, we calculate that $V^{16}M^{-30}\ge \frac{N^2}{16M^2}$. Thus, we obtain that there exists some constant $C_3>0$ such that 
$$\frac{\exp\left(-c_1^{-1}V^{16}M^{-30}\right)}{N^2/M^2}\ge \exp\left(-C_3V^{16}M^{-30}\right).$$

Thus, we finish the proof of \eqref{eq:LDP case 3} by computing the exponent in \eqref{eq:LDP Bernoulli bound} and combining with the fact that $2^k\le 2^{k_0}< N^{\frac{15}{8}}\sqrt{x}$.
Combining now \eqref{eq:LDP case 1}, \eqref{eq:LDP case 2} and \eqref{eq:LDP case 3} gives the desired result. 
\end{proof}

To finish this subsection, we will give a lower bound for the LDP to show that the exponent $16$ and $8$ in Theorem~\ref{thm:LDP for maiximal cluster} and Theorem~\ref{thm:LDP for sum of cluster squares} are optimal.

\begin{thm}\label{thm:LDP lower bound for maiximal cluster}
There exists a constant $c>0$ such that    
$$\phi_{\lamn}^{\w,0}\left(\max_{\mcc C\in\fC}|\cC|\ge N^{\frac{15}{8}}x\right)\ge c\exp\left(-c^{-1}x^{16}\right)$$
holds for all positive integers $N$.
\end{thm}
The proof of Theorem~\ref{thm:LDP lower bound for maiximal cluster} is the same as the proof of \cite[Theorem 1.4]{Kiss14}.
\section{Some analytic results}\label{sec: analytic results}
In this section, we present some analytic results on concentration and anticoncentration of random variables.

\begin{lem}\label{lem: dis concentration using variance}
    Let $X$ be a random variable with  $\var(X)\ge c_1>0$ and $\E X^4\le c_2$, then there exists a constant $c_3>0$ depending only on $c_1$ and $c_2$ such that there exist $\Omega_1,\Omega_2\subset\Omega$ with $\P(\Omega_1),\P(\Omega_2)\ge c_3$ such that \begin{align}
        |X(\omega_1)-X(\omega_2)|\ge \frac{\sqrt{c_1}}{2},~~~\forall\omega_1\in \Omega_1,\omega_2\in \Omega_2.\label{eq: var + moment bound 1}
    \end{align}
\end{lem}
\begin{proof}
    Consider $X_K=X\1_{|X|\le K}$ for any $K>0$. Since  $\E X^4\le c_2$, 
    we compute that for $K>\max\{2c_1,2\}$ \begin{align*}
      \var(X_K)&\ge \var(X)-\E~(X^1-X^2)^2\1_{\{|X^1|\ge K\text{ or }|X^2|\ge K\}}\\  &
      =\var(X)-2\E(X^1-X^2)^2\1_{\{|X^1|\geq K, |X^2|<K\}}-\E(X^1-X^2)^2\1_{\{|X^1|\geq K, |X^2|\geq K\}}\\
      &\geq \var(X)-2\E(|X^1|+K)^2\1_{\{|X^1|\geq K\}}-2\E[(X^1)^2+(X^2)^2]\1_{\{|X^1|\geq K, |X^2|\geq K\}}\\
      &\geq \var(X)-2\E(2|X^1|)^2\1_{\{|X^1|\geq K\}}-4\E[(X^1)^2]\1_{\{|X^1|\geq K\}}\\&
      \geq \var(X)-12c_1/K^2,
    \end{align*}where $X^1,X^2$ are sampled independently with the same law as $X$. Then there exists a constant $K_0=K_0(c_1,c_2)$ such that for any $K>K_0$, $\var(X_K)\ge \frac{c_1}{2}$. 

   Now choose $\Omega_1$ such that 
    $$\P(\Omega_1)=\frac{c_1}{64K^2}\text{ and }\forall \omega\in\Omega_1, \omega'\not\in\Omega_1, X_K(\omega)\geq X_K(\omega');$$
    choose $\Omega_2$ such that 
    $$\P(\Omega_2)=\frac{c_1}{64K^2}\text{ and }\forall \omega\in\Omega_2, \omega'\not\in\Omega_2, X_K(\omega)\leq X_K(\omega').$$
Therefore we have 
\begin{align*}&\var(X_K)=\E(X_K(\omega_1)-X_K(\omega_2))^2 \\
=&\E(X_K(\omega_1)-X_K(\omega_2))^2\1_{\{\omega_1\in\Omega_1\cup\Omega_2\text{ or }\omega_2\in\Omega_1\cup\Omega_2\}}+\E(X_K(\omega_1)-X_K(\omega_2))^2\1_{\{\omega_1,\omega_2\in\Omega_1^c\cap\Omega_2^c\}}\\
\leq & 4K^2\times \frac{4c_1}{64K^2}+(\min_{\omega_1\in\Omega_1}X_K(\omega_1)-\max_{\omega_2\in \Omega_2}X_K(\omega_2))^2.
\end{align*}
Combined with the fact that $\var(X_K)\geq \frac{c_1}{2}$ ,it yields that $\big|\min_{\omega_1\in\Omega_1}X_K(\omega_1)-\max_{\omega_2\in \Omega_2}X_K(\omega_2)\big|\\ \ge \sqrt{\frac{c_1}{4}}.$Thus, we have
    for any $\omega_1\in \Omega_1,\omega_2\in \Omega_2$ that
    \begin{equation}\label{eq: var + moment bound 2}
        |X_K(\omega_1)-X_K(\omega_2)|\ge \sqrt{\frac{c_1}{4}}.
    \end{equation} 
    Let $K$ large enough such that $c_2K^{-4}\le \frac{c_1}{64K^2}$ and $\Omega_1'=\Omega_1\cap\{X\le K\},\Omega_2'=\Omega_2\cap\{X\le K\}$, then we have 
    $$\min\left\{\P(\Omega_1'),\P(\Omega_2')\right\}\ge \min\left\{\P(\Omega_1),\P(\Omega_2)\right\}-c_2K^{-4}\ge \frac{c_1}{64K^2}.$$
    Since $X=X_K$ on $\Omega_1'\cup\Omega_2'$, \eqref{eq: var + moment bound 1} follows directly from \eqref{eq: var + moment bound 2}.
\end{proof}
\begin{rmk}
    It is easy to see that with $\E~e^{|X|}\le c_2$ in place of $\E X^4\le c_2$, the same result still holds.
\end{rmk}
\begin{cor}\label{cor: dis concentration using variance}
    Let $X$ be a random variable with  $\var(X)\ge c_1>0$ and $\E~e^{|X|}\le c_2$,, there exists a constant $c_3>0$ depending only on $c_1,c_2$ such that there exist $\Omega'\subset\Omega$ with $\P(\Omega')\ge c_3$ and for any $\omega\in \Omega'$, we have 
    \begin{align}
        \big|X(\omega)-\ln\left(\E~e^X\right)\big|\ge \frac{\sqrt{c_1}}{4}.\nonumber
    \end{align}
\end{cor}
\begin{lem}\label{lem: inverse Jenson}
    For any constant $c_1>0$, there exists a constant $c_2>0$ only relies on $c_1$ such that for any random variable $X$ with $\E\exp(X)< \infty$ and $\var(X)\ge c_1$, we have $\ln\big(\E\exp(X)\big)\ge \E X+c_2$.
\end{lem}
\begin{proof}
    Note that $e^x\ge 1+x+\frac{x^2}{3}$. 
    Thus, we compute 
    \begin{equation}
        \E\exp(X)\ge \exp(\E X)\left(1+\E \left[(X-\E X)+\frac{(X-\E X)^2}{3}\right]\right)=\exp(\E X)\Big(1+\frac{\var(X)}{3}\Big).\nonumber
    \end{equation}The desired result thus holds since  $\var(X)\ge c_2$.
\end{proof}
\begin{lem}\label{lem: chi fourth power tail bound}
    Let $X_1,X_2,\cdots,X_n$ be i.i.d Gaussian variables with mean $0$ and variance $1$. Let $Y_i=X_i^2-1$ and $Z_i=X_i^4-6X_i^2+3$. Then there exists a constant $c>0$ such that for any $a_i\ge 0~(i=1,2,\cdots,n)$, we have 
        \begin{align}
            &\P\left(\Bigg|\sum_{i=1}^na_iY_i\Bigg|\ge x\right)\le c\exp\left(-\frac{x^{\frac{1}{2}}}{c(\sum_{i=1}^na_i^2)^{\frac{1}{4}}}\right)\label{eq: wiener chaos input from DHX 1}\\&\P\left(\Bigg|\sum_{i=1}^na_iZ_i\Bigg|\ge x\right)\le c\exp\left(-\frac{x^{\frac{1}{4}}}{c(\sum_{i=1}^na_i^2)^{\frac{1}{8}}}\right).\label{eq: wiener chaos input from DHX 2}
        \end{align}
    In particular, if $x\ge 3\sum_{i=1}^na_i+\sqrt{\sum_{i=1}^na_i^2}$, we have \begin{equation}\label{eq: chi fourth power tail bound}
        \P\left(\Bigg|\sum_{i=1}^na_iX_i^4\Bigg|\ge x\right)\le c\exp\left(-\frac{\left(x-3\sum_{i=1}^na_i\right)^{\frac{1}{4}}}{c(\sum_{i=1}^na_i^2)^{\frac{1}{8}}}\right).
    \end{equation}
\end{lem}
\begin{proof}
    The proofs of \eqref{eq: wiener chaos input from DHX 1} and \eqref{eq: wiener chaos input from DHX 2} are similar to that of \cite[Lemma B.1]{DHX23} (see \cite[Lemma 1]{LM00} for an alternative proof to the first inequality), so we only give a sketch for \eqref{eq: wiener chaos input from DHX 2} here. We start from controlling the moments of $\Big|\sum_{i=1}^na_iZ_i\Big|$. By hypercontractivity for Guassian variables \cite{Nel73} (see also \cite[Theorem 1.4.1]{Nua06}), we get that \begin{equation}\label{eq: hypercontractivity for gaussian polynomials} \E\Big|\sum_{i=1}^na_iZ_i\Big|^p \le (p-1)^{2p} \Big(\E(\sum_{i=1}^na_iZ_i)^2\Big)^{\frac{p}{2}}=(p-1)^{2p} \Big(\sum_{i=1}^na_i^2\Big)^{\frac{p}{2}}.
\end{equation} By \cite[Equation (B.8)]{DHX23}, we can control the exponential moment of $\Big|\sum_{i=1}^na_iZ_i\Big|^{\frac{1}{4}}$ \begin{equation*}\label{eq: exponential moment for general wiener chaos}\begin{aligned}
    \exp(t\Big|\sum_{i=1}^na_iZ_i\Big|^{\frac{1}{4}}) = \sum_{p=0}^{\infty}\E\Big|\sum_{i=1}^na_iZ_i\Big|^{\frac{p}{4}}\cdot  t^p/p!&\le  \sum_{p=0}^{\infty} (\sum_{i=1}^na_i^2)^{\frac{p}{8}} ( t)^p\\&=\frac{1}{1-(\sum_{i=1}^na_i^2)^{\frac{1}{8}}\cdot t}
\end{aligned}
\end{equation*} where $t\in(0,\frac{1}{(\sum_{i=1}^na_i^2)^{\frac{1}{8}}})$. Choosing $t=\frac{1}{2(\sum_{i=1}^na_i^2)^{\frac{1}{8}}}$ and applying an exponential version of the Markov inequality completes the proof of \eqref{eq: wiener chaos input from DHX 2}. The proof of \eqref{eq: chi fourth power tail bound} comes from the fact that $X_i^4=3+6Y_i+Z_i$ and combining with \eqref{eq: wiener chaos input from DHX 1} and \eqref{eq: wiener chaos input from DHX 2}.
\end{proof}
\begin{lem}[Bernstein’s inequality]\label{lem: concentration inequality with finite exponential moment}
    Let $X_1,X_2,\cdots,X_n$ be independent random variables with a uniform exponential moment bound, i.e., $\E \exp(|X_i|)\le C_1$ for $1\leq i\leq n$. Then for any $x>0$, there exists a constant $c>0$ relies on $C_1$ but not $x$ such that
    $$\P\left(\Bigg|\sum_{i=1}^nX_i-\E X_i\Bigg|\ge x\right)\le \exp(-\min\{\frac{x^2}{nc},\frac{x}{c}\}).$$ 
\end{lem}
\begin{proof}
    See for example \cite[Theorem 2.8.1]{vershynin2018high}.
\end{proof}

\section{Some results on FK-Ising model without disorder}\label{sec: FK Ising not critical}
In this section, we prove some large deviation results for the FK-Ising model at low-temperature and high-temperature.
\begin{lem}\label{lem: low temperature cluster bound}
    Fix $p>p_c$. Let $\mathfrak{C}(\omega)$ denote the collection of clusters under the configuration $\omega$ and $\cC_{\diamond}$ denote the maximal cluster. Then for any constant $\alpha>0$, there exists a constant $c_1=c_1(\alpha,p)>0$ such that 
    $$\phi^{\w,0}_{p,\Omega}\left(\max_{\cC\in\mathfrak{C}\setminus\{\cC_{\diamond}\}}|\cC|\ge N^{\alpha}\right)\le c_1^{-1}\exp\left(-c_1N^{c_1}\right)$$
    holds for any positive integer $N$ and domain $\lamn\subset\Omega\subset\Lambda_{2N}$.
    Furthermore, there exists a constant $c_2=c_2(p)>0$ such that 
    $$\phi^{\w,0}_{p,\Omega}\left(\{\pari\Omega\subset\cC_{\diamond}\}\right)\ge 1- c_2^{-1}\exp\left(-c_2\sqrt{N}\right)$$
    holds for any positive integer $N$ and domain $\lamn\subset\Omega\subset\Lambda_{2N}$. 
\end{lem}
\begin{proof}
    
    Let $M=\min\left\{\frac{N}{6},\frac{N^{\frac{\alpha}{2}}}{16}\right\}$. Let $\{B_i=\Lambda_{M}(u_i) \}_{i=1}^k$ be a disjoint $2M$-box covering of $\Omega$. Similar to the proof of Lemma~\ref{lem: good external field for number of xi}, let $\mcc G(u)$ denote the event that there exists an open circuit in the annuli $\Lambda_{2M}(u)\setminus\Lambda_{M}(u)$ and it is connected to $\pari\Lambda_M(u)$ in $\omega.$ Then by the RSW theory,  
    there exists a constant $C_1>0$ that only relies on $p$ such that for any boundary condition $\gamma$ on $\pari\Lambda_{2M}(u_i),$ 
    $$\phi^{\gamma,0}_{p,\Lambda_{2M}(u_i)}\left(\mcc G(u_i)\right)\geq 1-\exp(-C_1M).$$
    Therefore, let $\mcc G=\cap_{i=1}^k\mcc G(u_i)$, then we have by CBC that \begin{align}
        \phi^{\w,0}_{p,\Omega}(\mcc G)\geq 1-\sum_{i=1}^k\phi^{\w,0}_{p,\Omega}(\mcc G(u_i)^c)\stackrel{(*)}{\ge} 1-\sum_{i=1}^k\phi^{\f,0}_{p,\Lambda_{2M}(u_i)}(\mcc G(u_i)^c)&\ge 1-\frac{N^2}{M^2}\exp(-C_1M)\nonumber\\&\ge 1-C_2^{-1}\exp(-C_2M),\label{eq: coarse graining using circuit}
    \end{align}where the last inequality holds since $M=\min\{\frac{N}{6},\frac{N^{\frac{\alpha}{2}}}{16}\}$. Note that in $(*)$, we do not restrict $\Lambda_{2M}(u_i)\subset\Omega$.

    For $\omega\in\mcc G$, the circuits in the annulus $\Lambda_{2M}(u_i)\setminus\Lambda_{M }(u_i)$ are connected, and we denote the cluster containing those circuits as $\cC_0$.
    In addition, we get that any cluster with diameter larger than $4M$ will be connected to one of the circuits in the annuli $\Lambda_{2M}(u)\setminus\Lambda_{M}(u)$, thus $\cC_{0}=\cC_\diamond$ is the maximal cluster in $\omega$ and $\max_{\cC\in\mathfrak{C}\setminus\{\cC_{\diamond}\}}|\cC|< 16M^2$. In conclusion, we get from \eqref{eq: coarse graining using circuit} that 
    $$\phi^{\w,0}_{p,\Omega}\left(\max_{\cC\in\mathfrak{C}\setminus\{\cC_{\diamond}\}}|\cC|\ge N^{\alpha}\right)\le\phi^{\w,0}_{p,\Omega}(\mcc G^c)\le\exp(-C_2M).$$
    Furthermore, let $\alpha=1$, then for $\omega\in\mcc G$, since $\cC_{0}$ is the maximal cluster in $\omega$ and $\pari\Omega\subset\cC_0$, we get that $\pari\Omega\subset\cC_{\diamond}$. Combined with \eqref{eq: coarse graining using circuit}, it yields that \begin{equation*}
        \phi^{\w,0}_{p,\Omega}\left(\{\pari\Omega\subset\cC_{\diamond}\}\right)\ge 1- \exp(-C_3\sqrt{N}).\qedhere
    \end{equation*}
\end{proof}
\begin{lem}\label{lem: low temperature maximal cluster deviation bound}
    Fix $p>p_c$. Then there exists a constant $c=c(p)>0$ such that for any positive integer $N$ and domain $\lamn\subset\Omega\subset\Lambda_{2N}$, we have $$\langle|\cC_*|\rangle-\langle|\cC_*\cap\tilde\cC_*|\rangle\ge cN^2$$ where $\cC_*$ and $\tilde\cC_*$ are the boundary clusters of two independently sampled configurations under measure $\fklow$.
\end{lem}

\begin{proof} 
    We can rewrite the left-hand side as $\langle M\rangle$ where
    \begin{align*} 
    M&=\sum_{x\in\Omega}\1_{\{x\stackrel{\omega}{\longleftrightarrow}\pari\Omega\}}-\sum_{x}\1_{\{x\stackrel{\omega}{\longleftrightarrow}\pari\Omega\}}\1_{\{x\stackrel{\tilde{\omega}}{\longleftrightarrow}\pari\Omega\}}
    =\sum_{x}\1_{\{x\stackrel{\omega}{\longleftrightarrow}\pari\Omega\}}\1_{\{x\stackrel{\tilde{\omega}}{\centernot\longleftrightarrow}\pari\Omega\}}.
    \end{align*}
    Taking expectation w.r.t. $\fklow$ we have 
    $$\langle M\rangle=\sum_{x}\fklow(x\in \mcc C_*)\fklow(x\not\in \tilde{\mcc C}_{*}).$$
    Noticing the fact that both probabilities have positive lower bounds only relying on $p$ completes the proof. More precisely, a point has a positive probability of connecting to the boundary since $p>p_c$, and a positive probability of not connecting to the boundary due to finite energy property.
\end{proof}
\begin{lem}\label{lem: MDP for boundary cluster at low temperature}
    Fix $p>p_c$. Then there exists a constant $c=c(p)>0$ such that for any positive integer $N$ and domain $\lamn\subset\Omega\subset\Lambda_{2N}$, we have $$\phi_{p,\Omega}^{\w,0}\Big(\big||\cC_{\diamond}|-\langle|\cC_{\diamond}|\rangle_{p,\Omega}^{\w,0}\big|\ge N^{1.5}\Big)\le N^{-0.8}.$$
\end{lem}
\begin{proof}

Let $\mathcal{E}$ denote the event that $\pari\Omega\subset\cC_{\diamond}$ and let $\cC_{*}$ denote the boundary cluster. Then we calculate by Lemma~\ref{lem: low temperature cluster bound} that \begin{equation}\nonumber
\Big|\langle|\cC_{\diamond}|-|\cC_{*}|\rangle_{p,\Omega}^{\w,0}\Big|\le \P(\cE^c)\cdot N^2\le\exp(-C_1\sqrt{N})\cdot N^2.
\end{equation} Combined with Lemma~\ref{lem: low temperature cluster bound} again, it suffices to show that \begin{equation}
    \phi_{p,\Omega}^{\w,0}\Big(\big||\cC_{*}|-\langle|\cC_{*}|\rangle_{p,\Omega}^{\w,0}\big|\ge N^{1.5}/2\Big)\le N^{-0.8}/2.\nonumber
\end{equation}Next, we calculate the variance of $|\cC_{*}|$ under the measure $  \phi_{p,\Omega}^{\w,0}$. \begin{align}
    \var_{p,\Omega}^{\w,0}(|\cC_{*}|)&=\sum_{x,y\in\Omega}\phi_{p,\Omega}^{\w,0}({x,y\in\cC_{*}})-\sum_{x,y\in\Omega}\phi_{p,\Omega}^{\w,0}({x\in\cC_{*}})\cdot\phi_{p,\Omega}^{\w,0}({y\in\cC_{*}}).\label{eq: variance calculation}
\end{align}For any $x,y\in\Omega,$
let $d_{x,y}=\min\{dist(x,y),dist(x,\pari \Omega),dist(y,\pari \Omega)\}$ and $M=\max\{d_{x,y}/2,\\N^{0.1}\}$.
Let $\cE_x$  and $\cE_y$  denote the event that there is a circuit in $\Lambda_{2M}(x)\setminus\Lambda_{M}(x)$ and $\Lambda_{2M}(y)\setminus\Lambda_{M}(y)$ which is connected to $\pari\Omega$ respectively. By the proof of Lemma~\ref{lem: low temperature cluster bound}, we get that \begin{equation}\label{eq: connect to a big box boundary and boundary}
    \phi_{p,\Omega}^{\w,0}(\cE_x\cap\cE_y)\ge 1-\exp(-C_2M).
\end{equation} Under $\cE_x\cap\cE_y$, the events $x\in\cC_{*}$ and $y\in\cC_{*}$ are independent, thus we calculate that \begin{align}
    &\phi_{p,\Omega}^{\w,0}({x,y\in\cC_{*}})-\phi_{p,\Omega}^{\w,0}({x\in\cC_{*}})\cdot\phi_{p,\Omega}^{\w,0}({y\in\cC_{*}})\nonumber\\
\le ~&\phi_{p,\Omega}^{\w,0}(\cE_x^c\cup\cE_y^c)+\phi_{p,\Omega}^{\w,0}(\cE_x\cap\cE_y)\cdot\phi_{p,\Omega}^{\w,0}({x,y\in\cC_{*}}\mid\cE_x\cap\cE_y)\nonumber\\&-[\phi_{p,\Omega}^{\w,0}({x\in\cC_{*}}\mid\cE_x)\cdot\left(1-\phi_{p,\Omega}^{\w,0}(\cE_x^c)\right)]\cdot [\phi_{p,\Omega}^{\w,0}({y\in\cC_{*}}\mid\cE_y)\cdot\left(1-\phi_{p,\Omega}^{\w,0}(\cE_y^c)\right)]\nonumber\\
    \le ~&\phi_{p,\Omega}^{\w,0}(\cE_x^c\cup\cE_y^c)+\phi_{p,\Omega}^{\w,0}({x,y\in\cC_{*}}\mid\cE_x\cap\cE_y)-\phi_{p,\Omega}^{\w,0}({x\in\cC_{*}}\mid\cE_x)\cdot \phi_{p,\Omega}^{\w,0}({y\in\cC_{*}}\mid\cE_y)\nonumber\\&+\phi_{p,\Omega}^{\w,0}(\cE_x^c)
    +\phi_{p,\Omega}^{\w,0}(\cE_y^c)\nonumber\\
    \le ~&2\phi_{p,\Omega}^{\w,0}(\cE_x^c)+2\phi_{p,\Omega}^{\w,0}(\cE_y^c).\nonumber
\end{align}Combined with \eqref{eq: connect to a big box boundary and boundary}, it yields that \begin{equation}\label{eq: truncated correlation function in a region}
    \phi_{p,\Omega}^{\w,0}({x,y\in\cC_{*}})-\phi_{p,\Omega}^{\w,0}({x\in\cC_{*}})\cdot\phi_{p,\Omega}^{\w,0}({y\in\cC_{*}})\le 4\exp(-C_2M).
\end{equation}
Plugging \eqref{eq: truncated correlation function in a region} into \eqref{eq: variance calculation}, we get that \begin{equation}
    \var_{p,\Omega}^{\w,0}(|\cC_{*}|)\le C_4N^{2.2}.
\end{equation}
The desired bound thus comes from a Markov inequality.
\end{proof}
\begin{lem}\label{lem: LDP for high temp}
    Fix $p<p_c$. Then there exists a constant $c_1=c_1(p)>0$ such that for any domain $\lamn\subset\Omega\subset\Lambda_{2N}$ and $x>c_1$, we have $\fkhof(\sum_{\cC\in\mathfrak{C}}|\cC|^2\ge N^2x)\le c_1\exp(-c_1^{-1}N\sqrt{x})$ and $\fkhof(\sum_{\cC\in\mathfrak{C}}|\cC|^4\ge N^2x)\le c_1\exp(-c_1^{-1}\sqrt{N}x^{1/4})$. Furthermore, for any $y\ge N^{0.01}$, we have $\fkhof(\max_{\cC\in\mathfrak{C}}|\cC|\ge y)\le c_2\exp(-c_2\sqrt{y})$.
\end{lem}
\begin{proof}
    We first compute the moments of $\sum_{\cC\in\mathfrak{C}}|\cC|^2$. Since $\sum_{\cC\in\mathfrak{C}}|\cC|^2=\sum_{u,v\in\Omega}\1_{\{u\longleftrightarrow v\}}$, we compute that for any positive integer $k$ 
    \begin{equation}
\Big\langle\Big(\sum_{\cC\in\mathfrak{C}}|\cC|^2\Big)^k\Big\rangle=\Big\langle\sum_{u_i,v_i\in\Omega}\prod_{i=1}^k\1_{\{u_i\longleftrightarrow v_i\}}\Big\rangle=\sum_{u_i,v_i\in\Omega}\fklow\big({\cap_{i=1}^k\{u_i\longleftrightarrow v_i\}}\big).
    \end{equation}
    
Recalling the definition of $\mathtt{Even}_{\{u_1,\cdots,u_k,v_1,\cdots,v_k\}}$ from the proof of Lemma~\ref{lem: small perturbation for partition function}, we get that $\cap_{i=1}^k\{u_i\longleftrightarrow v_i\}\subset\mathtt{Even}_{\{u_1,\cdots,u_k,v_1,\cdots,v_k\}}$. By a similar reason to \eqref{eq: second moment for Ising connecting probability} and \eqref{eq: second moment for Ising connecting probability 2}, we get that
\begin{equation}
    \sum_{u_i,v_i\in\Omega}\fklow\big({\cap_{i=1}^k\{u_i\longleftrightarrow v_i\}}\big)\le C_1^{2k}\prod_{j=1}^k (N+j)^{2}\le C_1^{2k}(N+k)^{2k}.\label{eq: high temperature moment k point function bound}
\end{equation}
With the moment bounds, we are able to compute an exponential moment for $\sqrt{\sum_{\cC\in\mathfrak{C}}|\cC|^2}$. For any $t>0$, we have \begin{align}
\Bigg\langle\exp\left(t\sqrt{\sum_{\cC\in\mathfrak{C}}|\cC|^2}\right)\Bigg\rangle\le\sum_{k=0}\frac{t^{2k}\Big\langle\Big(\sum_{\cC\in\mathfrak{C}}|\cC|^2\Big)^k\Big\rangle}{(2k)!}+\sum_{k=0}\frac{t^{2k+1}\Big\langle\Big(\sum_{\cC\in\mathfrak{C}}|\cC|^2\Big)^{\frac{2k+1}{2}}\Big\rangle}{(2k+1)!}.\label{eq: exponential moment bound 0}
\end{align} By \eqref{eq: high temperature moment k point function bound}, we get that \begin{equation}
\frac{t^{2k}\Big\langle\Big(\sum_{\cC\in\mathfrak{C}}|\cC|^2\Big)^k\Big\rangle}{(2k)!}\le \frac{(2C_1)^{2k}(N^{2k}+k^{2k})}{(2k)!}\le \frac{(2C_1)^{2k}N^{2k}}{(2k)!}+C_2^{2k}.\label{eq: exponential moment bound 1}
\end{equation}By Cauchy inequality and \eqref{eq: high temperature moment k point function bound} again, we get that \begin{align}
&\frac{\Big\langle\Big(\sum_{\cC\in\mathfrak{C}}|\cC|^2\Big)^{\frac{2k+1}{2}}\Big\rangle}{(2k+1)!}\le \frac{\sqrt{\Big\langle\Big(\sum_{\cC\in\mathfrak{C}}|\cC|^2\Big)^{k}\Big\rangle\cdot\Big\langle\Big(\sum_{\cC\in\mathfrak{C}}|\cC|^2\Big)^{k+1}\Big\rangle}}{(2k+1)!}\nonumber\\\le~&\frac{\sqrt{C_1^{2k}(N+k)^{2k}C_1^{2k+2}(N+k+1)^{2k+2}}}{(2k+1)!}\le\frac{C_1^{2k+1}(N+k+1)^{2k+1}}{(2k+1)!}\nonumber\\\le~& \frac{(2C_1)^{2k+1}(N^{2k+1}+(k+1)^{2k+1})}{(2k+1)!}\le \frac{(2C_1)^{2k+1}N^{2k+1}}{(2k+1)!}+C_2^{2k+1}.\label{eq: exponential moment bound 2}
\end{align}
Plugging \eqref{eq: exponential moment bound 1} and \eqref{eq: exponential moment bound 2} into \eqref{eq: exponential moment bound 0} and choosing $t=\frac{1}{2C_2}$, we get that\begin{align}
\Bigg\langle\exp\left(t\sqrt{\sum_{\cC\in\mathfrak{C}}|\cC|^2}\right)\Bigg\rangle&\le\sum_{k=0}\frac{(2C_1Nt)^{2k}}{(2k)!}+\frac{(2C_1Nt)^{2k+1}}{(2k+1)!}+(C_2t)^{2k}+(C_2t)^{2k+1}\nonumber\\&=\exp\left(\frac{C_1}{C_2}N\right)+2.
\end{align}
The desired result thus comes from the Markov inequality and letting $c_1$ large enough.\color{black} For $\sum_{\cC\in\mathfrak{C}}|\cC|^4$, the result comes in a similar way. For the last criterion, we have by \cite{DT20} that \begin{align*}
    \fkhof\left(\max_{\cC\in\mathfrak{C}}|\cC|\ge y\right)\le\sum_{\dist(x,y)\ge \sqrt y}\fkhof(x\longleftrightarrow y)\le N^2y\exp(-C_3\sqrt{y}).
\end{align*}
The desired result follows from the fact that $y\ge N^{0.01}$.
\end{proof}
\begin{cor}\label{cor: sum of squares of clusters all temperature}
For any $T>0,$ let $\mss S=\{\omega:\sum_{\cC\in\fC{(\omega)}}|\cC|^2\le \eps^{-1}N^{3\alpha(T)}\}.$ For any $p>0$, there exists a constant $c>0$ such that
    \begin{equation}\label{eq:sum of squares of clusters all temperature}
    \fk\gamma(\sS)\ge 1-c\exp\left(-c^{-1}\eps^{-1}N^{-\alpha(T)}\right).
\end{equation}
\end{cor}
\begin{proof}
    We divide the proof of this lemma into three parts according to the temperature. When $p=p_c$, then the desired result comes from Theorem~\ref{thm:LDP for sum of cluster squares}.

    When $p>p_c$, then $\alpha(T)=1$. If $\eps N\ge 1$, then we have $\eps^{-1}N^{-\alpha(T)}\le 1$. thus we can choose $c$ large enough such that $c\exp(-c)>1$ and then the desired result holds. If $\eps N\ge 1$, then we have $\eps^{-1}N^{3\alpha(T)}\ge N^4$ and thus $\fk\gamma(\sS)\ge 1$.

    When $p<p_c$, then the desired result comes from Lemma~\ref{lem: LDP for high temp}.
\end{proof}
\begin{lem}\label{lem: high temperature all cluster deviation bound}
    Fix $p<p_c$. Then there exists a constant $c=c(p)>0$ such that for any positive integer $N$ and domain $\lamn\subset\Omega\subset\Lambda_{2N}$, we have $$\Big\langle\sum_{\cC\in\mathfrak{C}}|\cC|^2\Big\rangle-\Big\langle\sum_{\cC\in\mathfrak{C},\tilde{\cC}\in\tilde{\mathfrak{C}}}|\cC\cap\tilde\cC|^2\Big\rangle\ge cN^2$$
    where $\mathfrak{C}$ and $\tilde{\mathfrak{C}}$ are the collections of all clusters of two independently sampled configurations under measure $\fkhof$.
\end{lem}

\begin{proof}
    Using the same idea as the proof of Lemma~\ref{lem: low temperature maximal cluster deviation bound}, we have 
    $$\Big\langle\sum_{\cC\in\mathfrak{C}}|\cC|^2\Big\rangle-\Big\langle\sum_{\cC\in\mathfrak{C},\tilde{\cC}\in\tilde{\mathfrak{C}}}|\cC\cap\tilde\cC|^2\Big\rangle=\sum_{x\in\Omega}\sum_{y\in\Omega}\fkhof\left(x\stackrel{\omega}{\longleftrightarrow}y\right)\fkhof\left(x\stackrel{\tilde\omega}{\centernot\longleftrightarrow}y\right).$$
    Only considering the pairs of neighboring $x,y$ and applying  finite energy property, we get  $cN^2$ as a lower bound.
\end{proof}

\section*{Acknowledgements}
We thank Jian Ding for several insightful discussions.
We thank Xinyi Li and Jo\~ao Maia for reading an early version of this paper and for providing many valuable suggestions. We thank Federico Camia and Jianping Jiang for reminding us the continuous version of Theorem \ref{thm:LDP for maiximal cluster}.
F. Huang is supported by the Beijing Natural Science Foundation Undergraduate Initiating Research Program (Grant No.\ QY23006). 
All authors are partially supported by the National Natural Science Foundation of China (NSFC) through the Tianyuan Mathematics Key Program (Grant No.\ 12226001) and the NSFC Key Program (Grant No.\ 12231002).
\small
\bibliography{myref}
\bibliographystyle{abbrv}
\end{document}